\documentclass{article}

\usepackage[a4paper, total={6in, 8in}]{geometry}

\RequirePackage[T1]{fontenc}
\RequirePackage{amsthm,amsmath}
\allowdisplaybreaks
\RequirePackage[numbers]{natbib}
\RequirePackage[colorlinks,citecolor=blue,urlcolor=blue]{hyperref}

\usepackage{subcaption}

\author{Guy Bresler\thanks{Guy Bresler is with the Department of Electrical Engineering and Computer Science at MIT and a member of LIDS and IDSS. {\tt\small guy@mit.edu} }  \ 
	and 
	Mina Karzand\thanks{Mina Karzand is with the Department of Electrical Engineering and Computer Science at MIT and a member of LIDS and IDSS. {\tt\small mkarzand@mit.edu}}\thanks{The authors are ordered alphabetically.}
}

\date{}

\newcommand{\rvw}[1]{#1}

\usepackage{chngcntr}
\counterwithin{figure}{section}

\usepackage{nicefrac}

\newcommand{\VV}{\mathcal{V}}
\renewcommand{\SS}{\mathcal{S}}
\newcommand{\TT}{\mathcal{T}}

\newcommand{\EE}{\mathcal{E}}
\newcommand{\EEstrong}{\mathcal{E}^{\text{strong}}}

\newcommand{\EdgErr}{\mathtt{E}^{\text{edge}}}

\newcommand{\Ex}{\mathbb{E}}

\newcommand{\al}{\alpha}

\DeclareMathOperator*{\argmin}{arg\,min}

\newcommand{\CC}{\mathcal{C}}  

\newcommand{\Tc}{\mathsf{T^{CL}}}  

\newcommand{\PP}{\mathcal{P}}
\newcommand{\Pt}{{\widetilde{P}}}
\newcommand{\Ph}{\widehat{P}}

\newcommand{\E}{\mathbb{E}}

\newcommand{\LL}{\mathcal{L}}

\newcommand{\T}{{\mathsf{T}}}
\newcommand{\Th}{{\widehat{\mathsf{T}}}} 
\newcommand{\Tt}{{\widetilde{\mathsf{T}}}} 

\renewcommand{\P}{\mathsf{P}}

\renewcommand{\path}{\mathsf{path}}
\newcommand{\subT}[3]{\mathsf{SubTree}_{#1,#2}(#3)}

\newcommand{\XX}{\mathcal{X}}

\newcommand{\wmu}{{\widehat{\mu}}}
\newcommand{\wtmu}{{\widetilde{\mu}}}

\newcommand{\wt}{{\widetilde{w}}}
\newcommand{\ut}{{\widetilde{u}}}
\newcommand{\vt}{{\widetilde{v}}}

\newcommand{\wttheta}{{\widetilde{\theta}}}

\newcommand{\A}{\mathcal{A}}
\newcommand{\At}{\widetilde{\mathcal{A}}}
\newcommand{\B}{\mathcal{B}}
\newcommand{\Bt}{\widetilde{\mathcal{B}}}
\newcommand{\C}{\mathcal{C}}
\newcommand{\Ct}{\widetilde{\mathcal{C}}}
\newcommand{\D}{\mathcal{D}}
\newcommand{\Dt}{\widetilde{\mathcal{D}}}

\newcommand{\F}{\mathcal{F}}

\newcommand{\err}{\mathsf{error}}

\newcommand{\eps}{\epsilon}

\newcommand{\Ecorr}{\mathtt{E}^\mathrm{corr}(\eps)}
\newcommand{\Pcorr}{\mathtt{E}^\mathrm{strong}(\eps)}
\newcommand{\Ncorr}{\mathtt{E}^\mathrm{cascade}(\eps)}
\newcommand{\Ncorrg}{\mathtt{E}^\mathrm{cascade}(\gamma)}

\newcommand{\Ecorrn}{\mathtt{E}^\mathrm{corr}(\eps_0)}
\newcommand{\Pcorrn}{\mathtt{E}^\mathrm{strong}(\eps_0)}
\newcommand{\Ncorrgn}{\mathtt{E}^\mathrm{cascade}(\gamma_0)}

\newcommand{\pr}{m}
\newcommand{\Mc}{\mathcal{M}}

\DeclareMathOperator*{\argmax}{arg\,max}
\newtheorem{defn}{Definition}
\newcommand{\sstv}{\mathsf{ssTV}}
\newtheorem{prop}{Proposition}

\newtheorem{lem}{Lemma}

\pdfoutput=1

\newtheorem{theorem}{Theorem}[section]
\newtheorem{lemma}[theorem]{Lemma}
\newtheorem{proposition}[theorem]{Proposition}
\newtheorem{corollary}[theorem]{Corollary}
\theoremstyle{definition}
\newtheorem{definition}[theorem]{Definition}

\usepackage{amsmath,amssymb,mathtools}
\usepackage{bm,verbatim}
\usepackage{algorithm}

\usepackage{dsfont}
\usepackage{xcolor}

\usepackage{xr}

\title{Learning a Tree-Structured Ising Model \\ in Order to Make Predictions}

\begin{document}
\maketitle




%
%
%






\begin{abstract}
\rvw{We study the problem of learning a tree Ising model from samples such that subsequent predictions made using the model are accurate. The prediction task
considered in this paper is that of predicting the values of a subset of variables given
values of some other subset of variables.} Virtually all previous work on graphical model learning has focused on recovering the true underlying graph. We define a distance (``small set TV" or ssTV) between distributions $P$ and $Q$ by taking the maximum, over all subsets $\SS$ of a given size, of the total variation between the marginals of $P$ and $Q$ on $\SS$; this distance captures the accuracy of the prediction task of interest. 
We derive non-asymptotic bounds on the number of samples needed to get a distribution (from the same class) with small ssTV relative to the one generating the samples. 
One of the main messages of this paper is that far fewer samples are needed than for recovering the underlying tree, which means that accurate predictions are possible using the wrong tree.
\end{abstract}

%
%

\section{Introduction}
Markov random fields, or undirected graphical models, are a useful way to represent high-dimensional probability distributions \cite{wainwright2008graphical,koller2009probabilistic}. A Markov random field is a probability distribution described by a graph: each node in the graph corresponds to a random variable, and the variables are required to satisfy the Markov property whereby a variable is conditionally independent of all other variables given its neighbors.

The practical utility of Markov random fields is in large part due to 1) edges between variables capture direct interactions, which make the model \emph{interpretable} and 2)
the graph structure facilitates efficient approximate \emph{inference} from partial observations, for example using loopy belief propagation or variational methods. A prediction for $X_i$ based on observed values $X_\SS=x_\SS$ for a subset of variables $\SS$  can be easily obtained from the conditional probability $P(X_i=x_i|X_\SS = x_\SS)$. The inference task relevant to this paper is therefore evaluation of conditional probabilities or marginals. 
 
In applications it is often necessary to learn the model from data, and it makes sense to measure accuracy of the learned model in a manner corresponding to the intended use. While in some applications it is of interest to learn the graph itself, in many machine learning problems the focus is on making predictions. In the literature, learning the graph is called \emph{structure learning}; this problem has been studied extensively in recent years, see e.g. \cite{meinshausen2006high,BMS08,vuffray2016interaction,abbeel2006learning,ravikumar2010high,meila2001learning}. In contrast, we consider in this paper the problem of learning a good model \emph{for the purpose of performing subsequent prediction from partial observations}. \rvw{For instance, one might wish to use the learned model to predict the preference of a user for an item in a recommendation system based on ratings obtained for a few items.} This objective has been called ``inferning" (inference + learning)~\cite{heinemann2014inferning}, and has received significantly less attention. This paper contains, to the best of our knowledge, the first results on estimating graphical models with a prediction-centric loss that are applicable to the high-dimensional setting. 

Structure learning becomes statistically more challenging, meaning more data is required, when interactions between variables are very weak or very strong \cite{santhanam2012information,
bento2009graphical,tandon2014information}. It is intuitively clear that very weak edges are difficult to detect, leading to non-identifiability of the model. 
The goal of this paper is to show that learning a model that makes accurate predictions is possible even when structure learning is not. 

With the goal of making predictions in mind, we introduce a loss function to evaluate learning algorithms based on the accuracy of low-order marginals. The small-set total variation between true distribution $P$ and learned distribution $Q$ is defined to be $$ \LL^{(k)}(P,Q)\triangleq \max_{\SS:|\SS|= k} d_\mathrm{TV}(P_\SS , Q_\SS)\,,$$ 
where $P_\SS$ denotes the marginal on set $\SS$.  The small-set total variation is inherently far less stringent than the total variation over the entire joint distribution and this makes a crucial difference in high-dimensional models. This same local total variation metric was used by Rebeschini and van Handel in a somewhat different context~\cite{rebeschini2015can} and has appeared earlier in Dobrushin's work on Gibbs measures~\cite{georgii2011gibbs}. As discussed in Section~\ref{s:Inferning}, small loss $\LL^{(k)}$ guarantees accurate posterior distributions conditioned on sets of size $k-1$. 

Tree-structured graphical models have been particularly well-studied. Aside from their theoretical appeal, there are two reasons for the practical utility of tree models: 1) The maximum likelihood tree can be easily computed, and the correct graph can be recovered with smaller sample and time complexity as compared to loopy graphs, and 2) Efficient exact inference (computation of marginals or maximum probability assignments) is possible using belief propagation. Sum-product or max-product algorithms are two well-studied examples \cite{lauritzen1996graphical,
pearl1988reasoning,wainwright2008graphical,
WaiWai03} of inference algorithms on trees.
Hence, we focus on tree-structured models.

In this paper, we (further) restrict attention to tree-structured Ising models with no external field, defined as follows. For tree $\T=(\VV,\EE)$ on $p$ nodes and edge parameters $\theta_{ij}$  for each edge $(i,j)\in\EE$, each configuration $x\in \{-1,+1\}^p$ is assigned probability
\begin{equation} \label{e:Ising}
P(x) =  \exp\bigg(\sum_{(i,j)\in \EE}\theta_{ij} x_ix_j - \Phi(\theta)\bigg)\,,
\end{equation}
where $\Phi(\theta)$ is the normalizing constant. We assume throughout that $\alpha\leq |\theta_{ij}|\leq \beta$ for some $\alpha, \beta\geq 0$ for each edge $(i,j)\in\EE$.  \rvw{Due to the tree structure, it turns out that the variables $Y_{ij}=X_i X_j$ for $(i,j)\in\EE$ are jointly independent (as shown in Lemma~\ref{l:IndependentError}), a fact that is useful in the analysis. As a consequence, the correlation between a pair of variables is equal to the product of the correlations $\Ex[X_i X_j]$ on the edges $(i,j)$ in the path connecting them. 

In general, we could have  an external field term $\sum_{i\in\VV}\theta_i x_i$ in the exponent of~\eqref{e:Ising}. The assumption of no external field ($\theta_i=0$) implies uniform singleton  marginals, \textit{i.e.,} $P(x_i=+1)=1/2$ for all $i$. This assumption helps to make the analysis tractable and at the same time captures the central features of the problem.}

Suppose we observe i.i.d. samples, generated from a tree-structured Ising model.  The main question we address is how many samples are required in order to learn a model with a guarantee on the accuracy of subsequent predictions computed using the learned model. Since  computation of marginals for a given tree model is easy, the crux of the task is in learning a model with marginals that are close to those of the original model. One of the take-home messages is that learning for the purpose of making predictions requires dramatically fewer samples than is necessary for correctly recovering the underlying tree. The central technical challenge is that our analysis must therefore also apply when it is impossible to learn the true tree, and this requires careful control of the sorts of errors that can occur.

Our main result gives lower and upper bounds on the number of samples needed to learn a tree Ising model to ensure small $\LL^{(2)}$ loss, which in this setting is equivalent to accurate pairwise marginals. We emphasize that the task is to learn a model from the same class (tree-structured Ising) with these guarantees; this is sometimes called \emph{proper} learning.  The main result concerns the maximum likelihood tree (also called Chow-Liu tree), defined in Section \ref{s:learningforPred}.

\begin{theorem} \label{t:mainResultIntro}
Fix $\eta>0$. 	
	Given \rvw{$n>C \max\{\eta^{-2},\,e^{2\beta}\} \log\frac{p}{\delta}$} samples generated according to a tree Ising model $P$ defined in~\eqref{e:Ising} with $|\theta_{ij}|\leq \beta$, denote the Chow-Liu tree by $\Tc$. The Ising model $Q$ on $\Tc$ obtained by matching correlations on the edges satisfies $\LL^{(2)}(P,Q)\leq \eta$ with probability at least $1-\delta$.  Conversely, if $\tanh\alpha+2\eta \leq \min\{\tanh\beta,1/2\}$ and \rvw{$n\leq C' \eta^{-2} \log p$}, then no algorithm can find a tree model $Q$ such that $\LL^{(2)}(P,Q)\leq \eta$ with probability greater than half. 
\end{theorem}
 
The result shows that the Chow-Liu tree, which can be found in time $O(p^2\log p)$, gives small $\LL^{(2)}$ error.
\rvw{We remark that the Chow-Liu algorithm uses only pairwise marginals of the
empirical distribution and can therefore be implemented with missing data as long as
the pairwise marginals can be estimated.}

\rvw{In Section~\ref{s:resultStatement}, we discuss the assumption  $\tanh\alpha+2\eta \leq \min\{\tanh\beta,1/2\}$ made in the theorem, which captures the fact that the learner does not know the magnitude of the edge parameters a priori.}

\rvw{ It turns out that for trees, accuracy of pairwise marginals translates to accuracy of higher order marginals, and a bound to this effect is proved in Appendix~\ref{sec:MarginalK}. We believe that the dependence on $k$ can be improved. 
\begin{corollary}\label{cor:margK}
Let $\T$ and $\T'$ be two (possibly distinct) trees. Let $P$ and $Q$ be probability distributions represented according to $\T$ and $\T'$ using~\eqref{e:Ising} such that $\LL^{(2)}(P,Q)<\eta$. Then for all $k$,  we have $\LL^{(k)}(P,Q)<k2^k \eta$.
\end{corollary} 

Numerical simulations in Section \ref{s:NumSim} show the dependence of the loss  on number of samples for different values of $\alpha$ and $\beta$, supporting Theorem~\ref{t:mainResultIntro}. 
There are a few important issues that are not addressed by Theorem~\ref{t:mainResultIntro}, which we also investigate via simulations in Appendix~\ref{sec:NumSimSUpp}.
These include robustness of the results to 
model misspecification, \textit{i.e.}, samples are not from a tree-structured Ising model; external field in the Ising model~\eqref{e:Ising} generating the data;  accurate marginals of size $k\geq 3$.
}

 We next place the result in the context of related work.

\subsection{Related work}
\label{sec:RelatedWork}

\rvw{Tree-structured graphical models have applications in image processing and computer vision \cite{romberg2001bayesian,freeman2000learning,wainwright2001random,portilla2003image}, artificial intelligence \cite{pearl1988reasoning}, coding theory \cite{gallager1962low} and statistical physics \cite{baxter1985exactly}.}

%
Structure learning in general graphical models has been studied extensively. Information-theoretic bounds on the number of samples have been derived \cite{santhanam2012information,
tandon2014information,bento2009graphical,
bresler2015efficiently,wu2013learning,loh2013structure}.
Structure learning of trees has been studied by \cite{chow1968approximating,tan2011large}.
Learning of generalizations of tree-structured models has been studied, including: forest approximations \cite{tan2011learning,liu2011forest},  polytrees \cite{dasgupta1999learning}, bounded treewidth graphs \cite{srebro2001maximum,PACBilmes}, loopy graphs with correlation decay \cite{anandkumar2013learning,anandkumar2012high}, and mixtures of trees \cite{anandkumar2012learningB,meila2001learning}.

Loopy belief propagation yields accurate marginals in high girth (locally tree-like) graphs with correlation decay. This fact was used by  Heinemann and Globerson \cite{heinemann2014inferning} to justify an algorithm that recovers all the edges of a model from this family, given sufficiently many samples. The output of their algorithm can have extra edges, which are proved to be weak. 
Given number of samples at least linear in $p$,
they show that the learned distribution is at most constant Kullback-Leibler divergence from the true one.

Narasimhan and Bilmes \cite{PACBilmes} found an algorithm with polynomial runtime that uses a polynomial number of samples  to learn bounded tree-width graphical models with respect to KL-divergence. They use ideas from submodular optimization and the specific factorization of the distribution over bounded tree-width graphs.

Structure learning of latent tree models has been well-studied in the phylogenetic reconstruction literature. Erd{\"o}s et al. \cite{erdos1999few} studied sample and time complexity of tree metric based algorithms to reconstruct phylogenetic trees. \rvw{Daskalakis et al.} \cite{DaskalakisContractPrune} and Mossel \cite{Mosseldistortedmetric} use distorted tree metrics to get approximations of phylogenetic trees when exact reconstruction of the tree is impossible. In \cite{Mosseldistortedmetric} a forest approximation of the latent tree is recovered. The maximum number of connected components in this forest is a function of the edge strengths, maximum distance between the leaves and the number of leaves.  Daskalakis et al.\cite{DaskalakisContractPrune} removed the prior assumptions on the phylogenetic tree and instead a forest structure is recovered that contains all edges that are sufficiently strong and sufficiently close to the leaves.

A tree metric over $p$ nodes is associated with a weighted spanning tree such that the distance between every pair of nodes is the sum of weights of the edges along the path between the nodes in the tree. 
 Agarwala et al. \cite{agarwala1998approximability}, approximate a pairwise distance matrix $D$ over $p$ nodes by a tree metric with induced distance matrix $T$. Let $\eps=\min_{T}\{\|T-D\|_{\infty}\}\,$ where $T$ is a tree metric. They propose an $O(p^2)$ algorithm which produces $\widehat{T}$ with $\|\widehat{T}-D\|_{\infty}\leq 3\eps$. They prove that finding a $T$ with $\|T-D\|_{\infty}\leq 9/8\eps$ is NP-hard. Ambainis et al. in \cite{ambainis1997nearly} studied the \textit{leaf variational distance} between the original distribution on a latent tree under the Cavender-Faris (CF) model and the learned latent tree. Let tree $\T$ with $p$ leaves $\VV$ be a CF-tree with the property that all its edges are of length at least $1/\sqrt{n}$ (this is translated as $\log \tanh \alpha <1/\sqrt{n}$ in our model). Then, given $n$ observations, their proposed algorithm produces a distribution $Q$ on a tree $\T'$ with leaves $\VV$ such that $\LL^{(2)}(P_{\VV},Q_{\VV})= O(\sqrt{{p e^{2\beta}}/{n}})$ ($P_{\VV}$ and $Q_{\VV}$ are marginals of $P$ and $Q$ on leaves $\VV$). We will further discuss these results and compare them with our setup in detail in Appendix \ref{sec:comparisonLiterature}.

Wainwright \cite{wainwright2006wrong} was motivated by the same general problem of learning a graphical model to be subsequently used for making predictions, but his focus was on computational rather than statistical limits.
For loopy graphs, both estimation of parameters and prediction based on partial observations are computationally challenging tasks.
Hence, for both, approximate heuristic methods are often used. For given model parameters, one such heuristic for prediction is re-weighted sum-product (a convex relaxation).
Intriguingly, when using such approximate prediction algorithms, an \textit{inconsistent} procedure for estimation of parameters can give better predictions. The results elucidate asymptotic performance, but the analysis does not apply to the high-dimensional setting of interest, with dimension $p$ larger than number of samples $n$.


\subsection{Outline of paper} 

The next section contains background on the Ising model, tree models, and graphical model learning. Section~\ref{s:ExactTree} introduces the problem of learning tree-structured Ising models and records the sample complexity of exact recovery. Then, in Section~\ref{s:Inferning} we define the small-set TV loss function motivated by prediction computations and state our main result in Section~\ref{s:resultStatement}. Section~\ref{s:MarkovChain} analyzes an illustrative example that gives intuition for the main result. Section~\ref{s:truncation} introduces a natural forest approximation algorithm and analyzes its performance in terms of ssTV. Section~\ref{s:ProofSketch} sketches the proof of the main result and Section~\ref{s:ProofProp} fills in the details. Sections~\ref{s:ExactRecovery} and~\ref{sec:lemmaProofs} contain further proofs. Numerical simulations addressing the theorems in the paper are in Section~\ref{s:NumSim}. \rvw{Appendices contain additional proofs, numerical simulations, and discussions on related work. }

\section{Preliminaries}

\subsection{Notation}
For a given tree $\T=(\VV,\EE)$ and positive numbers $\al$ and $\beta$, let $\PP_{\T}(\al,\beta)$ be the set of Ising models~\eqref{e:Ising} with the restriction $\al\leq |\theta_{ij}|\leq \beta$ for each edge $(i,j)\in \EE$ and $\theta_{ij}=0$ for $(i,j)\notin \EE$. Denote by  $\PP_{\T} = \PP_{\T}(0,\infty)$ the set of Ising models on $\T$ with no restrictions on parameter strength.

Denote by $\mu_{ij}=\E_{P}X_i X_j$ the correlation between the variables corresponding to  $i,j\in \mathcal{V}$. For an edge $e=(i,j)$ we write $\mu_{e}=\mu_{ij}$ and for a set of edges $\A\subseteq \EE$, $\mu_\A=\prod_{e\in\A} \mu_e$. 
Given $n$ i.i.d. samples $X^{(1:n)}=X^{(1)},\dots, X^{(n)}$, the empirical distribution is denoted by $
\Ph(x)  = \frac{1}{n}\sum_{l=1}^n\mathbf{1}_{\{X^{(l)} = x\}} $
and $\wmu_{ij}=\E_{\Ph}X_i X_j$ is the empirical correlation between nodes $i$ and $j$.

\subsection{Tree models}

A probability measure $P$ on $\XX^{\VV}$ is Markov with respect to a graph $G=(\VV,\EE)$ if for all $i \in \VV$, we have $P(x_i|x_{\VV\setminus\{i\}}) = P(x_i|x_{\partial i})$, where $\partial i$ is the neighborhood of $i$ in $G$. 
In this paper, we are interested in distributions $P$ that are Markov with respect to a tree $\T=(\VV,\EE)$, and a consequence (see \cite{lauritzen1996graphical}) is that $P(x)$ factorizes as
\begin{equation}\label{eq:TreeFac}
	P(x)=\prod_{i\in \VV}P(x_i) \prod_{(i,j)\in\EE}\frac{P(x_i,x_j)}{P(x_i)P(x_j)}.
\end{equation}

\subsection{Information projection}\label{sec:infprojdef}
Denote by $D(Q\|P)$ the Kullback-Leiber divergence between probability measures $Q$ and $P$ defined as $D(Q\|P) = \sum_{x\in\XX} Q(x)\log\frac{Q(x)}{P(x)}$.
For an arbitrary distribution $P$ and tree $\T$, the distribution 
$$\widetilde{P}(x)=\argmin_{Q \text{ is factorized as~\eqref{eq:TreeFac}} \atop \text{according to} \T} D(P\|Q)$$ 
is the best approximation to $P$ within the set of distributions Markov with respect to the tree $\T$. 
It was observed by Chow and Liu in \cite{chow1968approximating} that $\widetilde{P}$ is obtained by matching the first and second order marginals to those of $P$, i.e., for all $(i,j)\in\EE$, and all $x_i, x_j\in\mathcal{X}$, $\widetilde{P}(x_i,x_j)=P(x_i,x_j)$. 

Let \begin{equation}\label{eq:projoperator}
\Pi_{\T}(P)=\argmin\limits_{Q\in \mathcal{P}_{\T}}D(P\|Q )
\end{equation}
 be the reverse information projection of $P$ onto the class of Ising models on $\T$ with no external field. 
It follows from the definition of $\mathcal{P}_{\T}$ that $\widetilde{P}=\Pi_{\T}(P)$ can be represented as Equation~\eqref{e:Ising} for some $\widetilde{\theta}$ supported on $\T$.
It is shown in Appendix~\ref{sec:InfProj} in the supplementary material that 
$\widetilde{P}=\Pi_{\T}(P)$ has edge weights $\wttheta_{ij}$ for each $(i,j)\in \EE_{\T}$ satisfying $\tanh\widetilde{\theta}_{ij}=\mu_{ij}\triangleq \E_P X_i X_j$ (and $\widetilde{\theta}_{ij}=0$ if $(i,j)\notin\EE_{\T}$).

\subsection{Tree structure learning}
Denote the set of all trees on $p$ nodes by $\TT$.
For some tree $\T$ and distribution $P\in \PP_{\T}$, one observes $n$ independent samples (configurations) $X^{(1)},\dots, X^{(n)}\in \{-,+\}^p$ from the Ising model~\eqref{e:Ising}. 
In this context, a \emph{structure learning algorithm} is a (possibly randomized) map $\phi:\{-1,+1\}^{p\times n}\to \TT$
taking $n$ samples $X^{(1:n)}=X^{(1)},\dots, X^{(n)}$ to a tree $\phi(X^{(1:n)})$.  

The maximum likelihood tree or Chow-Liu tree plays a central role in tree structure learning.  
Chow and Liu \cite{chow1968approximating} observed that the maximum likelihood tree is the max-weight spanning tree in the complete graph, where each edge has weight equal to the empirical mutual information between the variables at its endpoints. The tree can thus be found greedily via Kruskal's algorithm  \cite{chow1968approximating,cormen2009introduction}, and the run-time is dominated by computing empirical mutual information between all pairs of nodes.

In order to support the following definition, we analyzed zero-field Ising models on trees in Lemma~\ref{l:TcMWST} in Appendix~\ref{sec:InfProj}. 
This analysis is similar to \cite{chow1968approximating}. 
 
\begin{definition}[Chow-Liu Tree]
	\label{def:Chow-Liu}
	Given $n$ i.i.d samples $X^{(1:n)}$ from distribution $P\in \PP_{\T}$, we define the Chow-Liu tree to be the maximum likelihood tree: 
	$$\Tc=\mathop{\mathrm{argmax}}_{\T\in\TT}\mathop{\mathrm{max}}_{P\in \PP_{\T}}P(X^{(1:n)})\,.$$
\end{definition}
This definition is slightly abusing the conventional terminology, as the Chow-Liu tree is classically the maximum likelihood tree assuming that the generative distribution is tree-structured \cite{chow1968approximating}, whereas in our definition we assume that the original distribution $P\in\PP_{\T}$ can be described by~\eqref{e:Ising}. Thus, it is not only tree-structured,	 but also has uniform singleton marginals. 

Note that maximizing the likelihood of i.i.d. samples corresponds to minimizing the KL divergence.
 Given the samples with empirical distribution $\Ph$, $\Tc=\argmin_{\T\in\TT}\min_{P\in \PP_{\T}}D(\Ph\|P)$. 
 It is shown in Lemma~\ref{l:MomentMatching} in Appendix~\ref{sec:InfProj} that
	\begin{equation}
	\label{e:ChowLiuTree}
	\Tc = \mathop{\mathrm{argmax}}_{\{\text{spanning trees } \T'\}}\sum_{e\in \EE_{\T'}} |\widehat{\mu}_e|\,,
	\end{equation}
	where for $e=(i,j)$, $\wmu_{e}=\E_{\Ph}X_iX_j$ is the empirical correlation between variables $X_i$ and $X_j\,.$

Chow and Wagner~\cite{chow1973consistency}  showed 
that the maximum likelihood tree is consistent for structure learning of general discrete tree models, i.e., in the limit of large sample size the correct graph structure is found. 
More recently, detailed analysis of error exponents was carried out by Tan et al.~\cite{tan2011large, tan2011learning}. 
A variety of other results and generalizations have appeared, including for example Liu et al.'s work on forest estimation with non-parametric potentials~\cite{liu2011forest} (we will not address general potentials in this paper).  

\section{Learning trees to make predictions} \label{s:learningforPred}
In order to place the learning for predictions problem into context, we first discuss the problem of exact structure learning and give tight (up to a constant factor) sample complexity for that problem. Then, in Section \ref{s:Inferning} we define the ssTV distance $\LL^{(k)}$, explain how it relates to prediction, and in Section~\ref{s:resultStatement} we
state our results.

\subsection{Exact recovery of trees}
\label{s:ExactTree}
The statistical performance of a structure learning algorithm is often measured using the zero-one loss, 
\begin{equation}
\label{e:zerooneLoss}
\LL^{0-1}(\T,\T') = \mathbf{1}_{\{\T\neq \T'\}}\,,
\end{equation}
meaning that the exact underlying graph must be learned (see e.g., \cite{BMS08, santhanam2012information,
tan2011learning,liu2011forest}).
The risk, or expected loss, of algorithm $\phi$ under some distribution $P\in \PP_{\T}(\al,\beta)$ is then given by the probability of reconstruction error,
$
\E_P\LL^{0-1}(\T,\phi(X^{(1:n)})) = \P (\phi(X^{(1:n)})\neq \T)\,,
$
and the maximum risk is
$\sup\{\P (\phi(X^{(1:n)})\neq \T): \T\in \TT, P\in\PP_\T(\al,\beta)\}$   for given $\al, \beta, p$ and $n$.

 The sample complexity of learning the correct tree underlying the distribution increases as edges become weaker, i.e., as $\alpha\to 0$, because weak edges are harder to detect. As the bound on maximum edge parameter $\beta$ increases, there is a similar increase in sample complexity (as shown by \cite{santhanam2012information,tandon2014information} for Ising models on general bounded degree graphs). In the context of tree-structured Ising models we have the following theorem:

\begin{theorem}[Samples necessary for structure learning]\label{t:structLower}
	Given $n<\frac1{16} e^{2\beta}/(\alpha \tanh\alpha)\log p$ samples, the worst-case probability of error over trees $\T\in \TT$ and distributions $P\in \PP_{\T}(\al,\beta)$ is at least half for any algorithm, i.e., $$\inf_{\phi}\sup_{\T\in \TT\atop P\in \PP_{\T}(\al,\beta)} P \big[\phi(X^{\small{(1:n)}})\neq \T\big] > 1/2\,.$$  
\end{theorem}
 The proof, given in Section~\ref{s:structLowerPf}, applies Fano's inequality (Lemma~\ref{e:Fano}) to a large set of trees that are difficult to distinguish. The next theorem gives an essentially matching sufficient condition.

\begin{theorem}[Samples sufficient for structure learning]\label{t:structUpper}
	Fix an arbitrary tree $\T$ and Ising model $P\in\PP_{\T}(\al,\beta)$. If the number of samples is  $n>C e^{2\beta}\tanh^{-2}(\al)\log(p/\delta)$, then with probability at least $1-\delta$ the Chow-Liu algorithm recovers the true tree, i.e., $\Tc = \T$.
\end{theorem}

The proof is presented in Section~\ref{s:structUpperPf}. Assuming that $\alpha$ is bounded above by a constant (which is the interesting regime), Theorems~\ref{t:structLower} and~\ref{t:structUpper} give matching bounds (up to numerical constant) for the sample complexity of learning the tree structure of an Ising model with zero external field. 
The necessary number of samples increases as the minimum edge weight $\alpha$ decreases, so if edges can be arbitrarily weak, it is impossible to learn the tree given any bounded number of samples. \rvw{Figures~\ref{fig:alpha-pr} and~\ref{fig:beta-pr} in Section~\ref{s:NumSim} present numerical simulation results supporting this observation.}

If the goal is merely to make accurate predictions, it is natural to seek a less stringent, approximate notion of learning. 
Several papers consider learning a model that is close in KL-divergence, e.g. \cite{abbeel2006learning,liu2011forest,
heinemann2014inferning,jog2015model,tan2011learning}.  The sample complexity of learning a model to within constant KL-divergence $\epsilon$ scales at least \emph{linearly} with the number of variables $p$, an unrealistic requirement in the high-dimensional setting of interest. Using a number of samples scaling logarithmically in dimension requires relaxing the KL-divergence to scale linearly in $p$, but this does not imply a non-trivial guarantee on the quality of approximation for marginals of few variables (as done in this paper). The same observation is true for the total variation as the measure of distance. The sample complexity of learning a model to within constant TV distance between the learned model \rvw{and} the original joint distribution over $p$ variables scales at least linearly with $p$.

In the next section,
we study estimation with respect to the small-set TV loss, which captures accuracy of prediction based on few observations. We will see that the associated sample complexity is independent of the edge strength lower bound $\alpha$ in the original model.

\subsection{Small set total variation}
\label{s:Inferning}

For a subset of nodes $\SS\subseteq [p]$, we denote by $P_{\SS}$ the marginal distribution $P_\SS(x_\SS) =\sum_{x_{\VV\setminus\SS}}P(x) $.

Given two distributions $P$ and $Q$ on the same space, for each $k\geq 1$ the small-set total variation distance is the maximum total variation over all size $k$ marginals, and is denoted by
$$\LL^{(k)}(P,Q)\triangleq \max_{\SS:|\SS|= k} d_\mathrm{TV}(P_\SS , Q_\SS)\,.$$ Note that $\LL^{(k)}$ is non-decreasing in $k$. 
One can check that $\LL^{(k)}$ satisfies the triangle inequality: for any three distributions $P,R,Q$, 
\begin{equation}\label{e:triangle}
\LL^{(k)}(P,R) + \LL^{(k)}(R,Q) \geq \LL^{(k)}(P,Q)\,.
\end{equation}

Closeness of $P$ and $Q$ in $\LL^{(k)}$ implies that the respective posteriors conditioned on subsets of variables of size $k-1$ are close on average.
To see this, suppose that we wish to compute $P(X_i=+|X_\SS)$. We measure performance of the approximation $Q$ by the expected magnitude of error $|P(x_i=+|X_\SS) - Q(x_i=+|X_\SS)|$ averaged over $X_\SS$: 
\begin{align*}
\Ex_{{X_\SS}} |P(X_i=+&|X_\SS) - Q(X_i=+|X_\SS)|\\
 &=\sum_{x_\SS} |P(X_i=+,X_\SS=x_\SS) - Q(X_i=+|X_\SS=x_\SS)P(X_\SS=x_\SS)|
\\
&\leq \sum_{x_\SS} |P(X_i=+,X_\SS=x_\SS) -Q(X_i=+,X_\SS=x_\SS) |
+ \sum_{x_\SS}|Q(x_\SS) 
-P(x_\SS)|
\\&\leq 2 \LL^{(|\SS|+1)}(P,Q)\,.
\end{align*}
The last inequality is a consequence of monotonicity of  $\LL^{(k)}$ in $k$. 

In this paper we focus mostly on $\LL^{(2)}$. Implications for $k\geq 3$ are stated in Corollary~\ref{cor:margK}, and discussed in Appendix~\ref{sec:MarginalK}.  For trees $\T,\Tt\in\TT$ and  distributions $P\in\mathcal{P}_{\T}$ and $\widetilde{P}\in\mathcal{P}_{\widetilde{\T}}$, let $e=(i,j)\in\EE_{\T}$, $\mu_{e}=\Ex_{P}X_iX_j$ and for $e'=(i,j)\in\EE_{\Tt}$, $\widetilde{\mu}_{e'}=\Ex_{\widetilde{P}}X_iX_j$. It will be useful to express $\LL^{(2)}$ as
\begin{align}
\LL^{(2)} (P , \widetilde{P}) & = \max_{w,\wt \in \VV}  \frac{1}{2} \sum_{x_w,x_{\wt}\in\{-,+\}^2} |P(x_w,x_{\wt})-\widetilde{P}(x_w,x_{\wt})|\nonumber\\
& = \max_{w,\wt \in \VV} \, \frac{1}{2} \left|\prod_{e\in\path_{\T}(w,\wt)}\mu_e-\prod_{e'\in\path_{\Tt}(w,\wt)}\widetilde{\mu}_{e'}\right|\,. \label{e:LL2}
\end{align}
The second equality is derived by noting that $P(x_i=+)=1/2$, $P(x_w,x_{\wt})=\left[1+x_w x_{\wt}\Ex_{P}[X_w X_{\wt}]\right]/4$ and analogously for $\widetilde{P}$. Also, as noted above after~\eqref{e:Ising}, it is immediate from Lemma~\ref{l:IndependentError} that $\Ex_{P} X_w X_{\wt} =\prod_{e\in\path_{\T}(w,\wt)}\mu_{e}$. The same holds for $\widetilde{P}\in\PP_{\Tt}$, which gives~\eqref{e:LL2}.
 
\subsection{Main result}
\label{s:resultStatement}

Our main contribution is to prove upper and lower bounds on the number of samples required to estimate a tree close in $\LL^{(2)}$ to the true one. 
An upper bound on the number of samples is obtained for the Chow-Liu algorithm by bounding the expression in~\eqref{e:LL2}. 
The Chow-Liu algorithm produces the maximum likelihood tree, which minimizes the expected zero-one loss in~\eqref{e:zerooneLoss}. 
As shown in Theorem~\ref{t:CLLowerBound}, the maximum likelihood tree also performs well in terms of accuracy of pairwise marginals. 

Recall from~\eqref{eq:projoperator} that $\Pi_{\T}(P)$ is the reverse information projection of the distribution $P$ onto the set of zero-field Ising models on tree $\T$.

\begin{theorem}[Learning for predictions using Chow-Liu algorithm] 
	\label{t:CLLowerBound}
 For $\T\in\TT$, let the distribution $P\in \PP_{\T}(0,\beta)$. Given \rvw{$n> C \max\{e^{2\beta},\eta^{-2}\}\log\frac{p} {\delta}$} samples, if  $\Tc$ is the Chow-Liu tree as defined in~\eqref{e:ChowLiuTree}, then with probability at least $1-\delta$ we have 
	$ \LL^{(2)}(P,\Pi_{\Tc}(\Ph))<\eta\,.$
	\end{theorem}

The main challenge is that the number of samples assumed to be available in Theorem~\ref{t:CLLowerBound} is not sufficient for structure learning, as can be seen by comparing with Theorem~\ref{t:structLower}. This means that accurate marginals must be computed \emph{using possibly the wrong tree}. The proof is sketched in Section~\ref{s:ProofSketch}. 

We also lower bound the  number of samples necessary for small $\LL^{(2)}$ loss. Let the learning algorithm be $\Psi: \{-1,+1\}^{p\times n}\to \PP$ where $\PP=\cup_{\T} \PP_{\T}$ is the set of tree-structured Ising models with no external field defined in~\eqref{e:Ising}.

\begin{theorem}[Samples necessary for small ssTV]\label{t:inferenceUpper}
	Fix $\eta>0$. Suppose $\tanh(\beta)>\tanh(\alpha)+2\eta$ and  \rvw{$n<\,C \,[1-(\tanh  (\alpha)+2\eta)^2]\,\eta^{-2}\,\log p$}. The worst-case probability of $\LL^{(2)}$ loss greater than $\eta$, taken over trees $\T\in \TT$ and distributions $P\in \PP_{\T}(\al,\beta)$, is at least half for any algorithm, i.e., $$\inf_{\Psi}\sup_{\T\in \TT\atop P\in\PP_{\T}(\al,\beta)} \P\left[\LL^{(2)}(P,\Psi(X^{\small{(1:n)}})) > \eta\right]>1/2\,.$$  
\end{theorem}

Theorem~\ref{t:inferenceUpper} is proved in Section~\ref{s:pairwiseUpper}. 
\rvw{As noted earlier, the assumption $\tanh(\beta)>\tanh(\alpha)+2\eta$ captures the scenario that the precise values of the edge parameters are not known a priori. In the extreme where $\alpha = \beta$, the entire problem is quite different (and we believe less realistic). We state a lower bound for this setting in Appendix~\ref{sec:LBSmallSSTV}, which is not directly comparable to the above theorems.}

If $\alpha$ is bounded above by a constant, then Theorems~\ref{t:CLLowerBound} and~\ref{t:inferenceUpper} have the same dependence on $\eta$. 
  The theorems imply the following bounds on risk.

\begin{corollary}[Upper bound for risk] 
\label{t:RiskUppererBound}
 For $\T\in\TT$, let the distribution $P\in \PP_{\T}(0,\beta)$. Given $n$  samples with empirical distribution $\Ph$, if the tree $\Tc$ is the Chow-Liu tree as defined in~\eqref{e:ChowLiuTree}, then 
	\rvw{$$\Ex{\left[\LL^{(2)}(P,\Pi_{\Tc}(\Ph))\right]}< C' p\exp(-C n e^{-2\beta})+ C''\sqrt{\frac{ \log p}{{n}}} \,.$$}
\end{corollary} 
\vspace{-.2in}
\begin{corollary}[Lower bound for risk]
\label{t:RiskLowererBound}
Suppose $\tanh(\beta)-\tanh(\alpha)>1/2$ and one observes  $n$ samples. Then the minimax risk over trees $\T\in \TT$ and distributions $P\in \PP_{\T}(\al,\beta)$ is lower bounded by \rvw{$$\inf_{\Psi}\sup_{\T\in \TT\atop P\in\PP_{\T}(\al,\beta)} \Ex\left[\LL^{(2)}(P,\Psi(X^{(1:n)})) \right]>\min\left\{\frac{1}{24}\sqrt{\frac{\log p}{n}},\frac{1}{12}\right\}\,.$$}
\end{corollary}
\rvw{Corollary~\ref{t:RiskUppererBound} is proved in Appendix~\ref{s:risk-upper-bd}.
Proof of Corollary~\ref{t:RiskLowererBound} is immediate from the statement of Theorem~\ref{t:inferenceUpper}.}
The upper bound in Corollary~\ref{t:RiskUppererBound} has an extra term that depends on $\beta$ compared to the lower bound of Corollary~\ref{t:RiskLowererBound}. 
We conjecture that the lower bound is tight.
\section{Illustrative example and algorithm comparison}
\subsection{Three node Markov chain}
\label{s:MarkovChain} 

A Markov chain with three nodes captures a few of the key ideas developed in this paper. 
Let $P(X_1,X_2,X_3)$ be represented by the tree $\T_1$ in Figure~\ref{f:MarkovChainExample} in which $X_1\leftrightarrow X_2\leftrightarrow X_3$ form a Markov chain with correlations $\mu_{12}$, $\mu_{23}$ and $\mu_{13}=\mu_{12}\mu_{23}$. Without loss of generality, we assume $\mu_{12},\mu_{23} > 0$. Suppose that for some small value $\epsilon$, $\mu_{12}=1-\mu_{23}=\epsilon$.

Given $n$ samples from $P$, the empirical correlations $\widehat{\mu}_{12}$, $\widehat{\mu}_{23}$ and $\widehat{\mu}_{13}$ are concentrated around $\mu_{12}=\eps$, $\mu_{23}=1-\eps$ and $\mu_{13}=\mu_{12}\mu_{23}=\eps(1-\eps)$. Let $\widehat{\mu}_{12}=\mu_{12}+z_{12}$, $\widehat{\mu}_{23}=\mu_{23}+z_{23}$ and $\widehat{\mu}_{13}=\mu_{12}\mu_{23}+z_{13}$, where the fluctuations of $z_{12},\, z_{23}$ and $z_{13}$ shrink as $n$ grows. It is useful to imagine the typical fluctuations of $z_{ij}$ to be on the order $\eps/10$.

\begin{figure}
			\begin{center}
				\includegraphics[width=3in]{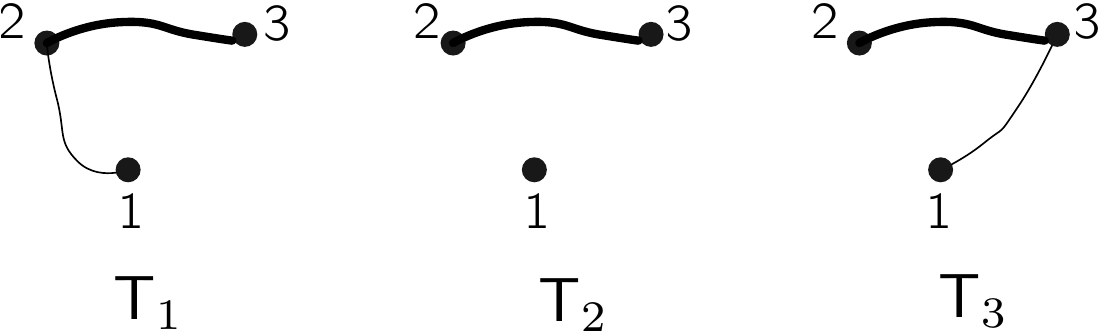}
				\caption{The original distribution factors according to $\T_1$. The width of the edges corresponds to their strength. The forest approximation algorithm (defined in Section~\ref{s:truncation}) recovers $\Th=\T_2$ since the correlation between $X_1$ and each of the other variables is not strong enough to confidently recover any edge to node $1$. The Chow-Liu tree $\Tc$ is either $\T_1$ or $\T_3$, depending on the realization of the samples.}
			\label{f:MarkovChainExample}
			\end{center}
\end{figure}

Since $\max\{\mu_{12},\mu_{13}\}=\max\{\epsilon\,,\,\epsilon(1-\epsilon)\}=\epsilon\ll \mu_{23}$, concentration bounds guarantee that  with high probability $\wmu_{23} > \max\{\wmu_{12}\,,\,\wmu_{13}\}$ and  the (greedy implementation of) Chow-Liu algorithm described in~\eqref{e:ChowLiuTree} adds the edge $(2,3)$  to $\Tc$. However, because $\mu_{12}-\mu_{13}=\epsilon^2$ is smaller than the fluctuations of $z_{12}$ and $z_{13}$  there is no guarantee that $\widehat{\mu}_{12} > \widehat{\mu}_{13}$: if $z_{13}-z_{12} > \epsilon^2$, then edge $(1,3)$ is added and $\Tc=\T_3\neq \T_1$. 

The preceding discussion provides the intuition underlying a statistical characterization of the possible errors made by the Chow-Liu algorithm. To make this quantitative, later on in the proof, we determine a value  $\tau:=\tau(n,\beta,\delta)$ so that any (strong) edge $e$ with $|\mu_{e}|\geq \tau$ is recovered by the Chow-Liu algorithm with probability at least $1-\delta$. Equivalently, if there is a mistake made by the Chow-Liu algorithm such that $e\in\EE_{\T}$ but $e\notin\EE_{\Tc}$, then $|\mu_e|\leq \tau$ (i.e., missed edges are weak). This is going to play a key role in bounding the ssTV $\LL^{(2)}$ for $\Tc$, whether or not it is equal to $\T$.

In the regime where $n$ is not large enough to guarantee the correct recovery of all the edges in the tree, there are two natural strategies:
\paragraph{I. Forest approximation algorithm}
This algorithm attempts to recover a forest $\mathsf{F}$, a good approximation of the original tree $\T$ in the sense that $\EE_{\mathsf{F}}\subseteq \EE_{\T}\,.$
This is accomplished by finding a forest consisting of sufficiently strong edges, i.e. having weight at least $\tau$ for an appropriate value of $\tau$. As shown by Tan et al. in \cite{tan2011learning}, such a forest can be obtained by running the Chow-Liu algorithm and removing the edges with weight below $\tau$.
The details of this algorithm and its sample complexity will be discussed in Section \ref{s:truncation}.
\paragraph{II. Chow-Liu Algorithm} We can use the Chow-Liu tree as our estimated structure despite the fact that it may well be incorrect.

For our three-node example, the forest approximation algorithm would return $\Th=\T_2$ in Figure~\ref{f:MarkovChainExample}, whereas Chow-Liu would give $\Th=\T_1$ or $\Th=\T_3$. 
To focus on the implication of graph structure estimation (as opposed to parameter estimation), we compare the \emph{loss due to graph estimation error}, defined as $\LL^{(2)}(P,\Pi_{\Th}(P))$, for the above cases:
\begin{align*}
\Th= \T_1 & \quad\rightarrow \quad\LL^{(2)}(P,\Pi_{\Th}(P)) = 0\,, \\
\Th= \T_2 & \quad\rightarrow \quad\LL^{(2)}(P,\Pi_{\Th}(P)) = |\mu_{12}|=\eps\,, \\
\Th= \T_3 & \quad \rightarrow\quad \LL^{(2)}(P,\Pi_{\Th}(P)) = |\mu_{12}|(1-\mu_{23}^2)=\epsilon^2 (2-\epsilon)\,.
\end{align*}
Evidently, the loss due to graph estimation error in the forest approximation algorithm is bigger than the Chow-Liu algorithm, whether or not the latter recovers the true tree. This is because the Chow-Liu algorithm does not make arbitrary errors in estimating the tree: errors happen when both the original tree and the estimated tree describe the original distribution rather well. Theorem~\ref{t:CLLowerBound} makes this formal.

\subsection{Forest approximation}		
\label{s:truncation}
In the regime where exact recovery of the tree is impossible, a reasonable goal is to instead find a forest approximation to the tree. In this section we analyze a natural truncation algorithm, which thresholds to zero edges with correlation below a specified value $\tau(\epsilon)$. 

There is extensive literature on estimating forests in the fully observed setting of this paper, including \cite{tan2011learning,liu2011forest}. Mossel in \cite{Mosseldistortedmetric} studied the problem of learning phylogenetic forests, where samples are only observed at the leaves of the tree. They quantified the idea that most edges of phylogenies are easy to reconstruct. In the regime that the  sample complexity of structure learning is too high, they instead estimate a forest. An upper bound on the number of edges necessary to glue together the forest to get the original tree is provided as a function of the number of leaves, the minimum edge weight, and the metric distortion bounds. Our results in this section are consistent with the asymptotic conditions on the thresholds given in \cite{tan2011learning} by 
Tan et al. for forest approximation of general distributions over trees. 

The forest approximation algorithm considered in this section thresholds to zero the edges with correlations below $\tau(\eps)=\frac{4\epsilon}{\sqrt{1-\tanh\beta}}$ for $\epsilon=\sqrt{2/n\log(2p^2/\delta)}$. 
This is equivalent to finding the maximum-weight spanning forest over the complete graph with edge weights $|\wmu_{e}|-\tau(\eps)-\eps$. The output $\Th=(\VV,\EE_{\Th})$ is a truncated version of $\Tc$ such that the empirical correlation between any pair of nodes $(i,j)\in\EE_{\Th}$ satisfies $|\wmu_{ij}| \geq \tau(\eps)+\epsilon$.
\begin{proposition} \label{prop:TruncationUpper}
Given \rvw{$n> C e^{2\beta}\eta^{-2}\log\frac{p}{\delta}$} samples, the forest approximation algorithm guarantees that $\EE_{\Th}\subseteq \EE_{\T}$ and  $\LL^{(2)}(P,\Pi_{\Th}(\Ph))<\eta$ with probability at least $1-\delta$ . 
\end{proposition}

The proof is presented  in Appendix~\ref{sec:ForestApprox} in 
the supplementary material.  \rvw{The forest approximation algorithm is trying to avoid adding incorrect edges and we then measure its performance according to $\LL^{(2)}$. Given this objective, it is natural to consider the loss
\begin{equation}\label{eq:forest-loss}
\widetilde{d}(P,Q)= \max\big\{\LL^{(2)}(P,Q), {1\!\!1}[\EE_{\mathsf F} \nsubseteq\EE_{\T}]\big\}\,,
\end{equation}
which is one if the learned forest is not a subgraph of the original tree, and otherwise is equal to $\LL^{(2)}$.

It turns out that the forest approximation algorithm is optimal (up to constant factor) for this loss, as shown in the following proposition.

\begin{proposition}\label{pro:forestLB}
Fix $\eta>0$. Let $\Psi$  be an estimator that takes $n$ samples generated from $\P_{\T}(0,\beta)$ and recovers a forest $\mathsf F$ and a distribution $Q\in\P_{\mathsf F}(0,\beta)$. Let the (asymmetric) distance $\widetilde{d}$ be defined in \eqref{eq:forest-loss}.
If $n<C\min\{p,e^{2\beta}\}/[\eta \,\mathrm{atanh}\eta)]\log p$, then  the worst-case probability of $\widetilde{d}$ loss greater than $\eta$, taken over trees $\T\in \TT$ and distributions $P\in \PP_{\T}(\al,\beta)$, is at least half for any algorithm, i.e., $$\inf_{\Psi}\sup_{\T\in \TT\atop P\in\PP_{\T}(\al,\beta)} \P\left[\widetilde{d}\big(P,\Psi(X^{\small(1:n\small)})\big) > \eta\right]>1/2\,.$$ 
\end{proposition}
}
Comparison with Theorem~\ref{t:CLLowerBound} shows that avoiding adding wrong edges (as forced by the loss function~\eqref{eq:forest-loss}) entails a degradation in performance.

\section{Outline of Proof}

\label{s:ProofSketch}

We now sketch the argument for the main result, Theorem \ref{t:CLLowerBound}, guaranteeing accurate pairwise marginals in the Chow-Liu tree.
The starting point is an application of the triangle inequality~\eqref{e:triangle}:
\begin{equation}
\label{e:TrianLoss}
\LL^{(2)}(P,\Pi_{\Tc}(\Ph))\leq \LL^{(2)}(P,\Pi_{\Tc}(P))+ \LL^{(2)}(\Pi_{\Tc}(P),\Pi_{\Tc}(\Ph))\,.
\end{equation}
The first term on the right-hand side of Equation~\eqref{e:TrianLoss} represents the error due to the difference in the structure of $\T$ and $\Tc$, so we call this first term the \textit{loss due to graph estimation error}.  Equation~\eqref{e:LL2} tells us that for each pair of nodes $u,v\in \VV\,,\,$ $\path_{\T}(u,v)$ and $\path_{\Tc}(u,v)$ must be compared. 

The second term on the right-hand side of  Equation~\eqref{e:TrianLoss} represents the propagation of error 
due to inaccuracy in estimated parameters. Recall that the estimated parameters on the Chow-Liu tree are obtained by  matching correlations to the empirical values. 

Theorem \ref{t:CLLowerBound} follows by separately bounding each term on the right-hand side of Equation~\eqref{e:TrianLoss}.

\begin{proposition}[Loss due to parameter estimation error]
	\label{p:EdgeEstError}
	Given \rvw{
	$n>C \max\{e^{2\beta},\,\eta^{-2}\} \log\frac{p}{\delta}$ }samples, with probability at least $1-\delta$ we have $\LL^{(2)}(\Pi_{\Tc}(P),\Pi_{\Tc}(\Ph)\leq\eta$.
\end{proposition}

\begin{proposition}[Loss due to graph estimation error]
	\label{p:projectedDist}
	Given $n>C' \max\{e^{2\beta},\,\eta^{-2}\} \log\frac{p}{\delta}$ samples, with probability at least $1-\delta$ we have $ \LL^{(2)}(P,\Pi_{\Tc}(P))\leq\eta$.
\end{proposition}

These two propositions are proved in full detail in Sections \ref{s:paramErrpr} and \ref{s:incorrectGraph}. In the remainder of this section we define probabilistic events of interest and sketch the proofs of the propositions. 

We define three highly probable  events $\Ecorr, \Pcorr$ and $\Ncorr$ as follows.
Let $\Ecorr$ be the event that all empirical correlations are within $\epsilon$ of population values:
\begin{equation}\label{e:Ecorrdef}
\Ecorr=\left\{\max_{w,\wt\in\VV}|\mu_{w,\wt}-\widehat{\mu}_{w,\wt}|\leq \eps\right\}\,.
\end{equation}

Let
\begin{equation}\label{e:Estrong}
\tau(\eps) = \frac{4\epsilon}{\sqrt{1-\tanh \beta}} 
\quad 
\text{and}
\quad
	\EEstrong_{\T}(\eps)=\big\{(i,j)\in\EE_{\T}:|\mu_{ij}|> \tau(\eps) \big\}
\end{equation}
consist of the set of strong edges in tree $\T$. Let $\Tc$ be the Chow-Liu tree defined in Equation~\eqref{e:ChowLiuTree}. Weak edges are those that are not strong, \textit{i.e.}, $\EE_{\T}\setminus \EEstrong_{\T}(\eps)$.
Let $\Pcorr$ be the event that all strong edges in $\T$ as defined in~\eqref{e:Estrong} are recovered by the Chow-Liu tree	:
\begin{equation}\label{e:Pcorr}
\Pcorr = \left\{ \EEstrong_{\T}(\eps)\subset \EE_{\Tc} \right\}\,.
\end{equation}

Finally,
define the event
\begin{equation}\label{e:Ncorrdef}
	 \Ncorr=\left\{\LL^{(2)}(P,\Pi_{\T}(\Ph))\leq \epsilon\right\}\,.
\end{equation}
Recall that $P$ factorizes according to $\T$ and $\Ph$ is the empirical distribution which does not factorize according to any tree. This event controls \emph{cascades} of errors in correlations computed along paths in $\T$.

 Since we are interested in the situation that all three events hold, let $\mathtt{E}(\eps,\gamma):=\Ecorr\cap\Pcorr\cap \Ncorrg$.
 \rvw{ Lemmas~\ref{l:EcorrLemma},~\ref{l:PcorrLemma} and~\ref{l:NcorrLemma} 
  prove that for
 \begin{equation}\label{eq:defepsgamma}
\epsilon_0=\min\{e^{-\beta}/24, \eta/16\}
 \quad \text{and}\quad \gamma_0= \eta/3\,,
 \end{equation} 
 if
 $
 n>\max\{1152 e^{2\beta},512\eta^{-2}\}\log(6p^3/\delta):=n_0\,,
 $
 then}
 \begin{equation}\label{e:probBoundE}
\rvw{\P[\mathtt{E}(\eps_0,\gamma_0)] \geq 1-\delta\,.}
\end{equation}

\paragraph{Sketch of proof of  Proposition~\ref{p:EdgeEstError}}
The proof entails showing that on the event $\mathtt{E}(\eps_0,\gamma_0)$ for $\eps_0$ and $\gamma_0$ defined in~\eqref{eq:defepsgamma}
we have the desired inequality $\LL^{(2)}(\Pi_{\Tc}(P),\Pi_{\Tc}(\Ph))\leq\eta$.  The result then follows from \eqref{e:probBoundE}.

To bound $\LL^{(2)}(\Pi_{\Tc}(P),\Pi_{\Tc}(\Ph))$ on event $\mathtt{E}(\eps_0,\gamma_0)$, we consider parameter estimation errors along paths in $\Tc$. First, observe that on event $\Ncorrgn$ defined in~\eqref{e:Ncorrdef}, the end-to-end error for each path in $\EE_{\Tc}\cap\EE_{\T}$ is bounded by $\gamma_0$. 
Next, we study parameter estimation error in paths containing (falsely added) edges in $\EE_{\Tc}\setminus\EE_{\T}$. 

For any pair of nodes $w,\wt$,  denote by $t = |\path_{\Tc}(w,\wt)\setminus \EE_{\T}|$ the number of falsely added edges in the path connecting them in $\Tc$. As discussed in the proof, these edges correspond to missed edges in $\T$ and thus $\Pcorrn$ guarantees that these edges are weak (as defined after~\eqref{e:Estrong}). These $t$ weak edges break up the $\path_{\Tc}(w,\wt)$ into at most $t+1$ contiguous segments $\F_0,\F_1,\cdots,\F_t$, each entirely within $\EE_{\Tc}\cap\EE_{\T}$. 

On event $\Ncorrgn$ the error on each segment $\F_i$, is bounded by $\gamma_0$, but now there are $t+1$ such segments and the errors may add up. This effect is counterbalanced by the fact that the falsely added edges are weak and hence scale down the error in a multiplicative fashion. 


\paragraph{Sketch of proof of  Proposition~\ref{p:projectedDist}} We want to show that $ \LL^{(2)}(P,\Pi_{\Tc}(P))\leq\eta$ on the event $\Ecorrn\cap\Pcorrn\supseteq \mathtt{E}(\eps_0,\gamma_0)$ for $\epsilon_0$ defined in~\eqref{eq:defepsgamma}. 


The proof of the proposition sets up a careful induction on the distance between nodes (computed in $\Tc$) for which we wish to bound the error in correlation. One of the ingredients is Lemma~\ref{l:edgePair}, a combinatorial statement relating trees $\T$ and $\Tc$.
The lemma states that for any two spanning trees on $p$ nodes, and for any two nodes $w,\wt\in [p]$, there exists at least one pair of edges $f\in\path_{\T}(w,\wt)$ and $g\in\path_{\Tc}(w,\wt)$ satisfying a collection of properties illustrated in Figure~\ref{f:proof} (and specified in the lemma).  

 A consequence is that the true correlation across $g$ according to $P$ can be expressed in terms of the correlation on $f$ as $\mu_{g}=\mu_f \mu_{\A}\mu_{\C}\mu_{\At}\mu_{\Ct}$, hence $|\mu_{g}|\leq |\mu_f|$.  
 But $|\wmu_{g}|\geq |\wmu_f|$, since the Chow-Liu algorithm chose $g$ in $\Tc$ instead of $f$. 
Tight control on the relationship between $|\wmu_{g}|$ and $|\wmu_f|$ yields a recurrence for $\Delta(d)$, where $\Delta(d)$  is an upper bound on the error due to graph estimation error for any pair of nodes $w,\wt$ with $|\path_{\Tc}(w,\wt)|=d$.

\section{Proof of main result}
\label{s:ProofProp}
As observed in Section \ref{s:ProofSketch}, \rvw{Theorem \ref{t:CLLowerBound}} is a direct consequence of Propositions \ref{p:EdgeEstError} and \ref{p:projectedDist}, which we prove in Sections \ref{s:paramErrpr} and \ref{s:incorrectGraph}. We prove Theorem~\ref{t:inferenceUpper} in Section~\ref{s:pairwiseUpper}.

\subsection{Loss due to parameter estimation (proof of Proposition~\ref{p:EdgeEstError})}
\label{s:paramErrpr}
We will prove that on the event $\mathtt{E}(\eps_0,\gamma_0)$ for $\eps_0$ and $\gamma_0$ defined in~\eqref{eq:defepsgamma}
the desired inequality $\LL^{(2)}(\Pi_{\Tc}(P),\Pi_{\Tc}(\Ph))\leq\eta$ holds. 
Equation~\eqref{e:probBoundE} gives the result. 
 
Let $\tau(\eps_{0})$ (defined in~\eqref{e:Estrong}) be the threshold to define $\EEstrong_{\T}(\epsilon_0)$, the set of strong edges in $\T$. For any pair of nodes $w,\wt$, consider $\path_{\Tc}(w,\wt)$. Let $0\leq t <p$ be the number of weak edges $e_1,\cdots,e_t\in\path_{\Tc}(w,\wt)$ such that $|\mu_{e_i}|\leq\tau(\eps_0)$.
There are at most $t+1$ contiguous sub-paths in $\path_{\Tc}(w,\wt)$ consisting of strong edges. We call these segments $\F_0,\F_1,\dots, \F_t$. If two weak edges $e_i$ and $e_{i+1}$ are adjacent in $\path_{\Tc}(w,\wt),$ then $\F_i=\varnothing$, in which case we define $\mu_{\F_{i}}=\wmu_{\F_i}=1$. By definition of $\F_i$, all edges $f\in\F_i$ are strong.

According to~\eqref{e:Pcorr}, under the event $\Pcorrn$ all strong edges in $\T$ are recovered in $\Tc$. Thus, $\F_i \subseteq \EE_{\T}$ is a path not only in   $\Tc$ but also in $\T$, which guarantees  $|\wmu_{\F_i}-\mu_{\F_i}|\leq \gamma_0$ under the event $\Ncorrgn$.
 
Note that if $t=0$ then $\path_{\Tc}(w,\wt)$ consists of all strong edges for which Lemma \ref{l:NcorrLemma} gives the desired bound.
For $t\geq 1$ we have:
\begin{eqnarray*}
\lefteqn{\left|\prod_{e\in \path_{\Tc}(w,\wt)} \wmu_e- \prod_{e\in \path_\Tc(w,\wt)} \mu_e\right| }
\\
&& 
\stackrel{(a)}{=}
 \left| \wmu_{\F_{0}}\prod_{i=1}^t \wmu_{\F_i} \wmu_{e_i}-\mu_{\F_{0}}\prod_{i=1}^t \mu_{\F_i} \mu_{e_i}\right|
 \\
&& 
\stackrel{(b)}{\leq} |\wmu_{\F_{0}}-\mu_{\F_{0}}|\prod_{j=1}^{t}|\mu_{\F_j} \mu_{e_j}|
 + 
\sum_{i=1}^{t}|\wmu_{\F_i} \wmu_{e_i}-\mu_{\F_i} \mu_{e_i}|\cdot|\wmu_{\F_0}|\prod_{j=1}^{i-1}|\wmu_{\F_j} \widehat{\mu}_{e_j}|
\prod_{k=i+1}^{t}|\mu_{\F_k} \mu_{e_k}| 
\\
&&
\stackrel{(c)}{\leq}
 \gamma_0 \big[\tau(\eps_0)\big]^t +
\big(\tau(\eps_0)+\epsilon_0\big)^{t-1}
 \sum_{i=1}^{t}|\wmu_{\F_i} \wmu_{e_i}-\mu_{\F_i} \mu_{e_i}|
\\
&&
\stackrel{(d)}{\leq}
 \gamma_0 \big[\tau(\eps_0)\big]^t +
\big(\tau(\eps_0)+\epsilon_0\big)^{t-1}
  \left[\sum_{i=1}^{t}\left| \mu_{\F_i} (\wmu_{e_i}-\mu_{e_i})|
  +|\wmu_{e_i}(\wmu_{\F_i} -\mu_{\F_i})\right|\right]
\\
&&
\stackrel{(e)}{\leq} 
\big(\tau(\eps_0)+\epsilon_0\big)^{t-1}
 (2t+1) \max\{\gamma_0,\epsilon_0\}
\stackrel{(f)}{\leq} \,\frac{2t+1}{4^{t-1}} \,\frac\eta 3
\,
 \stackrel{(g)}{\leq}  
 \eta\,.
\end{eqnarray*}
In (a) we use $\path_{\Tc}(w,\wt) = \{\F_0,e_{1},\cdots,\F_t,e_t, \F_{t}\}$. 
(b) uses the bound $|\prod_{i=1}^t a_i - \prod_{i=1}^t b_i|\leq \sum_{i=1}^{t}|a_i-b_i|\prod_{j=1}^{i-1}|a_j| \prod_{k=i+1}^t |b_k| $ obtained via telescoping sum and triangle inequality.
In (c), we use $|\mu_{\F_i}|, |\wmu_{\F_i}|\leq 1$, $|\mu_{e_i}|\leq \tau(\eps_0)$, $|\wmu_{e_i}|\leq \tau(\eps_0)+\epsilon_0$ on $\Ecorrn$ and $|\wmu_{\F_{0}}-\mu_{\F_0}|\leq \gamma_0$ on $\Ncorrgn$. (d) uses triangle inequality. In (e), we use $|\wmu_{\F_{i}}-\mu_{\F_i}|\leq \gamma_0$ on the event $\Ncorrgn$ and $|\wmu_{e_{i}}-\mu_{e_i}|\leq \epsilon_0$ on the event $\Ecorrn$. 
(f) is true since $\gamma_0, \epsilon_0\leq \eta/3$ (as in~\eqref{eq:defepsgamma}). Also,  $1-\tanh\beta\geq e^{-2\beta}$ and the definition of $\tau(\eps_0)$ in~\eqref{e:Estrong} gives $\tau(\eps_0)\leq 4\eps_0 e^{\beta}$. Hence, $\tau(\eps_0)+\epsilon_0\leq 5\eps_0 e^{\beta}\leq 1/4$  where the last inequality uses $\eps_0 \leq e^{-\beta}/20$ (according to~\eqref{eq:defepsgamma}). (g) holds for all $t\geq 1$.  \qed

\subsection{Loss due to graph estimation error (proof of Proposition~\ref{p:projectedDist})}
\label{s:incorrectGraph}

The following lemma, proved in Section~\ref{sec:lemmaProofs} for completeness, is a well-known consequence of a  spanning tree being max-weight~\cite{cormen2009introduction}. 
We use Lemma~\ref{l:greed} to bound the loss due to graph estimation error  by the Chow-Liu algorithm.
\begin{lemma}[Error characterization in the Chow-Liu tree] \label{l:greed}
Consider the complete graph on $p$ nodes with weights $|\wmu_{ij}|$ on each edge $(i,j)$. Let $\Tc$ be the maximum weight spanning tree of this graph. If edge $(u,\ut)\notin \EE_{\Tc}$, then $|\wmu_{u\ut}|\leq |\wmu_{ij}|$ for all $(i,j)\in\path_{\Tc}(u,\ut)\,.$
\end{lemma}

Recall that $P$ factorizes according to $\T$. 
The error in correlation between any two variables $X_w$ and $X_{\wt}$ computed along the $\path_{\T}(w,\wt)$ as compared to  $\path_{\Tc}(w,\wt)$ is
\begin{align*}
\err_{P,\Tc}(w,\wt) 
& =
 \frac{1}{2}\cdot \left|\Ex_P X_w X_{\wt} - \Ex_{\Pi_{\Tc}(P)}X_w X_{\wt}\right|
 \\
& = 
 \frac{1}{2}\cdot  \bigg|\prod_{e\in\path_{\T}(w,\wt)}\mu_e - \prod_{e\in\path_{\Tc}(w,\wt)}\mu_e \bigg| 
 \,.
\end{align*}
Our goal is to bound
$
\LL^{(2)}(P,\Pi_{\Tc}(P)) 
=
\max_{w,{\wt}\in\VV}\err_{P,\Tc}(w,\wt)\,.$
We will prove that 
on the event 
$\Ecorrn\cap\Pcorrn\supseteq \mathtt{E}(\eps_0,\gamma_0)$
 for 
 $\epsilon_0$
  defined in~\eqref{eq:defepsgamma}, 
  $\LL^{(2)}(P,\Pi_{\Tc}(P))<\eta$ 
  holds. 
The result then follows from~\eqref{e:probBoundE}.

The core of the argument uses induction to derive a recurrence on the maximum of $\err_{P,\Tc}(w,\wt)$ in terms of the distance (as measured in $\Tc$) between the nodes $w,{\wt}$. 
Define 
$$
\Delta(d)\triangleq\max_{\substack{w,{\wt}\in\VV\\|\path_{\Tc}(w,{\wt})|=d }}\err_{P,\Tc}(w,\wt) \,.
$$
 For nodes at distance one in $\Tc$, i.e. $|\path_{\Tc}(w,{\wt})|=1$, it follows that 
 $\err_{P,\Tc}(w,\wt)=0$ from the definition of the projected distribution $\Pi_{\Tc}(P)$ (matching pairwise marginals on edges) in Section~\ref{sec:infprojdef}. 
 Hence, $\Delta(1) = 0 \leq \eta$. 
 We define $\Delta(0)=0$. 
 For $d>1$, we bound $\Delta(d)$ in terms of $\Delta(k)$ for $k<d$:
  we will show that  on the event $\Ecorrn\cap\Pcorrn$ 
  with 
  $\epsilon_0$ defined in~\eqref{eq:defepsgamma},
 if 
 $\Delta(k) \leq \eta$ 
 for all 
 $k<d$, then 
 $\Delta(d)\leq \eta$ 
 which gives the result.
  
	Note that if $\path_{\Tc}(w,\wt)=\path_\T(w,\wt)$, then $\err_{P,\Tc}(w,\wt)=0$ (again because correlations are matched on edges). 
Thus, we assume $\path_{\Tc}(w,{\wt})\neq \path_{\T}(w,{\wt})$. Lemma \ref{l:edgePair} shows the existence of  a pair of edges $f=(u,\ut)\in\EE_{\T}\setminus\EE_{\Tc}$ and $g=(v,\vt) \in \EE_{\Tc}\setminus\EE_{\T}$ such that (See Figure \ref{f:proof}):
\begin{itemize}
\item $f\in\path_{\T}(w,\wt) \cap \path_{\T}(v,\vt)$ and $g\in\path_{\Tc}(w,\wt) \cap \path_{\Tc}(u,\ut)$.
\item $f\notin\path_{\Tc}(w,\wt)$ and $g\notin\path_{\T}(w,\wt)$.
\item $u,v\in \subT{\T}{f}{w}$ and $\ut,\vt\in \subT{\T}{f}{\wt}$.
\end{itemize}
Here $\subT{\T}{f}{w}=\{i\in\VV; f\notin\path_{\T}(w,i)\}$ is the set of nodes connected to $w$ in $\T$ after removing edge $f$ (see Figure~\ref{f:proof}).
	\begin{figure}[t] 
		\begin{center}
		\includegraphics[width=3.5in]{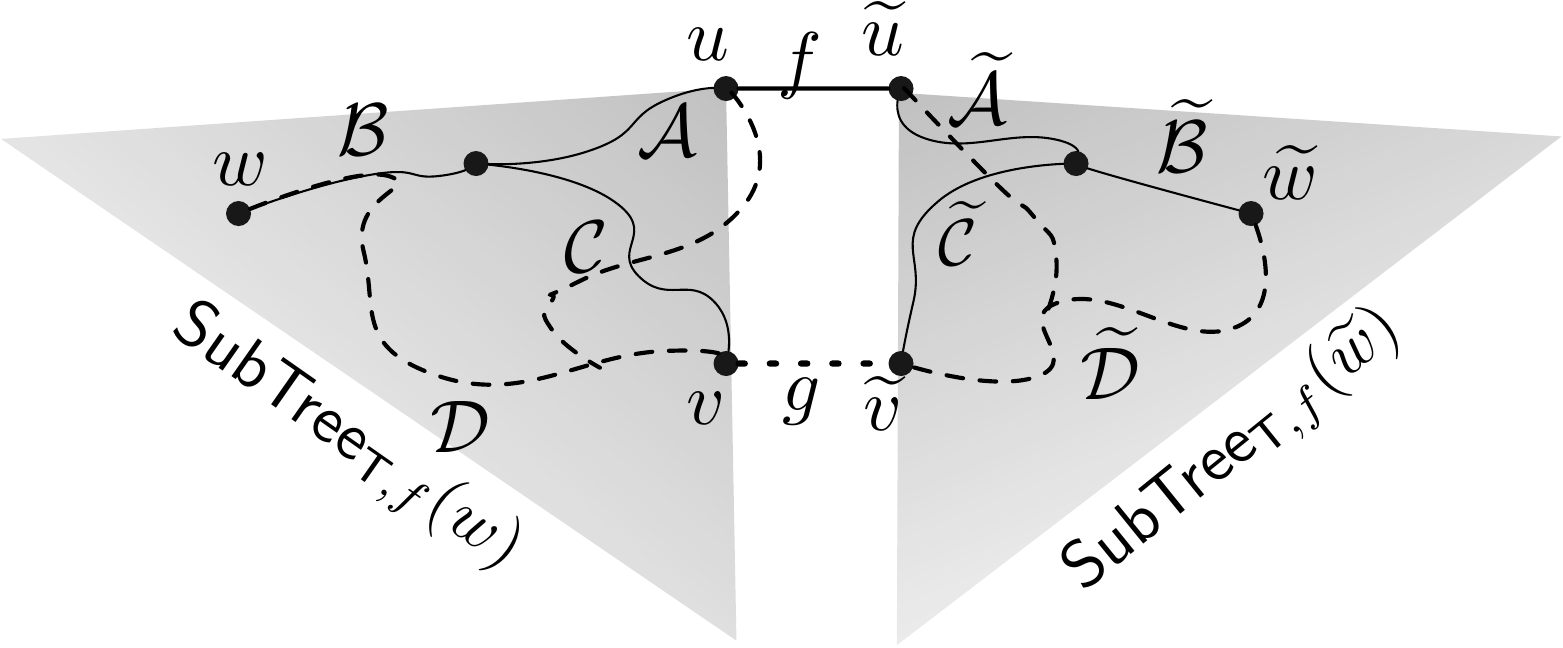}
		\caption{Schematic for the proof of Proposition~\ref{p:projectedDist}. The solid lines represent paths in $\T$ and the dashed lines represent paths in $\Tc$. 
		The sets of edges $\D$ and $\Dt$ may overlap with $\A\cup\B$ and $\At\cup\Bt$.}\label{f:proof}
		\end{center}
	\end{figure}
We define several sub-paths:
	\begin{align*}
	\A  = \path_{\T}(u,w) \cap \path_{\T}(u,v)\,, \qquad&\qquad
	\B  = \path_{\T}(u,w) \setminus \path_{\T}(u,v)\,,  \\
	\C  = \path_{\T}(u,v) \setminus \path_{\T}(u,w)\,,\qquad& \qquad
	\D  = \path_{\Tc}(w,v)\,.
	\end{align*}
	Recall that for set of edges $\SS$, we defined $\mu_{\SS}=\prod_{e\in\SS}{\mu_e}$. 
	Since $\path_{\T}(w,v) = \B\cup \C$ and $\B\cap\C=\varnothing$, we have $\mu_{w,v}=\mu_{\B} \mu_{\C}$.
Similarly, in $\subT{\T}{f}{\wt}$ we define
	\begin{align*}
	\At  = \path_{\T}(\ut,\wt) \cap \path_{\T}(\ut,\vt)\,, \qquad &\qquad
	\Bt  = \path_{\T}(\ut,\wt) \setminus \path_{\T}(\ut,\vt)\,,\\
	\Ct  = \path_{\T}(\ut,\vt) \setminus \path_{\T}(\ut,\wt)\,, \qquad & \qquad
	\Dt  = \path_{\Tc}(\wt,\vt)\,.
	\end{align*}
	The sets are defined so that 
	$\path_{\T}(v,\vt)= \C \cup \A \cup \{f\} \cup \Ct \cup \At$ 
	 where $f=(u,\ut)$ and 
	 $g=(v,\vt)\in\EE_{\Tc}$. 
	 Thus, $\mu_g = \mu_f \mu_{\A} \mu_{\C} \mu_{\At} \mu_{\Ct}$. 
	 Since $\path_{\T}(w,{\wt})= \A \cup \B \cup \{f\} \cup \At \cup \Bt$ and $\path_{\Tc}(w,{\wt})= \D \cup \{g\} \cup \Dt$, our goal amounts to finding an upper bound for the quantity 
	 $ |\mu_{\D} \mu_{g} \mu_{\Dt}-\mu_{\A}\mu_{\B}\mu_{f}\mu_{\At}\mu_{\Bt}|$.

Lemma \ref{l:greed} applied to $f=(u,\ut)\notin\EE_{\Tc}$ and 
$g=(v,\vt)\in\path_{\Tc}(u,\ut)$ 
gives $|\wmu_{f}|\leq |\wmu_g|$.
 Also, $f\in\path_{\T}(v,\vt),$ hence $|\mu_{g}|\leq |\mu_{f}|$. 
 On the event $\Ecorrn$,
  $|\mu_f|-2\epsilon_0 \leq |\wmu_{f}|\leq |\wmu_{g}|\leq |\mu_g|+2\epsilon_0\leq |\mu_f|+2\epsilon_0$ which gives $|\mu_f|-4\epsilon_0 \leq | \mu_f \mu_{\A} \mu_{\C} \mu_{\At} \mu_{\Ct}| \leq |\mu_f|$. 
  \rvw{ Also, $|\mu_{\A} \mu_{\C} \mu_{\At} \mu_{\Ct}|\leq |\mu_{\C} \mu_{\Ct}|\leq 1$.} Thus,
\begin{align}
\label{eq:edgeRecoveryError}
|\mu_f| \left(1- \mu^2_{\C} \mu^2_{\Ct}\right) &\leq 2|\mu_f| \left(1-| \mu_{\C} \mu_{\Ct}|\right) \leq 2|\mu_f| \left(1-| \mu_{\A} \mu_{\C} \mu_{\At} \mu_{\Ct}|\right) \leq 8 \epsilon_0\,.
\end{align}
Since $f\in\EE_{\T}\setminus \EE_{\Tc}$, under the event $\Pcorrn$,  $f$ cannot be a strong edge as defined in~\eqref{e:Estrong}. 
It follows that $|\mu_f|\leq \tau(\eps_0)$ for $\tau(\eps_0)$ defined in~\eqref{e:Estrong}.

	Let $k=|\D|=|\path_{\Tc}(w,v)| $ and $\widetilde{k}=|\Dt|=|\path_{\Tc}(\wt,\vt)| $, so $d=k+\widetilde{k}+1$. 
By definition of $\Delta(\cdot )$, 
	\begin{eqnarray} \label{e:tempIneq}
	\err_{P,\Tc}(w,v) & = &|\mu_{\D} - \mu_{\B} \mu_{\C}|\leq \Delta(k)\,, \\
	\err_{P,\Tc}(\wt,\vt) & = & |\mu_{\Dt} - \mu_{\Bt} \mu_{\Ct}| \leq \Delta(\widetilde{k})\,.\notag
	\end{eqnarray}
	We now prove  $\err_{P,\Tc}(w,\wt)\leq \eta$ assuming inductively  $\err_{P,\Tc}(i,j)\leq \eta$ for all pairs $i,j$ such that $\mathrm{dist}_{\Tc}(i,j)+1\leq d = \mathrm{dist}_{\Tc}(w,\wt)$ (where $\mathrm{dist}_{\Tc}(w,\wt)$ denotes the graph distance in $\Tc$).
	 Using $\mu_g =   \mu_{\A} \mu_{\C}\mu_f \mu_{\At} \mu_{\Ct}$,
	\begin{align*}
	\err_{P,\Tc}(w,\wt) 
	&= |\mu_{\D} \mu_{g} \mu_{\Dt}-\mu_{\A}\mu_{\B}\mu_{f}\mu_{\At}\mu_{\Bt}|
	 \\
	&
	 =
	 |\mu_{\D}  \mu_{\A} \mu_{\C}\mu_f \mu_{\At} \mu_{\Ct}
	\mu_{\Dt}-\mu_{\A}\mu_{\B}\mu_{f}\mu_{\At}\mu_{\Bt}|  
	\\
	&
	=
	  |\mu_{\A}\mu_f \mu_{\At} | 
\cdot
 |\mu_{\C}  \mu_{\Ct} (\mu_{\D}-\mu_{\B} \mu_{\C}+\mu_{\B} \mu_{\C})
	( \mu_{\Dt}- \mu_{\Bt} \mu_{\Ct}+ \mu_{\Bt} \mu_{\Ct}) -\mu_{\B}\mu_{\Bt}|
	\\
	&
	 \leq 
	 |\mu_{\A}\mu_f \mu_{\At}|
	 \cdot 
	 \left[|\mu_{\C}  \mu_{\Ct} \mu_{\B} \mu_{\C}
	\mu_{\Bt} \mu_{\Ct}-\mu_{\B}\mu_{\Bt}|
	+|\mu_{\C}  \mu_{\Ct} ( \mu_{\D}- \mu_{\B} \mu_{\C})( \mu_{\D}- \mu_{\B} \mu_{\C})| 
\right. 
\\
	&
	 \qquad\qquad\qquad
	+ 
	\left.|\mu_{\C}  \mu_{\Ct} \mu_{\Bt} \mu_{\Ct}( \mu_{\D}- \mu_{\B} \mu_{\C})|  
	+	 |\mu_{\C}  \mu_{\Ct} \mu_{\B} \mu_{\C}( \mu_{\Dt}- \mu_{\Bt} \mu_{\Ct})|   \right]
	\\
	&
	\stackrel{(a)}{\leq} 
	|\mu_f \mu_{\A} \mu_{\At} \mu_{\B} \mu_{\Bt}||\mu^2_{\C}\mu^2_{\Ct}-1|
  +
   |\mu_f|\left(\Delta(k)\Delta(\widetilde{k})+\Delta(k)+\Delta(\widetilde{k})\right)
   \\
	&
	\stackrel{(b)}{\leq} 
	8\epsilon_0 
	+
	\tau(\eps_0) \left(\Delta(k)+\Delta(\widetilde{k})+\Delta(k)\Delta(\widetilde{k})\right)
	\stackrel{(c)}{\leq}
	 8\eps_0
	  + 4\eps_0 e^{\beta} (2\eta + \eta^2) \leq \eta  \,.
	\end{align*}
Inequality (a) follows from~\eqref{e:tempIneq}. 
(b) uses Equation~\eqref{eq:edgeRecoveryError} and $|\mu_f|\leq \tau(\eps_0)\leq 4\eps_0 e^{\beta}$. 
We showed that $\Delta(1) = 0$. 
In (c) we use the inductive assumption $\Delta(k)\leq \eta$ for all $k<d$ and the assumption $\epsilon_0$ defined in~\eqref{eq:defepsgamma}.
 Since $w$ and $\wt$ were arbitrary, this proves $\Delta(d) \leq \eta$, and moreover this holds for all $d$. \qed

\subsection{
Necessary samples for accurate pairwise marginals (proof of Theorem \ref{t:inferenceUpper})}
\label{s:pairwiseUpper}

We construct a family of trees that are difficult to distinguish from one another.
Applying the version of Fano's inequality below in Lemma~\ref{e:Fano}, gives a lower bound on the error probability. The bound on the sample complexity is in terms of the KL-divergence between pairs of points in the parameter space. 
The symmetrized KL-divergence between two zero-field Ising models with parameters $\theta$ and $\theta'$ has the convenient form
\begin{equation}
\label{e:KLIsing}
J(\theta\|\theta') \triangleq D(\theta\|\theta') + D(\theta'\|\theta) = \sum_{i<j} (\theta_{ij} - \theta'_{ij})(\mu_{ij} - \mu'_{ij})\,.
\end{equation}
Here $\mu_{ij}$ and $\mu'_{ij}$ are the pairwise correlations between nodes $i$ and $j$ computed according to $\theta$ and $\theta'$, respectively.

\begin{lemma}[Fano's inequality, \rvw{Corollary 2.6 in  \cite{tsybakov2008introduction}}]\label{e:Fano}
	Assume that $M\geq 2$ and that $\Theta$ is a family of models $\theta^{0}, \theta^{1}, \dots, \theta^{M}$. Let $Q_{\theta^j}$ denote the probability law of the observation $X$ under model $\theta^j$. Let $\Phi: \{-1,+1\}^{p\times n}\to \{0,1,\cdots,M\}$ denote an estimator using 
 $n$ i.i.d. samples $X^{(1:n)}$.  If 
	\begin{equation}
	\label{e:KLbound}
n<(1-\delta) \frac{\log M}{\frac1{M+1}\sum_{j=1}^M J(Q_{\theta^{j}}\|Q_{\theta^{0}})}\,,
	\end{equation} 
	then the probability of error of any algorithm  is bounded as:
	\[\underset{{\Phi}}\inf\,\max_{0\leq j \leq M} Q_{\theta^j}\big[\Phi\big(X^{(1:n)}\big)\neq j\big]\geq \delta - \frac{1}{\log M}\,.\]
	\end{lemma}
	
\rvw{The following corollary is a restatement of Equation (2.9) in \cite{tsybakov2008introduction} using the above lemma and tailored to our setup.

\begin{corollary}\label{cor:Fano-loss}
Let $d$ be a distance (i.e., a metric). Suppose there are $M$ different parameter vectors $\theta^0,\cdots,\theta^{M}$ such that $d({\theta^k},\theta^j)\geq 2\eta$ for all $j\neq k$. If $n$ satisfies~\eqref{e:KLbound}, then any estimator $\Psi:X^{(1:n)}\to \Theta$ mapping $n$ samples to a set of parameters associated with a tree $\T$ and a distribution on $P\in\mathcal{P}_{\T}$ incurs a loss greater than $\eta$ with probability at least $1/2$:
\[\inf_{\Psi}\sup_{\T\in \TT\atop P\in\PP_{\T}(\al,\beta)} \P\Big[d\big(\Psi(X^{\tiny(1:n\tiny)}),\theta\big)\geq \eta\Big]\geq \frac{1}{2}\,.\]
\end{corollary}
}

\begin{proof}[Proof of Theorem~\ref{t:inferenceUpper}]
\rvw{We consider a fixed tree structure given by a path, or in other words a Markov chain $X_1- X_2- \cdots -X_p$. We choose $M$ different parameter vectors $\theta^{m}$, $0\leq m \leq M-1$, for $M=p$. 

Let $\theta^{0}_{i,i+1}=\al$ for $i = 1,\cdots,p-1$. In the $m$-th model we have $\theta^{m}_{m,m+1}= \mathrm{atanh}(\tanh \al + 2\eta)$ and the remaining edge weights $\theta^{m}_{i,i+1}=\al$ for $i\neq m$. For $m'\neq m\,,$
\begin{align*}
\max_{i,j} \big|\Ex_{\theta^{m'}}[X_i X_j]-\Ex_{\theta^{m}}[X_i X_j]\big|  \geq 2\eta\,.
\end{align*}
Also, using~\eqref{e:KLIsing} \[J(\theta^{m} \| \theta^{m'}) \leq 2\eta\big[\mathrm{atanh}(\tanh \al + 2\eta)-\al\big] \leq 2\eta  \,\frac{2\eta}{1-\big[\tanh\al+2\eta\big]^2}, \]
where we used $\frac{d}{dx}\mathrm{atanh}(x)=\frac{1}{1-x^2}$ to get the last inequality. 
Fano's inequality  (Corollary \ref{cor:Fano-loss}) with the distance $\LL^{(2)}$ gives the bound. }
\end{proof}

\section{Sample complexities of structure learning and forest approximation}
\label{s:ExactRecovery}
\subsection{Samples necessary for structure learning}
\label{s:structLowerPf}
\begin{proof}
[proof of Theorem~\ref{t:structLower}]
Suppose that $p$ is odd (for simplicity) and let the graph $\T_0$ be a path with associated parameters $\theta^{0}$ given by $\theta^{0}_{i,i+1}=\alpha$ for odd values of $i$ and $\theta^{0}_{i,i+1}=\beta$ for even values of $i$. For each odd value of $m\leq p-2$ we let $\theta^{m}$ be equal to $\theta^{0}$ everywhere except $\theta^{m}_{m,m+1}=0$ and \rvw{$\theta^{m}_{m,m+2}=\mathrm{atanh}(\tanh(\al) \tanh(\beta))$.} There are $(p+1)/2$ models in total (including $\theta^0$). A small calculation using~\eqref{e:KLIsing} leads to
\rvw{\begin{align*}
J(\theta^m\|\theta^0)& = \al \tanh \al \big[1 - (\tanh\beta)^2\big] 
\leq 4 \al^2 e^{-2\beta}\,.
\end{align*}}
Here we used $\tanh \alpha \leq \alpha$ and $1-(\tanh \beta)^2\leq 4 e^{-2\beta}$ for $\alpha,\beta \geq 0$. Plugging the last display into Fano's inequality (Lemma~\ref{e:Fano}) completes the proof.\qedhere
\end{proof}

\subsection{Samples sufficient for structure learning (proof of Theorem~\ref{t:structUpper})}
\label{s:structUpperPf}

Consider the original tree $\T$ with parameters $\alpha\leq |\theta_{ij}|\leq\beta$ for $(i,j)\in\EE_{\T}$.
Using Definition~\ref{e:Pcorr}, the Chow-Liu algorithm recovers strong edges on the event $\Pcorr$, where edge $(i,j)$ is strong if its parameter $\theta_{ij}$ satisfies $|\tanh\theta_{ij} >\tau$. Thus, if the edge strength lower bound $\alpha$ in the original tree $\T$ satisfies $\tanh\alpha > \tau$, on event $\Pcorr$ we have $\Tc=\T$. 
Note that by Lemma~\ref{l:PcorrLemma}, the event $\Pcorr$ (defined in~\eqref{e:Pcorr}) with $\epsilon = \sqrt{2/n \log(2p^2/\delta)}$ occurs with probability at least $1-\delta$.
The bound \[n>\frac{16}{\tanh^2(\alpha)(1-\tanh \beta)}\log{\frac{2p^2}{\delta}}\] on the number of samples guarantees $\tanh\alpha > \tau$ and $\Tc=\T$ with probability at least $1-\delta$.
Using ${1-\tanh \beta}\geq e^{-2\beta}$  gives the statement of Theorem~\ref{t:structUpper}. \qed

\subsection{Lower bounding sample complexity of forest approximation algorithm, Proof of Proposition~\ref{pro:forestLB}}\label{sec:forestLBproof}
\rvw{
Since the loss function $\widetilde{d}(P,Q)$ defined in~\eqref{eq:forest-loss} is not symmetric and hence not a distance, one cannot use Corollary~\ref{cor:Fano-loss} to lower bound the sample complexity of the forest approximation algorithm.

Define a class of $M=p-1$ models $\Theta$ as follows: Let $\theta^1$ be a model in which $\theta^1_{12}=0$, $\theta^1_{2i}=\beta$ for $3\leq i \leq p$, and $\theta^1_{ij}=0$ for all other edges so that $\mu_{1i}=0$ for all $i\geq 2$ in this model. 
For  $2\leq m\leq M$, define $\theta^m$ ($m$-th model) such that $\theta^m_{12}=0$, $\theta^m_{2i}=\beta$ for $3\leq i \leq p$ and $\theta^m_{1,m+1}=\mathrm{atanh}(\eta)$.     
Using~\eqref{e:KLIsing}, 
since $\mu_{1,m+1}=\eta$ and $\mu_{1i}=\eta \tanh^2(\beta)$ for $i\geq 3$ in model $m\geq 2$, 
\begin{align}\label{e:forestKL}
\frac{1}{M}\sum_{m=1}^M J(\theta^m\|\theta^{M})
&\leq 
	\eta\, \mathrm{atanh}(\eta)\Big[ \frac{1}{M}\,
		+ \Big(1-\frac{1}{M}\Big)  8e^{-2\beta}\Big] \leq 
	16\,\eta \,\mathrm{atanh}(\eta) \max\Big\{\frac{1}{p},\, e^{-2\beta}\Big\}\,.
\end{align} 

 It will be useful to decompose an estimator into an estimator for the neighborhood of node 1, and an estimator for the rest of the model with neighborhood of node 1 fixed. 
 Let $\Phi:\{-1,+1\}^{p\times n}\to \{0,1\}^{\VV\setminus\{1\}}$ be an estimator for the neighborhood of node 1, and let $\Psi'_{\Phi}$
 be an estimator that generates a forest $\mathsf{F}$, constrained to have the same neighborhood for node $1$ that the estimator $\Phi$ recovers, and a distribution $Q\in\PP_{\mathsf F}(0,\beta)$. Any estimator $\Psi$ can be decomposed in this way to $\Phi$ and $\Psi'_{\Phi}$.
  It follows that
 \begin{align*}
\inf_{\Psi} \sup_{\T\in \TT\atop P\in\PP_{\T}(\al,\beta)} \P\left[\widetilde{d}\big(P,\Psi(X^{\small(1:n\small)})\big) \geq \eta\right]  
&
 \stackrel{(a)}\geq
  \inf_{\Phi}\inf_{\Psi'_{\Phi}} \max_{m}\P_{\theta^m}\left[\widetilde{d}\big(P,\Psi(X^{\tiny(1:n\tiny)})\big) \geq \eta\right]\\
  & 
  \stackrel{(b)}\geq    \inf_{\Phi} \max_{m}
\,  \inf_{\Psi'_{\Phi}}\,\,
\P_{\theta^m}\left[\widetilde{d}\big(P,\Psi(X^{\small(1:n\small)})\big) \geq \eta\right]\\
 & 
 \stackrel{(c)}\geq    \inf_{\Phi} \max_{m}\,\,
 \P_{\theta^m}\left[\Phi(X^{\tiny{(1:n)}}) \neq \partial_{\theta^m}(1) \right]\,.
 \end{align*}
 (a) holds since $\big\{P_{\theta^m}\big\}_{m=1}^M \subseteq \bigcup_{\T\in\TT}\PP_{\T}(0,\beta)$.
  (b) holds since we swapped the order of infimum and max operations. 
  (c) is justified as follows: If the data is generated from $P_\theta$, $\theta\in\Theta$, incorrect recovery of the neighborhood of node $1$ by the estimator implies that $\widetilde{d}\geq\eta$.
   This is because $\widetilde{d}=1$ if any extra edges are added and otherwise, if there are no extra edges and an edge incident to node~1 is missing, then $\widetilde{d}\geq \eta$ due to the true correlation $\eta$ on the edge being zero in the estimated model.
 
 Now, the final quantity in the last display can be lower bounded by a standard application of Fano's inequality (Corollary~\ref{cor:Fano-loss}) on the family $\Theta$ introduced in the beginning of the proof. The main ingredient is the average symmetric KL from~\eqref{e:forestKL} and specifying the distance. Let $\partial_{\mathsf{F}}(1)$ be the neighborhood of node~1 in forest $\mathsf{F}$. For $P$ defined on forest $\mathsf F$ and $Q$ defined on $\mathsf{F}'$, we use the distance $\widetilde{d}'(P,Q)={1\!\!1} [\partial_{\mathsf{F}}(1)\neq \partial_{\mathsf{F}'}(1)]$ (zero-one loss on neighborhood of node~1).
}

\label{s:detailsProof}
\section{Control of events $\mathtt{E}^\mathrm{corr},\mathtt{E}^\mathrm{strong}$ and $\mathtt{E}^\mathrm{cascade}$}\label{sec:lemmaProofs}
We state a standard form of Hoeffding's inequality \cite{hoeffding1963probability}
in Appendix~\ref{app:devBD} and use it here.
\begin{lemma}\label{l:EcorrLemma}
The event $\Ecorr$ defined in~\eqref{e:Ecorrdef} occurs with probability at least $1-2p^2\exp(-n\eps^2/2)$. 
\end{lemma}
\begin{proof}
 For a given pair of nodes $w,\wt$, let $Z^{(i)}=X^{(i)}_w X^{(i)}_{\wt}$ and apply Hoeffding's inequality  (Lemma~\ref{lem:Hoeffding} in Appendix~\ref{app:devBD}) to get 
 $\P\left[|\mu_{w,\wt} -  \wmu_{w,\wt} |>\epsilon\right]\leq 2\exp(-n \epsilon^2/2)$. Applying the union bound over ${p \choose 2}$ pairs $w,\wt\in\VV$ of nodes completes the proof. 
\end{proof}

\rvw{We next prove Lemma ~\ref{l:greed} for completeness.}
\begin{proof} [{Proof of Lemma \ref{l:greed}}]
 For edge $(u,\ut)\notin \EE_{\Tc}$, if there is an edge $(i,j)\in\path_{\Tc}(u,\ut)$ such that $|\wmu_{u\ut}|> |\wmu_{ij}|$, then  $\Tc$ cannot be the maximum weight spanning tree. To show that, consider the tree $\T'$ identical to $\Tc$ except $(u,\ut)\in\EE_{\T'}$ and $(i,j)\notin\EE_{\T'}$ (\textit{i.e.,} $\EE_{\T'}= (\EE_{\Tc}\setminus\{(i,j)\})\cup\{(u,\ut)\}\,.$) Note that $\T'$ is a spanning tree and observe that
 $\text{weight}(\T') \triangleq\sum_{e\in\EE_{\T'}}|\wmu_{e}| =\sum_{e\in\EE_{\Tc}}|\wmu_{e}| +|\wmu_{u\ut}|-
|\wmu_{ij}|> \text{weight}(\Tc)$.  
\end{proof}

We define a pair of random variables that will help to characterize the mistakes made by the Chow-Liu algorithm. For a given pair of nodes $v,\vt$ and edge $f=(u,\ut)\in\path_{\T}(v,\vt)$
, let
\begin{eqnarray}
Z_{f,v,\vt} & = & X_u X_{\ut} -X_v X_{\vt} = X_u X_{\ut} \,(1 -X_u X_v X_{\vt}X_{\ut})\,,\label{e:DefZBern}\\
Y_{f,v,\vt} & = & X_u X_{\ut}+X_v X_{\vt}= X_u X_{\ut} \,(1 +X_u X_v X_{\vt}X_{\ut})\,. \label{e:DefYBern}
\end{eqnarray}

\begin{lemma}
\label{l:MissingEdgeZY}
If there exists a pair of edges $f=(u,\ut)$ and $g=(v,\vt)$ such that $f\in\EE_{\T}\setminus \EE_{\Tc}$, $g\in\EE_{\Tc}\setminus \EE_{\T}$ and additionally $f\in\path_{\T}(v,\vt)$ and $g\in\path_{\Tc}(u,\ut)$, then, 
$$\left(\sum_{i=1}^{n} Z^{(i)}_{f,v,\vt}\right) \left(\sum_{i=1}^{n} Y^{(i)}_{f,v,\vt}\right)\leq 0\,. $$
\end{lemma} 
\begin{proof}
Using Lemma \ref{l:greed}, $f=(u,\ut)\notin\EE_{\Tc}$ and $g=(v,\vt)\in\path_{\Tc}(u,\ut)$ implies that $|\wmu_{g}|\geq |\wmu_{f}|$. Hence $\wmu_{g}^2\geq \wmu_{f}^2$ and 
\begin{align*}
0 & \geq \widehat{\mu}_f^2 - \widehat{\mu}_{g}^2 =  (\wmu_f-\wmu_g)(\wmu_f+\wmu_g) \\ 
& =\frac{1}{n^2}\left(\sum_{i=1}^n X^{(i)}_u X^{(i)}_{\ut}-X^{(i)}_v X^{(i)}_{\vt}\right)\left(\sum_{i=1}^n X^{(i)}_u X^{(i)}_{\ut}+X^{(i)}_v X^{(i)}_{\vt}\right)\\
& = \frac{1}{n^2}\left(\sum_{i=1}^{n} Z^{(i)}_{f,v,\vt}\right) \left(\sum_{i=1}^{n} Y^{(i)}_{f,v,\vt}\right)
\end{align*}
where in the last step we used $f\in\path_{\T}(v,\vt)$ and the definition of $Z_{f,w,\wt}$ and $Y_{f,w,\wt}$ in~\eqref{e:DefZBern} and~\eqref{e:DefYBern}.
\end{proof}
Later we will bound the probability of the event in Lemma~\ref{l:MissingEdgeZY}.
 To this end, we will derive deviation bounds on $Z_{f,v,\vt}$ and $Y_{f,v,\vt}$ in Lemma \ref{l:Zconcentration}.
We use the standard Bernstein's inequality as quoted from \cite{tsybakov2008introduction} in Appendix~\ref{app:devBD}

\begin{lemma}
\label{l:Zconcentration}
For all pairs of nodes $v,\vt\in\VV$ and edges $f=(u,\ut)\in\path_{\T}(v,\vt)$, let $\A_{f,v,\vt} = \path_{\T}(v,\vt)\setminus \{f\}$ 
 such that 
 $\mu_{v\vt}=\mu_{f} \mu_{\A_{f,v,\vt}}$.
 Given $n$ i.i.d.  samples let $Z^{(1)}_{f,v,\vt},  \cdots, Z^{(n)}_{f,v,\vt}$ be defined in~\eqref{e:DefZBern}, and $Y^{(1)}_{f,v,\vt},  \cdots, Y^{(n)}_{f,v,\vt}$ be defined in~\eqref{e:DefYBern}. Let $\epsilon=\sqrt{2/n \log(2p^2/\delta)}$. Then,  with probability at least $1-\delta$
\begin{align}
\label{eq:Zconcentration}
\bigg|\frac 1 n\sum_{i=1}^{n} Z^{(i)}_{f,v,\vt} - \mu_f(1-\mu_{\A_{f,v,\vt}})\bigg|
& \leq 
\max\left\{4 \epsilon^2, 4\epsilon \sqrt{1-\mu_{\A_{f,v,\vt}}}\right\}\,\,\,\text{and}\\
\label{eq:Yconcentration}
\bigg|\frac 1n\sum_{i=1}^{n} Y^{(i)}_{f,v,\vt} - \mu_f(1+\mu_{\A_{f,v,\vt}})\bigg| 
&\leq
 \max\left\{4 \epsilon^2,4\epsilon \sqrt{ 1+\mu_{\A_{f,v,\vt}}}\right\}\,.
\end{align}
\end{lemma}

\begin{proof}
We prove that~\eqref{eq:Zconcentration} holds with probability at least $1-\delta/2$. The proof of~\eqref{eq:Yconcentration} is analogous.

We use the abbreviation $\A$ instead of $\A_{f,v,\vt}$ and $Z$ instead of $Z_{f,v,\vt}$ in this proof.
Applying Lemma \ref{l:IndependentError}, it follows from the fact that $P$ is Markov with respect to $\T$ and $f=(u,\ut)\in\path_{\T}(v,\vt)$ that \rvw{$X_u X_{\ut}$ and $X_u X_v X_{\ut}X_{\vt}$} are independent random variables. Note that 
$P(X_u X_{\ut}=1) = (1+\mu_{f})/2$ and $P(X_u X_{\ut}=-1) = (1-\mu_{f})/2$.
Similarly, the distribution of \rvw{$X_u X_v X_{\ut} X_{\vt}$} is a function of $\mu_{\A}$. 
As a result, the random variable $Z_{f,v,\vt}\in \{-2,0,2\}$ defined in~\eqref{e:DefZBern} has the following distribution:
$$Z
= \begin{cases}
-2& \qquad \text{w.p.}\quad \frac{1-\mu_{f}}{2}\frac{1-\mu_{\A}}{2}\\
0& \qquad \text{w.p.} \quad\frac{1+\mu_{\A}}{2}\\
+2& \qquad \text{w.p.} \quad\frac{1+\mu_{f}}{2}\frac{1-\mu_{\A}}{2}\,.\\
\end{cases}$$
The first and second moments of $Z$ are $ \Ex[Z]=\mu_{f}(1-\mu_{\A})$ and $\text{Var}[Z]= (1-\mu_{\A})[2-\mu_{f}^2(1-\mu_{\A})]\leq 2(1-\mu_{\A})$.
By Bernstein's inequality (Lemma \ref{l:Bernstein} in Appendix~\ref{app:devBD}), with probability at least $1-\delta/2$ 
$$\left|\sum_{i=1}^{n} Z^{(i)} - n\Ex[Z]\right| \leq n\,\max\left\{ \frac{8}{3n}\log\frac{4}{\delta}, \sqrt{\frac{4\text{Var}[Z_{f,v,\vt}]}{n}\log\frac{4}{\delta}} \right\}\,.$$
Using a union bound, we show that for any pair of nodes $v,\vt$ and any edge $f=(u,\ut)\in\path_{\T}(v,\vt)$,
\begin{align*}
&\bigg|\sum_{i=1}^{n} Z^{(i)} 
-
 n\mu_f(1-\mu_{\A})\bigg| 
 \leq 
 n\,
 \max\left\{ \frac{8}{3n}\log\frac{4p^3}{\delta}
 ,
   \sqrt{ \frac{8(1-\mu_{\A})}{n}\log\frac{4p^3}{\delta}}
   \right\}\,.
\end{align*}
The definition of $\epsilon$  gives $\frac{8}{3n}\log\frac{4p^3}{\delta}\leq 4\epsilon^2$ and $ \sqrt{ \frac{8(1-\mu_{\A})}{n}\log\frac{4p^3}{\delta}}\leq 4\epsilon \sqrt{1-\mu_{\A}}$  which gives the lemma.
\end{proof}

Event $\Pcorr$ in~\eqref{e:Pcorr} occurs if all of the strong edges in $\T$  (defined in~\eqref{e:Estrong}) are recovered in $\Tc$. Lemma \ref{l:MissingEdgeCharac} shows that the deviation bounds for the variables $Z_{f,v,\vt}$ and  $Y_{f,v,\vt}$  stated in~\eqref{eq:Zconcentration} and~\eqref{eq:Yconcentration} imply $\Pcorr$. 

\begin{lemma}
\label{l:MissingEdgeCharac}
Under the events described in Lemma~\ref{l:Zconcentration}
, if there is an edge $f\in \EE_{\T}$ missing from the Chow-Liu tree, 
$f\notin\EE_{\Tc}$,
 then 
 $ |\mu_{f}| 
  \leq 
  \tau(\eps) 
  = 
  \frac{4\epsilon}{\sqrt{1-\tanh \beta }}$
   (\textit{i.e.}, $\Pcorr$ defined in Equation~\eqref{e:Pcorr} holds).
\end{lemma}

\begin{proof}
Applying Lemma \ref{l:edgePair} to $f=(u,\ut)$ shows that for the edge $f\in \EE_{\T}\setminus\EE_{\Tc}$, there exists an edge $g=(v,\vt)\in\EE_{\Tc}\setminus\EE_{\T}$ such that, $f\in\path_{\T}(v,\vt)$ and $g\in\path_{\Tc}(u,\ut)$ (Figure \ref{f:proof}). 
Let $Z=Z_{f,v,\vt}$ defined in~\ref{e:DefZBern} and $Y=
Y_{f,v,\vt}$ defined in ~\ref{e:DefYBern} in the scope of this proof. 
Applying Lemma \ref{l:MissingEdgeZY}, this implies that $\left(\sum_{i=1}^{n} Z^{(i)}\right) \left(\sum_{i=1}^{n} Y^{(i)}\right)\leq 0$. Note that $\Ex Z =\mu_f(1-\mu_{\A})$ and $\Ex Y =\mu_f(1+\mu_{\A})$ using the definition $\A=\A_{f,v,\vt}=\path_{\T}(v,\vt)\setminus \{f\}$. Hence $\left(\Ex Z\right)\, \left(\Ex Y\right) =\mu_f^2(1-\mu^2_{\A}) \geq 0\,. $ Thus, $\left(\sum_{i=1}^{n} Z^{(i)}\right) \left(\sum_{i=1}^{n} Y^{(i)}\right)<0$ holds only if either one of the following inequalities  holds: 
\begin{align*}
\Big|\sum_{i=1}^{n} Z^{(i)}-n\Ex Z\Big|\geq n\big|\Ex Z\big|
\qquad \text{or}
\qquad
\Big|\sum_{i=1}^{n} Y^{(i)}-n\Ex Y\Big|\geq n\big|\Ex Y\big|\,.
\end{align*}
On the events described in Lemma \ref{l:Zconcentration}
, there is an upper bound on $\big|\sum_{i=1}^{n} Z^{(i)}-n\Ex Z\big|$ and $\big|\sum_{i=1}^{n} Y^{(i)}-n\Ex Y\big|$. Hence, on these events, the property in above display holds only if either one of these inequalities holds:
\begin{align*}
|\mu_{f}(1-\mu_{\A})| & \leq \max\left\{4\epsilon^2,4\epsilon \sqrt{ 1-\mu_{\A}}\right\} \quad \text{or}\\
|\mu_{f}(1+\mu_{\A})|&  \leq \max\left\{4\epsilon^2,4\epsilon \sqrt{ 1+\mu_{\A}}\right\}\,,
\end{align*}
which is true if
\begin{align*}|\mu_{f}|&\leq \max\left\{ \frac{4\epsilon}{\sqrt{1-\mu_{\A}}},\frac{4\epsilon^2}{1-\mu_{\A}},\frac{4\epsilon}{\sqrt{1+\mu_{\A}}},\frac{4\epsilon^2}{1+\mu_{\A}}\right\} 
\leq \max\{\tau(\eps),\tau^2(\eps)\}\,,\end{align*}
where $\tau(\eps)$ is defined in~\eqref{e:Estrong}.
Note that if $\tau(\eps)\geq 1$ then the bound on $|\mu_f|\leq 1 \leq \tau(\eps)$ is trivial. If $\tau(\eps)< 1$, then $\tau^2(\eps)< \tau(\eps)$ which gives $|\mu_{f}|  \leq  \tau(\eps)$.
\end{proof}

\begin{lemma}\label{l:PcorrLemma}
With $\eps=\sqrt{2/n \log(2p^2/\delta)}$, event $\Pcorr$ defined in~\eqref{e:Pcorr} occurs with probability at least $1-\delta$.
\end{lemma}

\begin{proof}
\label{l:strongRecovered}
Lemma \ref{l:MissingEdgeCharac} shows that, under the events described in Lemma~\ref{l:Zconcentration}
, the event $\Pcorr$ defined in Equation~\eqref{e:Pcorr} holds. Using Lemma \ref{l:Zconcentration}
, with probability at least $1-\delta$, all edges $e\in\T$ such that $|\mu_e|> \tau(\eps)$  are recovered by Chow-Liu algorithm $e\in\Tc$. 
\end{proof}

\begin{lemma}
\label{l:IndependentError}
Let the distribution $P(x)\in\PP(\T)$ be a zero-field Ising model on the tree $\T=(\VV,\EE)$. 
For all $e=(i,j) \in\EE$ let $Y_e=X_iX_j$. Then the random variables $\{Y_e\}_{e\in \T}$ are jointly independent.
\end{lemma}
This follows from the factorization of distribution $P(x)\in\PP(\T)$ in~\eqref{e:Ising}.

\rvw{
Next, we prove an upper bound on the end-to-end error on paths in the tree $\T$. Interestingly, the bound is dimension-free: the error is independent of the length of path. 
Appendix~\ref{sec:NcorrProof} contains the proof of Lemma~\ref{l:NcorrLemma}.

\begin{lemma}\label{l:NcorrLemma}
Suppose $\gamma<1$. If $n>\max\{25/\gamma^2 \log(4p^2/\delta), 108 e^{2\beta}\log(2p^3/\delta)\}$, then the event $\Ncorrg$ defined in~\eqref{e:Ncorrdef} occurs with probability at least $1-\delta$.
\end{lemma}
}

\begin{lemma}\label{l:edgePair}
	Let $\T_1$ and $\T_2$ be two spanning trees on a set of nodes $\VV$. Let $w,\wt$ be a pair of nodes  such that $\path_{\T_1}(w,\wt)\neq \path_{\T_2}(w,\wt)$. Then there exists a pair of edges $f\triangleq (u,\ut)\in \path_{\T_1}{(w,\wt)}$ and $g\triangleq (v,\vt)\in \path_{\T_2}{(w,\wt)}$ such that 
	\begin{enumerate}
		\item[(i)] 	$f\notin \path_{\T_2}{(w,\wt)}$ and $g\notin \path_{\T_1}{(w,\wt)}$
		\item[(ii)]  $f\in \path_{\T_1}{(v,\vt)}$ and $g\in \path_{\T_2}{(u,\ut)}$
	\end{enumerate}
\end{lemma}
\noindent Since $f\in\path_{\T_1}(w,\wt) \cap \path_{\T_1}{(v,\vt)}$, $w$ and $\wt$ (and respectively $v$ and $\vt$) are in different subtrees of $\T_1$ after removing edge $f$, one can label the end points of the edges $f=(u,\ut)$ and $g=(v,\vt)$ such that $u,v\in \subT{\T_1}{f}{w}$ and $\ut,\vt\in \subT{\T_1}{f}{\wt}$ (Figure \ref{f:proof}).
Lemma~\ref{l:edgePair} is proved in Appendix~\ref{s:AppendixTwoTrees}.


\section{Numerical simulations}
\label{s:NumSim}

We use numerical simulations to demonstrate the performance of the Chow-Liu algorithm in terms of both the probability of incorrect recovery of underlying structure (zero-one loss defined in~\eqref{e:zerooneLoss}) and the $\LL^{(2)}$ loss defined in~\eqref{e:LL2}.
We are specifically interested in the regime in which the number of samples is not large enough to guarantee the correct recovery of the underlying tree.

In these simulations, the generative probability distributions of the samples are factorized according to~\eqref{e:Ising} for a randomly chosen tree uniform over the set of trees on $p$ nodes. To observe the effect of upper and lower bounds on the edge parameters, for each edge $(i,j)\in\EE$,  the edge parameter $\theta_{ij}$ takes one of the values $\alpha$ or $\beta$ with equal probability.

	\begin{figure}[H]
	\centering
	\begin{subfigure}{.49\textwidth} 
	\includegraphics[width=.97\linewidth , height=1.4in]{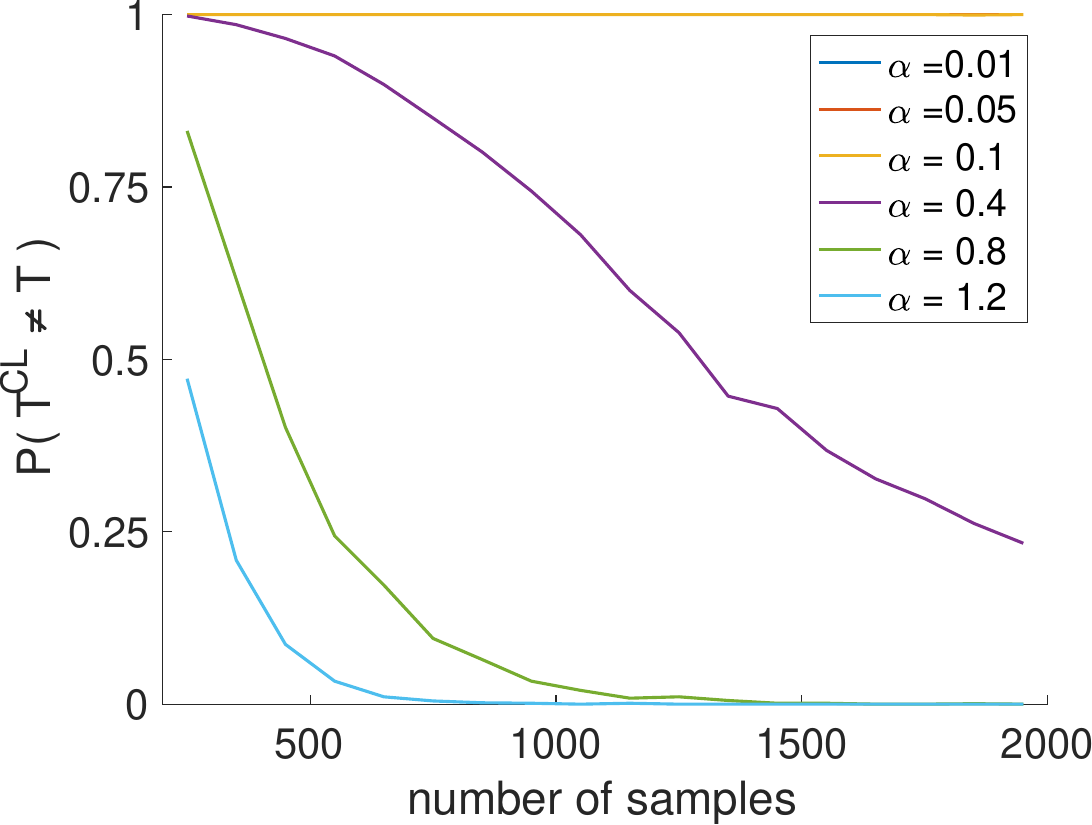}
	\caption{}
	\label{fig:alpha-pr}
	\end{subfigure}%
	~
		\begin{subfigure}{.49\textwidth}
	\includegraphics[width=.97\linewidth , height=1.4in]{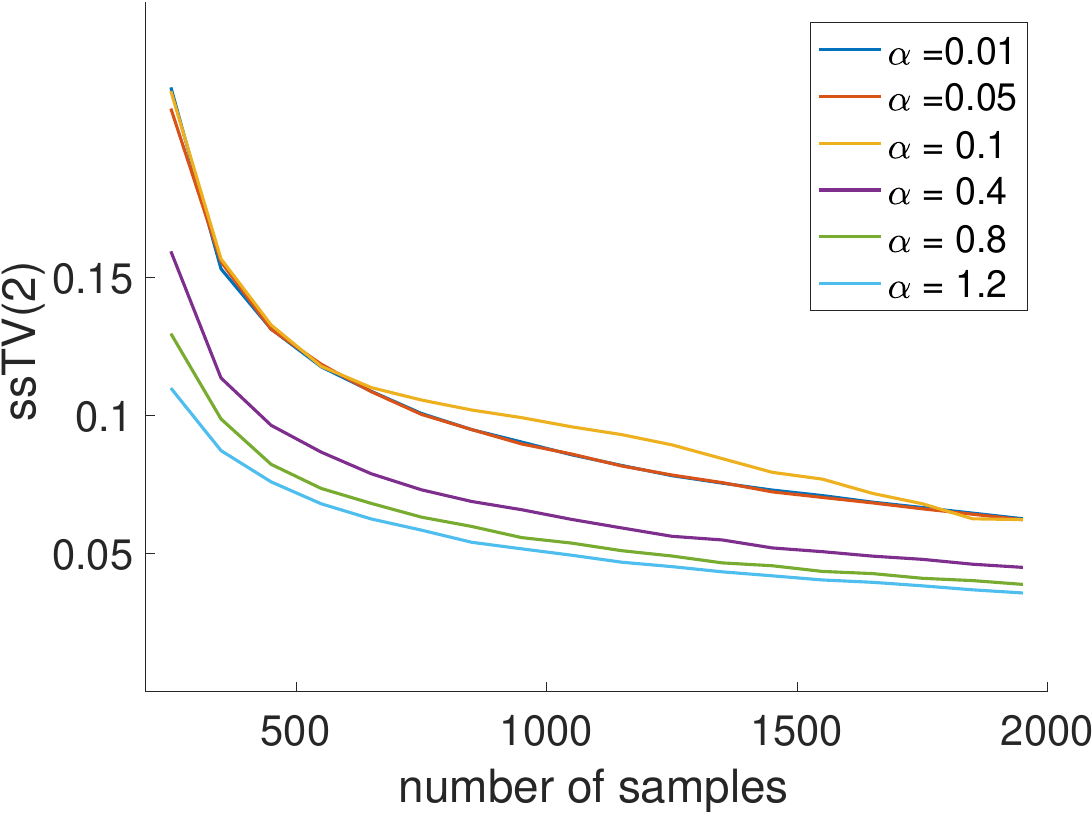}
	\caption{}
	\label{fig:alpha-sstv}
	\end{subfigure}%
	\\
	\begin{subfigure}{.49\textwidth} 
	\includegraphics[width=.97\linewidth , height=1.4in]{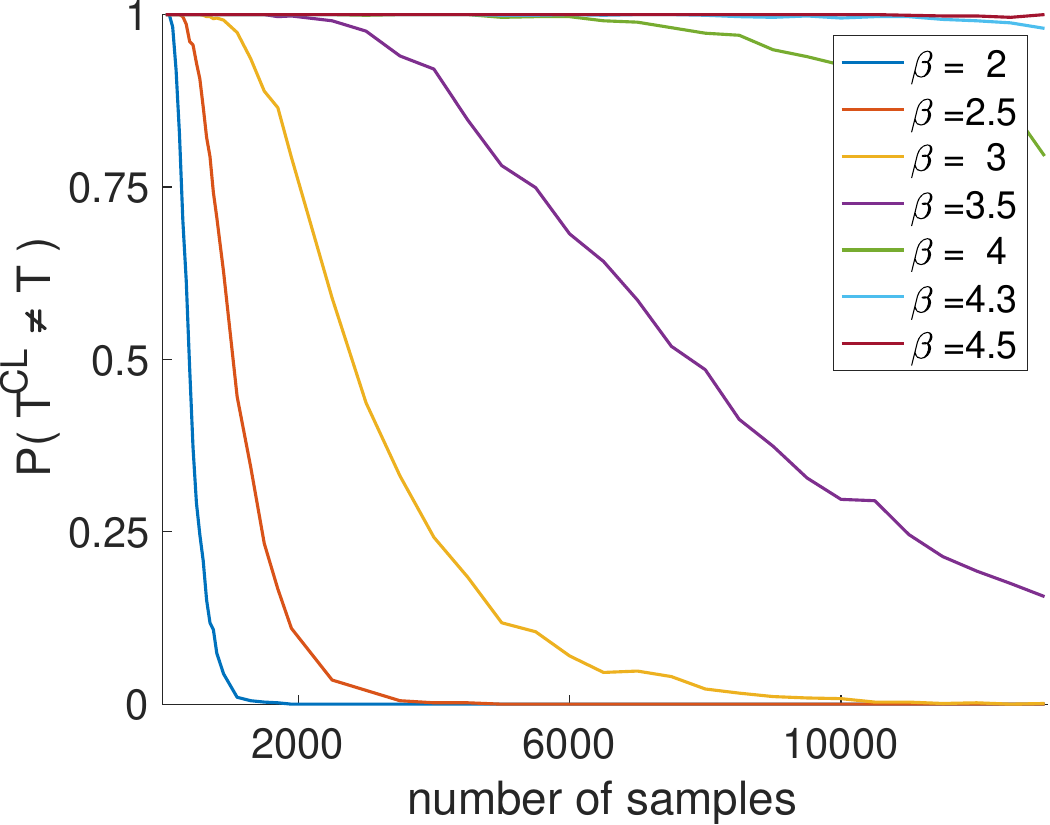}
	\caption{}
	\label{fig:beta-pr}
	\end{subfigure}%
	~
		\begin{subfigure}{.49\textwidth}
	\includegraphics[width=.97\linewidth , height=1.4in]{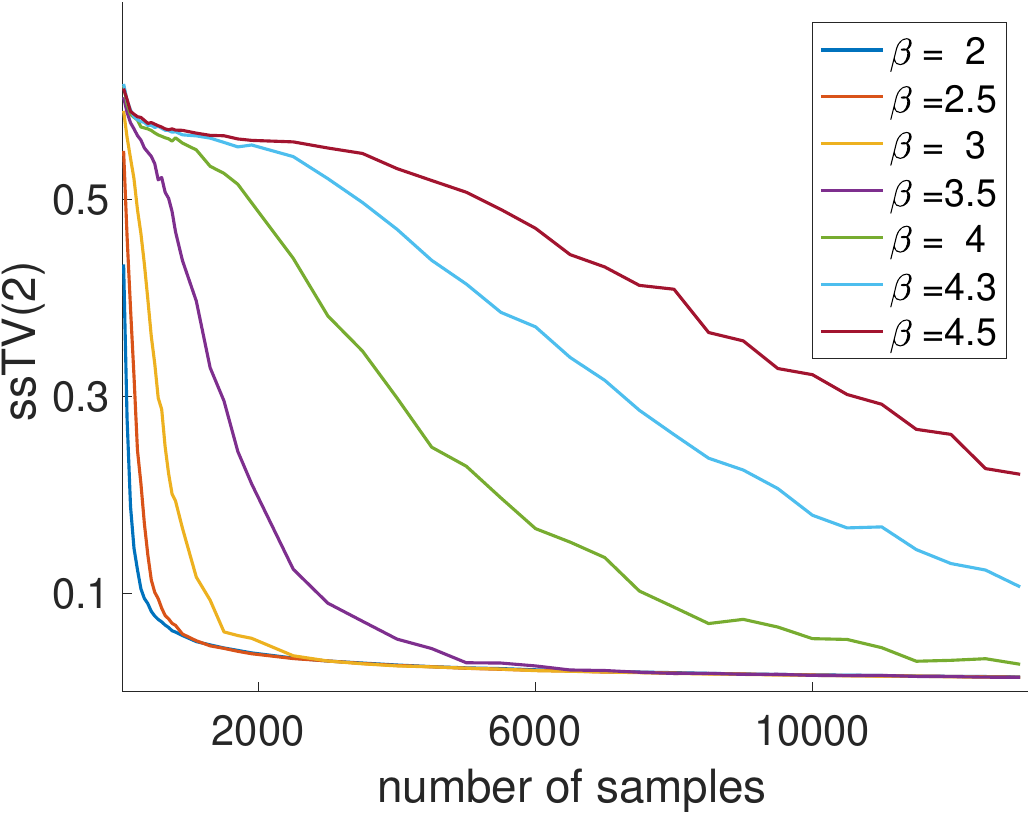}
	\caption{}
	\label{fig:beta-sstv}
	\end{subfigure}
	  \caption{The performance of the Chow-Liu algorithm as a function of the number of samples for generative distribution factorized according to~\eqref{e:Ising} with $p=31$. Figures (a) and (b) use $\beta=2$ and different values of $\alpha$. Figures (c) and (d) use $\alpha=.8$ and different values of $\beta$. Figures (a) and (c) depict  the probability of error and Figures (b) and (d) show the small set TV $\LL^{2}$ as the performance metric.}
	\end{figure}

In Figure~\ref{fig:alpha-pr}, we plot $\P[\Tc\neq\T]$ as a function of the number of samples with $p=31$, $\beta=2$ and different values of $\alpha$. One can observe that the probability of error is higher for smaller values of $\alpha$ (it increases to one as $\alpha$ decays to zero for any value of $n$).  Figure~\ref{fig:alpha-sstv} illustrates the $\LL^{(2)}$ loss in the same setup. This plot shows that as $\alpha$ decays, the loss remains bounded (although not necessarily monotonic in $\alpha$), consistent with Theorem~\ref{t:CLLowerBound}.

 Figures~\ref{fig:beta-pr} and~\ref{fig:beta-sstv} plot the probability of error and the $\LL^{(2)}$ loss as a function of $n$ with $p=31$, $\alpha=.8$, and different values of $\beta$. As $\beta$ increases,  Figure~\ref{fig:beta-pr} shows that the probability of error in learning the tree increases (for any $n$, the probability of error goes to one as $\beta$ grows large enough).  Figure~\ref{fig:beta-sstv} is consistent with the statement of Theorem~\ref{t:CLLowerBound}, which states that expected $\LL^{(2)}$ loss decays as $C' p \exp(-Cne^{-2\beta})+ C'' \sqrt{\log p/n}$.

Appendix~\ref{sec:NumSimSUpp} contains numerical simulations on implementation of forest approximation algorithm and its comparison with the Chow-Liu algorithm. It also contains the simulation results depicting the performance of Chow-Liu algorithm with misspecified models. A tree-structured Ising model  is changed isotropically on the space of distributions by a small offset to construct the generative distribution. This suggests that the the output of the Chow-Liu algorithm is robust to misspecification in the model and close to the generative distribution with respect to $\LL^{(2)}$ loss. 
The performance of the Chow-Liu algorithm in term of $\LL^{(k)}$ loss for general value of $k$ and for generative tree-structured Ising models in presence of external field are also studied in Appendix~\ref{sec:NumSimSUpp} . 
\section*{Discussion}
\label{sec:discussion}

In this paper we prove guarantees on accuracy of prediction for a learned tree-structured Ising model. 
There is a large literature on learning tree-structured Markov random fields, and it is useful to carefully compare the guarantees obtained by each when applied to our setting.
\rvw{In the supplementary material, we review different approaches that could be taken toward learning a tree-structured distribution. 
We also review some known algorithms and their sample complexity.}

There are many interesting questions remaining, including those mentioned in the introduction: model misspecification, how to close the gap between the upper and lower bound on sample complexity, Ising models with external field, and obtaining tight guarantees for $\LL^{(k)}$ (marginals of order $k$) for $k>2$. Of course, it is also of great interest to go beyond tree models and study other classes of models.

\section*{Acknowledgment}
We thank Jerry Li for pointing out an error in an early version of the manuscript and the anonymous reviewers for their feedback which greatly improved the quality of the paper.
We also thank Gregory Wornell, Lizhong Zheng, Gabor Lugosi, Devavrat Shah, David Gamarnik, Elchanan Mossel and Andrea Montanari for stimulating discussions. 

This work was supported in part by the grants ONR N00014-17-1-2147, DARPA  W911NF-16-1-0551, and NSF CCF-1565516.

\bibliographystyle{acm}
\bibliography{MRF_BIB}
\appendix
\section{Information projection lemmas}\label{sec:InfProj}
In the following, we prove two lemmas regarding the maximum likelihood tree.

\begin{lemma}\label{l:MomentMatching}
For a given tree $\T=(\VV,\EE)$, the reverse information-projection of the distribution  $P(x)$ onto the class of the Ising models on $\T$ with no external field (as in Equation~\eqref{e:Ising}), $\widetilde{P}(x)=\Pi_{\T}(P)\triangleq \argmin\limits_{Q\in \mathcal{P}_{\T}}D(P\|Q )$ with the parameter $\widetilde{\theta}\in \Omega_{0,\infty}(\T)$ has the following property:
$\tanh(\widetilde{\theta}_{ij})= \mu_{ij}$
for all $(i,j)\in \EE$, where $\mu_{ij}=\E_{P}X_i X_j$ is the pairwise correlation of the variables under the distribution $P(x)$. Also, we have $D(P\|\Pt)=-H(P) + \sum_{(i,j)\in\EE_{\T}}H_{B}(\frac{1+\mu_{ij}}{2})$ where $H_{B}(x)= x\log \frac{1}{x} + (1-x)\log \frac{1}{1-x}$ is the binary entropy function.
\end{lemma}

\begin{proof}
This lemma is a direct corollary of Theorem 3.3 in \cite{csiszar2004information}. Note that $\widetilde{P}(x)$ is the reverse I-projection of distribution $P(x)$ onto the exponential family with sufficient statistics $X_iX_j$ for $(i,j)\in\EE$. It follows that $\E_{\widetilde{P}}X_iX_j=\E_{P}X_iX_j$ for $(i,j)\in\EE$ which gives $\tanh(\widetilde{\theta}_{ij})= \mu_{ij}$. Section 3.4.2 in \cite{wainwright2008graphical} also addresses this problem in depth.

We provide the proof for the sake of completeness here: 
For the tree $\T=(\VV,\EE)$ and any $Q\in\PP_{\T}$ with parameter vector $\theta$, using~\eqref{e:Ising}, $Q$  factorizes as  $Q(x)=\prod_{(i,j)\in\EE}(1+\tanh \theta_{ij} x_i x_j)/2 $, where $\Ex_{Q} X_iX_j =\tanh \theta_{ij}$. Let $Y_e=X_iX_j$ for all $e=(i,j)\in\EE$ so that $Q(Y_e= +)=(1+\tanh\theta_{ij})/2$ and $P(Y_e=+)=(1+\mu_{ij})/2$ where $\mu_{ij}=\Ex_{P} X_iX_j$.

Given probability mass function $P$ on $\{-1,+1\}^p$ and any $Q\in\PP_{\T}$ with parameters $\theta$, 
\begin{align*}
D(P\|Q) &= \Ex_{P} \left[\log \frac{P}{Q}\right] = -H(P) -\Ex_{P}[\log Q(X)] \\
& = 
 -H(P) -\Ex_{P}\left[\log \prod_{(i,j)\in\EE}\frac{1+\tanh\theta_{ij}X_i X_j}{2}\right] \\
& = -H(P) -\sum_{e=(i,j)\in\EE} \Ex_{P} \left[\log \frac{1+\tanh \theta_{ij}Y_e}{2}\right] \\
&= -H(P) -\sum_{e\in\EE} \Ex_{P} [\log Q(Y_e)]
\\
& = -H(P)+\sum_{e\in\EE} H_B(P(Y_e)) + \sum_{e\in\EE}{D(P(Y_e)\|Q(Y_e))}\,.
\end{align*}
$\Pt=\argmin_{Q\in\PP_{\T}}D(P\|Q)$ with parameters $\widetilde{\theta}$ satisfies the property $\Pt(Y_{e})=P(Y_e)$ for all $e\in\EE$, hence $\tanh(\widetilde{\theta}_{ij})=\mu_{ij}$ for all $(i,j)\in\EE$.
Consequently, $D(P\|\Pt)=-H(P)+\sum_{e\in\EE}H_{B}\Big(\frac{1+\mu_{ij}}{2}\Big)\,.$
\end{proof}

\begin{lemma}\label{l:TcMWST}
Given empirical distribution $\Ph$, the tree $\Tc$ defined in Definition \ref{def:Chow-Liu} can be found as the maximum weight spanning tree over a complete weighted graph where the weights of each edge $(i,j)$ is $|\wmu_{ij}| \,.$
\end{lemma}

\begin{proof}
Using Lemma \ref{l:MomentMatching}, $\Tc=\argmin_{\T\in\TT} \sum_{e\in\EE_{\T}}H_{B}\big(\frac{1+\wmu_{e}}{2}\big)-H(\Ph) = \argmin_{\T\in\TT} \sum_{e\in\EE_{\T}}H_{B}\big(\frac{1+\wmu_{e}}{2}\big)\,.$
The maximum weight spanning tree can be implemented greedily using Kruskal algorithm or Prim's algorithm \cite{cormen2009introduction}. So, finding the maximum weight spanning tree only depends on the sorted order of the edges of the graph. 

$H_{B}\big(\frac{1+\wmu_{ij}}{2}\big)$ is a monotonically increasing function of $|\wmu_{ij}|$. So, sorting all the edges $(i,j)$ in the complete graph based $H_{B}\big(\frac{1+\wmu_{ij}}{2}\big)$ or $|\wmu_{ij}|$ gives the same order. This gives $$\Tc=\argmin_{\T\in\TT} \sum_{(i,j)\in\EE_{\T}}H_{B}\Big(\frac{1+\wmu_{ij}}{2}\Big)=\argmin_{\T\in\TT} \sum_{(i,j)\in\EE_{\T}}|\wmu_{ij}|\,.\qedhere$$
\end{proof}

\section{Samples Necessary for small SSTV}\label{sec:LBSmallSSTV}
In this theorem, we prove a lower bound for learning a tree-structured Ising model with the ssTV loss function $\LL^{(2)}$. Unlike the bound in Theorem~\ref{t:inferenceUpper}, the following lower bound does not require  $\tanh(\beta)>\tanh(\alpha)+2\eta$. This bound can be used specifically in the extreme case of $\alpha=\beta$.

\begin{theorem}\label{t:LBMainDegenerate}
Fix $\eta>0$. Let $\mu=\tanh\alpha$ and
\begin{equation}\label{eq:defdlowerBD2}
d=\min\Big\{\Big\lceil\frac{\log(1-2\eta/\mu)}{\log \mu}\Big\rceil, p-1 \Big\}\,.
\end{equation}
Suppose one observes $$n<\frac{\log(p-d)}{4\alpha \tanh(\alpha)[1-\tanh(\alpha)^d]}$$ samples. Then, the worst-case probability of $\LL^{(2)}$ loss greater than $\eta$ taken over trees $\T\in\TT$ and distributions $P\in\PP_{\T}(\alpha,\beta)$ is at least half for any algorithm, i.e., $$\inf_{\phi}\sup_{\T\in \TT\atop P\in\PP_{\T}(\al,\beta)} \P\left[\LL^{(2)}(P,\phi(X^{\small{(1:n)}})) > \eta\right]>1/2\,.$$  
\end{theorem}

\begin{proof}
We construct a family of $M=p-d$ models. Each of the models correspond to a tree-structured model in which the edge parameters on the edges of the tree are taking a fixed value $\theta_e=\alpha$. 
Let $\theta^0_{i,i+1}=\alpha$ for $i=1,\cdots,p-1$. For $m=1,\cdots,p-1-d$, in the $m$-th model, we have $\theta_{m, m+d+1}=\alpha$ and $\theta_{i,i+1}=\alpha$ for $i=1,\cdots,m-1,m+1,\cdots,p-1$. Hence, there are $p-d$ models. For $m'\neq m$,
\begin{align*}
\max_{i,j} |\Ex_{\theta^{m'}}[X_i X_j]-\Ex_{\theta^{m}}[X_i X_j]|  \geq \tanh(\alpha)[1-\tanh(\alpha)^d]\geq 2\eta\,.
\end{align*}
Also, using~\eqref{e:KLIsing} \[J(\theta^{m} \| \theta^{m'}) \leq 2\alpha \tanh(\alpha)[1-\tanh(\alpha)^d]\,. \]
Fano's inequality  (Corollary \ref{cor:Fano-loss}) gives the bound.
\end{proof}

The following corollary looks into the asymptotic value of the above lower bound when $p\to \infty$ ($p$ is large).

\begin{corollary}\label{cor:LBD2}
Fix $\eta>0$. 
If there exists a constant $0<c<1$ such that 
\[\tanh(\alpha)\big[1-\tanh(\alpha)^{(p^c)}\big]\geq 2\eta\]
(or equivalently, $d<p-\max\{1,p^{c}\}$ using the definition in Equation~\eqref{t:LBMainDegenerate}), and one observes
\[n<\frac{c}{16}\,\min\Big\{\frac{1}{\alpha\eta},\frac{e^{2\alpha}}{\alpha\tanh(\alpha)}\Big\}\log p\]
samples, then
$$\inf_{\phi}\sup_{\T\in \TT\atop P\in\PP_{\T}(\al,\beta)} \P\left[\LL^{(2)}(P,\phi(X^{\small{(1:n)}})) > \eta\right]>1/2\,.$$  
\end{corollary}

This corollary is a direct consequence of Theorem~\ref{t:LBMainDegenerate}.
  In particular, in the asymptotic regime of interest in which $\alpha=\beta$ (we don't assume any gap between upper and lower bounds of edge parameters), $p\to\infty$  and $\eta\to 0$, the above lower bound entails that if $n<Ce^{2\alpha}/[\alpha\tanh(\alpha)]\,\log p$, then the $\LL^{(2)}>\eta$ with probability greater than $1/2$. Interestingly, this lower bound is independent of $\eta$. This lower bound is also within multiplicative constant of  lower and upper bounds of correct recovery of structure (zero-one loss function) presented in Theorems~\ref{t:structLower} and~\ref{t:structUpper} (on the assumption $\alpha=\beta$). To get the matching upper bound in this regime, one could simply run Chow-Liu algorithm to recover $\Tc$ (which is guaranteed to be equal to $\T$ with probability greater than $1-\delta$ if $n>C'e^{2\alpha}\tanh^{-2}(\alpha)\,\log (p/\delta)$) and choose $\widetilde{\theta}_{ij}=\alpha$ for all $(i,j)\in\EE_{\Tc}$. 
\section{Proof of Sample Complexity Upper Bound of Forest Approximation Algorithm}
\label{sec:ForestApprox}
\begin{proof}
[Proof of Proposition \ref{prop:TruncationUpper}]

Let $\epsilon_1=\eta e^{-\beta}/8$ and $\gamma_1=\eta/4$. Using  Lemmas~\ref{l:EcorrLemma},~\ref{l:PcorrLemma} and~\ref{l:NcorrLemma}, given $n> \max\{128e^{2\beta}/\eta^{2} \log\frac{6p^2}{\delta} ,400/\eta^2\log\frac{12p^2}{\delta} \}$ samples, the event $\mathtt{E}(\eps_1,\gamma_1)$  occurs with probability greater than $1-\delta$. We will show that on the event $\mathtt{E}(\eps_1,\gamma_1)$, the statement of the proposition holds.
The proof has two steps. In Step~1, we prove that the output of the forest approximation algorithm is a subgraph of the original tree, \textit{i.e.}, $\EE_{\Th}\subseteq \EE_{\T}$. In Step~2, we show that the ssTV bound holds. 
\paragraph{Step~1}
To prove $\EE_{\Th}\subseteq\EE_{\T}$, we will show that on the event $\mathtt{E}(\eps_1,\gamma_1)$, all the edges in $(i,j)\in\EE_\Tc$  with $|\wmu_{ij}|>\tau(\eps_1)+\eps_1$ (the edges which survive the truncation algorithm and construct the approximated forest) belong to the original tree, \textit{i.e.}, $(i,j)\in\EE_{\T}$.

We prove that if  $(i,j)\in\EE_\Tc$  and $|\wmu_{ij}|>\tau(\eps_1)+\eps_1$, then $(i,j)\in\EE_\T$.  On the event $\mathtt{E}^\mathrm{corr}(\eps_1)$, defined in~\eqref{e:Ecorrdef}, $|\wmu_{ij}|>\tau(\eps_1)+\eps_1$ implies $|\mu_{ij}|>\tau(\eps_1)$. 
We prove  $(i,j)\in\EE_\T$ by contradiction:
 Let us assume $(i,j)\notin\EE_\T$. Then $\path_{\T}(i,j)$ has at least two edges in it. Also, $\tau(\eps_1)<|\mu_{ij}|=\prod_{e\in\path_{\T}(i,j)}|\mu_e|$. Then, for all $e\in\path_{\T}(i,j)$, we have $\tau(\eps_1)<|\mu_e|$ which implies $e\in\EEstrong_{\T}(\eps_1)$ as defined in~\eqref{e:Estrong}. Hence, on the event $\mathtt{E}^\mathrm{strong}(\eps_1)$, defined in~\eqref{e:Pcorr} $e\in\EE_{\Tc}$  for all $e\in\path_{\T}(i,j)$. This construct a path of length at least two between $i$ and $j$ in $\Tc$, which contradicts $(i,j)\in\EE_{\Tc}$.
 It follows that $\EE_{\Th} \subseteq \EE_{\T}$.

\paragraph{Step~2}
Next, we consider the edges $e\in\EE_{\T}$, with $|\mu_e|> \tau(\epsilon_1)+2\epsilon_1$.
On event $\mathtt{E}^\mathrm{strong}(\eps_1)$, these edges are recovered by the Chow-Liu algorithm.
On event $\mathtt{E}^\mathrm{corr}(\eps_1)$, the empirical correlation on these edges are bounded as $|\mu_e|> \tau(\epsilon_1)+\eps_1$. Hence, the forest approximation algorithm does not remove these edges. 
This implies that all  the  edges in the original tree $\T$ with $|\mu_{e}|> \tau(\eps_1)+2\epsilon_1$ are in the output of the forest approximation algorithm, $\Th$.

To bound the ssTV, we use the decomposition into two terms as per~\eqref{e:TrianLoss}. We first study the loss due to graph estimation error $\LL^{(2)}(P,\Pi_{\Th}(\P))$. For any pair of nodes $w$ and $\wt$ in $\Th$, since $\EE_{\Th} \subseteq \EE_{\T}$, if there is a path between $w$ and $\wt$, then this is the same as the path between them in $\T$. Hence looking at~\eqref{e:LL2}, $\Ex_{P}[X_wX_{\wt}]=\Ex_{\Pi_{\Th}(\P)}[X_wX_{\wt}]$ and the error is zero. 
If there is no path between $w$ and $\wt$ in $\Th$, then $\Ex_{\Pi_{\Th}(\P)}[X_wX_{\wt}]=0$ and the error is $\Ex_{P}[X_w X_{\wt}]$. This occurs whenever there is an edge $e\in\path_{\T}(w,\wt)$ such that $e\notin\EE_{\Th}$. All the missing edges in $\Th$ satisfy $|\mu_e|\leq \tau(\eps_1)+2\epsilon_1$. This implies that $|\Ex_{P}[X_wX_{\wt}]|\leq \tau(\eps_1)+2\epsilon_1$ and it follows that $\LL^{(2)}(P,\Pi_{\Th}(\P))\leq \tau(\eps_1)+2\epsilon_1\,.$

To bound the loss due to parameter estimation error $\LL^{(2)}(\Pi_{\Th}(\P),\Pi_{\Th}(\Ph))$, we again consider an arbitrary pair of nodes $w,\wt$. If there is a path between them in $\Th$, then there is the same path between them in $\T$ since $\EE_{\Th}\subseteq\EE_{\T}$. On the event $\mathtt{E}^\mathrm{cascade}(\gamma_1)$, for such a pair of nodes, $|\Ex_{\Pi_{\Th}(\P)}[X_wX_{\wt}]-\Ex_{\Pi_{\Th}(\Ph)}[X_wX_{\wt}]|\leq \gamma_1$ . If there is no path between $w$ and $\wt$ in $\Th$, then $\Ex_{\Pi_{\Th}(\P)}[X_wX_{\wt}]=\Ex_{\Pi_{\Th}(\Ph)}[X_wX_{\wt}]=0$. Hence, $\LL^{(2)}(\Pi_{\Th}(\P),\Pi_{\Th}(\Ph))\leq \gamma_1$.

Combining the two error terms gives $\LL^{(2)}(P,\Pi_{\Th}(\Ph))\leq \tau(\eps_1)+2\epsilon_1+\gamma_1$. With the choice of $\epsilon_1 = \eta e^{-\beta}/8$ and $\gamma_1= \eta/4\,$ then $\tau(\eps_1)=4\epsilon_1/\sqrt{1-\tanh \beta}\leq 4\epsilon e^{\beta}\leq \eta/8$ and $\LL^{(2)}(P,\Pi_{\Th}(\Ph))\leq \tau(\eps_1)+2\epsilon_1+\gamma_1 \leq \eta$.
\end{proof}

\section{Tail bounds}
\label{app:devBD}
We state a standard form of Hoeffding's inequality \cite{hoeffding1963probability}. This bound is used in the proof of Lemma~ref{l:EcorrLemma}
to show that event $\Ecorr$ defined in~\eqref{e:Ecorrdef} occurs with probability at least $1-2p^2\exp(-n\eps^2/2)$. 
\begin{lemma}[Hoeffding's inequality]\label{lem:Hoeffding}
Let $Z^{(1)},\cdots,Z^{(n)}$ be $n$ i.i.d. random variables taking values in the interval $[-1,1]$ and let $S_n=\frac1{n}\big(Z^{(1)}+\cdots+\\ Z^{(n)}\big)$. Then,
$\P\left[|S_n -  \Ex Z^{(1)} |>t\right]\leq 2\exp(-n t^2/2)\,.$
\end{lemma}

The next lemma is the standard Bernstein's inequality as quoted from \cite{tsybakov2008introduction}.
This bound is used in the proof of 
Lemma~\ref{l:PcorrLemma} to show that 
the event $\Pcorr$ defined in~\eqref{e:Pcorr} with $\eps=\sqrt{2/n \log(2p^2/\delta)}$ occurs with probability at least $1-\delta$.
\begin{lemma}[Bernstein's inequality]\label{l:Bernstein}
 Let $Z^{(1)}, ..., Z^{(n)}$ be i.i.d. random variables. 
 Suppose that for all $i$,  $|Z^{(i)}-\Ex[Z^{(i)}]|\leq M$ almost surely.
  Then, for all positive $t$,
\[\P\left[\bigg|\sum_{i=1}^n Z^{(i)} - n\Ex[Z^{(1)}] \bigg|\geq t\right]\leq 2 \exp\left(-\frac{t^2}{2n\mathsf{VAR}(Z)+\frac{2}{3}Mt}\right)\,.\]
\end{lemma}

The following lemma is essentially a restatement of Theorem 3 of Hoeffding \cite{hoeffding1963probability}.
This bound is used in the proof of 
 Lemma~\ref{l:NcorrLemma} in Appendix~\ref{sec:NcorrProof} to show that if $n>\max\{25/\gamma^2 \log(4p^2/\delta), 108 e^{2\beta}\log(2p^3/\delta)\}$, then the event $\Ncorrg$ defined in~\eqref{e:Ncorrdef} occurs with probability at least $1-\delta$. 
 
\begin{lemma} \label{l:hoeffd2}
Let $Z^{(1)}, \cdots, Z^{(n)}$ be i.i.d. random variables such that $\Ex{\big[Z^{(i)}\big]}=0,$ $|Z^{(i)}|\leq b$ and $\Ex{\big[(Z^{(i)})^2\big]}\leq \sigma^2$. Let $\bar{Z}=\frac{1}{n}\sum_{i=1}^n Z^{(i)}$. Then
\begin{align}
\Ex{\big[\exp(\lambda \bar{Z})\big]} &\leq \exp\left\{n\,f\Big(\frac{\sigma^2}{b^2} ,\frac{b|\lambda|}{n}\Big)\right\},
\end{align}
where we defined 
\begin{equation}
f(u,c)=\log\left(\frac{1}{1+u}e^{-cu}+ \frac{u}{1+u}e^c\right).
\end{equation}
Also, if $c\geq 0$, the function $f(u,c)$ is concave and nondecreasing in $u$ for $u\geq 0$. And for any $t\geq 0,$
\begin{equation} \label{eq:hoeff2-deviation}
\P\Big[|\bar{Z}|> t\Big] \leq 2 \left\{\Big(1+\frac{bt}{\sigma^2}\Big)^{-(1+bt/\sigma^2) \frac{\sigma^2}{b^2+\sigma^2}} \,\Big(1-\frac{t}{b}\Big)^{-(1-t/b) \frac{b^2}{b^2+\sigma^2}} \right\}^n\,.
\end{equation}
\end{lemma}
\begin{proof}
 Using Equation (11) in Bennett \cite{bennett1962probability}, for any $\lambda$,
\begin{align*}
\Ex{\left[\exp{\big(\lambda Z_i\big)}\right]} &\leq \frac{b^2}{b^2+\sigma^2} \exp\Big(-\frac{\sigma^2}{b}|\lambda|\Big) +\frac{\sigma^2}{b^2+\sigma^2} \exp\big(b|\lambda|\big) = \exp\left\{ f\Big(\frac{\sigma^2}{b^2} , b|\lambda|\Big)\right\}.
\end{align*}
Bennett \cite{bennett1962probability} proves this for $\lambda>0$ with the assumption $Z_i<b$, but the above modification is proved similarly using $|Z_i|< b$.
 Hence, 
 \begin{align*}
\Ex{\Big[\exp{\big(\lambda \bar{Z}\big)}\Big]} &=\prod_{i=1}^n \Ex{\left[\exp\Big(\frac{\lambda}{n}Z_i\Big)\right]} 
\leq \exp\left\{n\,f\Big(\frac{\sigma^2}{b^2} , b\frac{|\lambda|}{n}\Big)\right\}.
\end{align*}
Lemma 3 in \cite{hoeffding1963probability} shows that  $\frac{d^2}{du^2}f(u,c)\leq 0$ for $u\geq 0$ and $c\geq 0$. Following the notation in the same proof, to show that $f(u,c)$ is increasing in $u$, we write $f(u,c)=c+ \log f_1(y,c),$ where $y=1+u$ and $f_1(y,c)= \big(\exp(-cy) - 1 + y\big)/y$. So, for any $c>0,$
\[\frac{d}{dy}f_1(y,c) =\big[ 1-(1+cy)e^{-cy}\big]y^{-2} \geq 0\,\]
since $e^{cy}\geq 1+cy\,.$ This shows that $f(u,c)$ is nondecreasing in $u$ for $u\geq 0$. 

Using a Chernoff bound, 
\begin{align*}
\P[\bar{Z}> t] \leq \min_{\lambda>0} \exp\left\{n\Big[\,f\Big(\frac{\sigma^2}{b^2} , b\frac{|\lambda|}{n}\Big)- \lambda \frac{t}{n}\Big]\right\}\,.
\end{align*}
Plugging $\lambda=\lambda^*$ with
\[\lambda^* =\frac{n b}{b^2+\sigma^2} \log \frac{1+ \frac{tb}{\sigma^2}}{1-\frac{t}{b}}.\]
 with change of variables $u=\sigma^2/b^2$ and $s=tb/\sigma^2$ gives
\begin{align*}
\P[\bar{Z}> t] \leq  &\left[\frac{1}{1+u} \exp\Big(\frac{-u}{1+u} \log\frac{1+s}{1-su}\Big) + \frac{u}{1+u}\exp\Big(\frac{1}{1+u} \log\frac{1+s}{1-su}\Big)\right]^n \\
&\exp\left(-\frac{s u n}{u+1}\log\frac{1+s}{1-su}\right) \\
=& (1+s)^{-(1+s)\frac{un}{u+1}} (1-su)^{-(1-su)\frac{n}{u+1}}
\left[\frac{1-su}{1+u} + \frac{u(1+s)}{u+1}  \right]^n\\
=& (1+s)^{-(1+s)\frac{un}{u+1}} (1-su)^{-(1-su)\frac{n}{u+1}}\,.
\end{align*}
A similar Chernoff bounds with $\lambda=-\lambda^*$ upper bounds $\P[\bar{Z}<-t]$. Combining these two bounds gives~\eqref{eq:hoeff2-deviation}.
\end{proof}
\section{End-to-end error on paths in the original tree (Proof of Lemma 8.10)}\label{sec:NcorrProof}

For each edge $e\in\EE_{\T}$,
let 
\begin{equation}\label{eq:def-gammae}
\gamma_e = \sqrt{3 \frac{1-\mu_e^2}{n}\log(2p/\delta)}
\end{equation}
and 
\begin{equation}\label{eq:defEdgErr}
  \EdgErr_e = \Big\{|\wmu_e-\mu_e|\leq \gamma_e\Big\}
\end{equation}
be the event that the empirical correlation on edge $e$ is within $\gamma_e$ of the population value. For $\A\subseteq\EE_{\T}$, let $\EdgErr({\A})=\cap_{e\in\A}\EdgErr_e$ (This event will only be used for the purpose of proof of Lemma~\ref{l:NcorrLemma}). 

Now, if $n> 108 e^{2\beta} \log(2 p/\delta)$, then 
\begin{equation}\label{eq:bound-gamma-e}
\gamma_e < \sqrt{\frac{3(1-\mu_e^2)}{108 e^{2\beta} \log(2p/\delta)}\log(2p/\delta)}\leq \frac{1}{6}(1-\mu_e^2)\,,
\end{equation}
where the last inequality is due to $e^{-2\beta}\leq 1-\tanh(\beta)\leq 1-|\mu_e|\leq 1-\mu_e^2$. 

\begin{lemma}\label{l:EdgErrprb}
If $n>108 e^{2\beta} \log(2 p/\delta)$, then $\P\big[\,\big(\EdgErr({\EE_{\T})}\big)^c\big]\leq \delta$.
\end{lemma}

\begin{proof}
For each edge $e$, the variance of $\wmu_e$ is $(1-\mu_e^2)/n$.
By Bernstein's inequality (Lemma~\ref{l:Bernstein}) and the bound~\eqref{eq:bound-gamma-e} on $\gamma_e$,
\[\P[\,(\EdgErr_{e})^c]\leq 2\exp\left(-\frac{n\gamma_e^2}{2(1-\mu_e^2)+4\gamma_e/3}\right)\leq 2 \exp\left(-\frac{n\gamma_e^2}{3(1-\mu_e^2)}\right)\]
The definition of $\gamma_e$ in~\eqref{eq:def-gammae} together with~\eqref{eq:bound-gamma-e} and a union bound over $p$ edges  prove the lemma. 
\end{proof}

To prove Lemma~\ref{l:NcorrLemma}, an upper bound on the moment generating function of a random variable with  tight dependency on the variance  is required. This is given in the Lemma~\ref{l:hoeffd2} in Appendix~\ref{app:devBD} which is essentially a restatement of Theorem 3 of Hoeffding \cite{hoeffding1963probability}.

\begin{proof}[Proof of Lemma~\ref{l:NcorrLemma}]
For given pair of nodes $w,\wt$, let $\path_{\T}(w,\wt)=\{e_1, e_2, \cdots, e_d\}$ and use the shorthand $\mu_i=\mu_{e_i}$ and $\wmu_i=\wmu_{e_i}$. We will show that if $n>\max\{25/\gamma^2 \log(4p^2/\delta), 108 e^{2\beta}\log(2p^3/\delta)\}$, then 
\begin{equation} \label{eq:path-error-pairerror}
\P\bigg[\,\Big|\prod_{i=1}^d\wmu_i - \prod_{i=1}^d\mu_i\Big|> \gamma\,\bigg]\leq 2\delta/p^2\,,
\end{equation}
independent of $d$. A union bound over ${p \choose 2}\leq p^2/2$ pairs of nodes $w$ and $\wt$ provides the desired lower bound for $\P\big[\Ncorrg\big]$.

If $d=1$, applying Lemma~\ref{lem:Hoeffding} gives that if $n>2/\gamma^2\log(p^2/\delta)$ then~\eqref{eq:path-error-pairerror} holds. To show~\eqref{eq:path-error-pairerror} for $d\geq 2$, let \begin{equation}\label{def:Ai}
M_i=(\wmu_i - \mu_i) \prod_{j=1}^{i-1} \wmu_j \prod_{j=i+1}^d \mu_j\,.
\end{equation}
Then, $\prod_{i=1}^d\wmu_i - \prod_{i=1}^d\mu_i= \sum_{i=1}^d M_i$.
Let $\A_k=\{e_1, \cdots, e_k\}$ for $k\leq d$. Conditioning on $\EdgErr{(\A_{d-1})}$ (defined in~\eqref{eq:defEdgErr}),
\begin{align}\label{eq:totalP-errorpath}
\P\Big[\,\big|\sum_{i=1}^d M_i\big| > \gamma\,\Big] &\leq\P\Big[\,\big|\sum_{i=1}^d M_i\big|> \gamma \,\big|\, \EdgErr{(\A_{d-1})}\,\Big]
 + \P\Big[\,\big(\EdgErr{(\A_{d-1})}\big)^c\,\Big].
\end{align}
The containment $\EdgErr(\EE_{\T})\subseteq \EdgErr{(\A_{d-1})}$ gives $\P[(\EdgErr{(\A_{d-1})})^c]\leq \P[(\EdgErr(\EE_{\T}))^c]$ and Lemma~\ref{l:EdgErrprb} shows that if $n>108 e^{2\beta} \log(2p^3/\delta)$, then $\P[(\EdgErr(\EE_{\T}))^c]\leq \delta/p^2$.

We bound the first term in the last display via bounding the moment generating function and applying Chernoff bound:   Given a real number $\lambda,$ the recurrence  below upper bounds  $\Ex\Big[\exp\big(\lambda\sum_{i=1}^d M_i\big)\\\,\big|\, \EdgErr{(\A_{d-1})}\Big]$.
For any $k\leq d$,
\begin{align*}
\Ex\bigg[\exp\Big(\lambda\sum_{i=1}^k M_i\Big)&\,\big|\, \EdgErr{(\A_{k-1})}\bigg]\\
&= \Ex\bigg[\Ex\Big[\exp\Big(\lambda\sum_{i=1}^{k-1} M_i\Big)\, \exp(\lambda M_k)\,\big|\,\wmu_1,\cdots,\wmu_{k-1}\Big]\,\Big|\, \EdgErr{(\A_{k-1})} \bigg]\\
&\overset{(a)}{=} \Ex\bigg[\exp\Big(\lambda\sum_{i=1}^{k-1} M_i\Big)\, \Ex\Big[\exp(\lambda M_k)\,\big|\,\wmu_1,\cdots,\wmu_{k-1}\Big]\,\Big|\, \EdgErr{(\A_{k-1})}  \bigg]\\
&\overset{(b)}{\leq} \Ex\bigg[\exp\Big(\lambda\sum_{i=1}^{k-1} M_i\Big)\, \exp\Big\{n f\Big(\frac{1-\mu_k^2}{4}\prod_{j=1}^{k-1} \wmu_j^2\prod_{j=k+1}^{d} \mu_j^2, 2\frac{|\lambda|}{n}\Big)\Big\} \,\Big|\, \EdgErr{(\A_{k-1})} \bigg]\\
&\overset{(c)}{\leq} 
\exp\Big\{n f\Big(\frac{1-\mu_k^2}{4}\prod_{{j=1,\atop j\neq k}}^{d} (\mu_j^2+2\gamma_e), 2\frac{|\lambda|}{n}\Big)\Big\}
 \,\,\Ex\bigg[\exp\Big(\lambda\sum_{i=1}^{k-1} M_i\Big)\,\Big|\, \EdgErr{(\A_{k-1})} \bigg]
\end{align*}
where (a) holds since for $i\leq k-1,$ $M_i\in\sigma(\wmu_1,\cdots,\wmu_{k-1})$.  
(b) is derived by applying Lemma~\ref{l:hoeffd2} to random variable $M_k$ conditional on $\wmu_1,\cdots,\wmu_{k-1}$. Lemma~\ref{l:IndependentError} shows that $\wmu_k$ is independent of $\wmu_1, \cdots, \wmu_{k-1}$. Thus, $M_k = \frac{1}{n}\sum_{i=1}^n Z^{(i)}$ where  $Z^{(i)}$'s are independent and 
\[Z^{(i)}=\begin{cases}(1-\mu_k)\prod_{j=1}^{k-1} \wmu_j \prod_{j=k+1}^d \mu_j, \quad &\text{with probability } \frac{1+\mu_k}{2}\\
-(1+\mu_k)\prod_{j=1}^{k-1} \wmu_j \prod_{j=k+1}^d \mu_j, \quad &\text{with probability } \frac{1-\mu_k}{2}\,.
\end{cases}\] 
 So, $|Z^{(i)}|\leq 2$. Conditional on $\wmu_1,\cdots,\wmu_{k-1}$, the mean of $Z^{(i)}$ is zero,
 and the variance of $Z^{(i)}$ is ${(1-\mu_k^2)}\prod_{j=1}^{k-1} \wmu_j^2\prod_{j=k+1}^{d} \mu_j^2$. 
Inequality (c) holds since conditional on event $\EdgErr{(\A_{k-1})}$ (defined in \ref{eq:defEdgErr}), $\wmu_j^2 \leq \mu_j^2 + 2\gamma_j$ for all $j\leq k-1$. Also, the function $f(u,c)$ is nondecreasing in $u$ for $u,c\geq 0$ (Lemma~\ref{l:hoeffd2}).

 To get a recurrence, 
\begin{align*}
&\hspace{-.1cm}\Ex\bigg[\exp\Big(\lambda\sum_{i=1}^{k-1} M_i\Big)\,\Big|\, \EdgErr{(\A_{k-1})} \bigg] \P[\EdgErr{(\A_{k-1})}]\\
&= \Ex\bigg[\exp\Big(\lambda\sum_{i=1}^{k-1} M_i\Big) \,\Big|\, \EdgErr_{e_{k-1}}, \EdgErr{(\A_{k-2})} \bigg] \P\big[\EdgErr_{e_{k-1}}\,\big|\,\EdgErr{(\A_{k-2})}\big] \P[\EdgErr{(\A_{k-2})}] \\
&= \Ex\bigg[\exp\Big(\lambda\sum_{i=1}^{k-1} M_i\Big),  {1\!\!1}\big[\EdgErr_{e_{k-1}}\big] \,\Big|\, \EdgErr{(\A_{k-2})} \bigg]  \P[\EdgErr{(\A_{k-2})}] \\
&\leq \Ex\bigg[\exp\Big(\lambda\sum_{i=1}^{k-1} M_i\Big) \,\Big|\, \EdgErr{(\A_{k-2})} \bigg]  \P\big[\EdgErr{(\A_{k-2})}\big]\,. 
\end{align*}
By the last two displays,
\begin{align}
&\Ex\bigg[\exp{\Big(\lambda\sum_{k=1}^d M_i\Big)} \,\Big|\, \EdgErr{(\A_{d-1})}  \bigg] \P[\EdgErr{(\A_{d-1})}]\notag \\
&\leq \exp\bigg\{n \sum_{k=1}^d f\Big(\big(1-\mu_k^2\big)\prod_{{j=1\atop j\neq k}}^{d} \big(\mu_j^2+2\gamma_e\big), 2\frac{|\lambda|}{n}\Big)\bigg\}\notag\\
 &\leq \exp\bigg\{nd\,  f\Big(\frac{1}{d}\sum_{k=1}^d\big(1-\mu_k^2\big)\prod_{{j=1\atop j\neq k}}^{d} \big(\mu_j^2+2\gamma_e\big), 2\frac{|\lambda|}{n}\Big)\bigg\}\label{eq:recurrence-errpronpath}
\end{align}
where the second inequality is due to the concavity of  $f(u,c)$ in $u$ for $u,c\geq 0$ (Lemma~\ref{l:hoeffd2}).

 Now, we upper bound the first argument to $f$, $\sum_{k=1}^d(1-\mu_k^2)\prod_{{j=1\atop j\neq k}}^{d} (\mu_j^2+3\gamma_e)$, for any sequence of $\mu_j$ and $\gamma_j\leq (1-\mu_j^2)/6$. Let $x_j=\mu_j^2+2\gamma_j$. Since according to~\eqref{eq:bound-gamma-e}, $2\gamma_j\leq (1-\mu_j^2)/3$, we have that $x_j<1$ and $\frac{3}{2}(1-x_j) = \frac{3}{2}(1-\mu_j^2-2\gamma_j)\geq 1-\mu_j^2$.
Therefore,
 \begin{align}
 \sum_{i=1}^d{(1-\mu_i^2)}\prod_{{j=1\atop j\neq i}}^{d} (\mu_j^2+2\gamma_e)&<\sum_{i=1}^d \frac{3}{2}{(1-x_i)}\prod_{{j=1\atop j\neq i}}^{d} x_j
 \leq 
 \frac{3}{2} d x^{d-1}{(1-x)}\,.\label{eq:maximization-d-errpronpath}
 \end{align}
Symmetry of the function $\sum_{i=1}^d(1-x_i)\prod_{j=1, j\neq i}^{d} x_j$ in $x_j$'s implies that it is maximized when $x_i=x$ for all $i$. Setting the partial derivatives of this term with respect to $x_i$ equal to zero for any $i$ shows this. This gives the second inequality in the above display. 

Now, we find an upper bound for $ \tfrac{3}{2} dx^{d-1}{(1-x)}$ for any $d\geq 2$ and $x<1$. 
The function $x^{d-1}(1-x)$ takes its maximum value at $x^*=1-1/d$. So,
\[\frac{3}{2} dx^{d-1}{(1-x)}\leq \frac{3}{2} \big(1-\frac 1d\big)^{d-1}\leq \frac{3}{4},\]
where the last inequality holds for all $d\geq 2$.
 Since $f(u,c)$ is nondecreasing in $u$ (Lemma~\ref{l:hoeffd2}), plugging this into~\eqref{eq:recurrence-errpronpath} and~\eqref{eq:maximization-d-errpronpath} gives
\begin{align*}\label{eq:maximization-x-errpronpath}
\Ex\bigg[\exp\Big(\lambda\sum_{k=1}^d M_i\Big) \,\Big|\, \EdgErr{(\A_{d-1})}  \bigg] \P\Big[\EdgErr{(\A_{d-1})}\Big] &\leq  \exp\Big\{nd\, f\Big(\frac{3}{4 d}, 2\frac{|\lambda|}{n}\Big)\Big\}\,.
\end{align*}
With $n>108 e^{2\beta} \log(4p)$ (as assumed in the statement of lemma), $\P\Big[\big(\EdgErr{(\A_{d-1})}\big)^{c}\Big]\leq 1/2$ (Lemma~\ref{l:EdgErrprb}) which gives 
\[\Ex\bigg[\exp\big(\lambda\sum_{i=1}^d M_i\big)\,\big|\, \EdgErr{(\A_{d-1})}\bigg]\leq 2\exp\Big\{nd\, f\Big(\frac{3}{4 d}, 2\frac{|\lambda|}{n}\Big)\Big\}.\]
This yields
\begin{align*}
\P\Big[\,\sum_{i=1}^d M_i\geq \gamma\,\big|\, \EdgErr{(\A_{d-1})} \Big] & \leq 2\min_{\lambda>0} \exp\Big\{nd\, f\Big(\frac{3}{4 d}, 2\frac{\lambda}{n}\Big)- \lambda\gamma\Big\} \\
& \overset{(a)}{\leq} 2 \left\{ \big(1+\frac{2\gamma}{3}\big)^{-\big(1+\frac{2\gamma}{3}\big)\frac{3}{3+4d}}  \big(1-\frac{\gamma}{2d}\big)^{-\big(1-\frac{\gamma}{2d}\big) \frac{4d}{4d+3}} \right\}^{nd}\\
&\overset{(b)}{\leq} 2\exp(-n\gamma^2/116)\,,
\end{align*}
where (a) uses $\lambda=\lambda^*$ with
\[\lambda^*= \frac{n/2}{1+3/(4d)}\log\frac{1+2\gamma/3}{1-\gamma/(2d)}\,,\]
and (b) uses  the following bound valid for any $d\geq 2$:
\begin{align*}
\frac{3d}{3+4d}\Big(1+\frac{2\gamma}3\Big)&\log\Big(1+\frac{2\gamma}3\Big) + \frac{4d^2}{3+4d}\Big(1-\frac{\gamma}{2d}\Big)\log\Big(1-\frac{\gamma}{2d}\Big) \\
&\overset{(i)}{\geq}\frac{3d}{3+4d}\Big(1+\frac{2\gamma}3\Big)\Big(\frac{2\gamma}3-\frac{2\gamma^2}9\Big) - \frac{4d^2}{3+4d}\Big(1-\frac{\gamma}{2d}\Big)\log\Big(1+\frac{\gamma/2d}{1-\gamma/2d}\Big)  \\
&\overset{(ii)}{\geq} \frac{d}{3+4d}\big[\,3\Big(1+\frac{2\gamma}3\Big)\Big(\frac{2\gamma}3-\frac{2\gamma^2}9\Big) -2\gamma \,\big] \overset{(iii)}{\geq}\gamma^2/25\,.
\end{align*}
Here we used $x\geq \log(1+x)\geq x-x^2/2$ for $x>0$ in (i) and (ii), and $\gamma<1$ gives $(iii)$.

Similarly, $\P\Big[\,\sum_{i=1}^d M_i\leq -\gamma\,\big|\, \EdgErr{(\A_{d-1})}\Big]\leq 2\exp(-n\gamma^2/25)$.
Plugging these into~\eqref{eq:totalP-errorpath}, we get that if $n>\max\{25/\gamma^2 \log(4p^2/\delta), 108 e^{2\beta}\log(2p^3/\delta)\}, $ then~\eqref{eq:path-error-pairerror} holds. 
\end{proof}

\section{Two trees lemma} \label{s:AppendixTwoTrees}
\begin{lemma}
	Let $\T_1$ and $\T_2$ be two spanning trees on set of nodes $\VV$. Let $w,\wt$ be a pair of nodes  such that $\path_{\T_1}(w,\wt)\neq \path_{\T_2}(w,\wt)$. Then there exists a pair of edges $f\triangleq (u,\ut)\in \path_{\T_1}{(w,\wt)}$ and $g\triangleq (v,\vt)\in \path_{\T_2}{(w,\wt)}$ such that 
	\begin{enumerate}
		\item[(i)] 	$f\notin \path_{\T_2}{(w,\wt)}$ and $g\notin \path_{\T_1}{(w,\wt)}$
		\item[(ii)]  $f\in \path_{\T_1}{(v,\vt)}$ and $g\in \path_{\T_2}{(u,\ut)}$
	\end{enumerate}
\end{lemma}
Since $f\in\path_{\T_1}(w,\wt) \cap \path_{\T_1}{(v,\vt)}$, $w$ and $\wt$ (and respectively $v$ and $\vt$) are in different subtrees of $\T_1$ after removing edge $f$, one can label the end points of the edges $f=(u,\ut)$ and $g=(v,\vt)$ such that $u,v\in \subT{\T_1}{f}{w}$ and $\ut,\vt\in \subT{\T_1}{f}{\wt}$ (See Figure \ref{f:proof}).
		
\begin{proof}
We assume that $\T_1$ and $\T_2$ are two spanning trees with no common edges over a set of nodes $\VV$.  Otherwise, if $\T_1$ and $\T_2$ have any common edges, we can contract them and construct a new pair of spanning trees over set of nodes $\VV'$ in a way that preserves paths. Note that by assumption $\path_{\T_1}(w,\wt)\neq \path_{\T_2}(w,\wt)$, so $w\,,\wt \in\VV'$ are not merged in the contraction process. 

Let $u_1,\cdots,u_{K+1}$ be the vertices of $\path_{\T_1}(w,\wt)$ in order with $ u_1 = w$ and $u_{K+1} = \wt$. Similarly, let $v_1,\cdots,v_{L+1}$ be the vertices of $\path_{\T_2}(w,\wt)$ with $v_1 = w$ and $v_{L+1} = \wt$. Now, for $1 \leq k \leq K,$ let $$h_{\T_1}(k)= \min\{t: (u_k, u_{k+1}) \in \path_{\T_1}(v_t, v_{t+1})\}\,. $$

We first show that $h_{\T_1}(k)$ is well-defined (i.e., for all $1\leq k\leq K$, there exists a $1\leq t\leq L$ such that $(u_k, u_{k+1}) \in \path_{\T}(v_t, v_{t+1})$). For each $1\leq k \leq K,$ let $\SS_{k}=\subT{\T_1}{(u_k,u_{k+1})}{u_k}$ and $\widetilde{\SS}_{k}=\subT{\T_1}{(u_k,u_{k+1})}{u_{k+1}}$. Note that $w\in\SS_{k}$ and $\wt\in\widetilde{\SS}_k$ for all $k$. Also, $(u_k, u_{k+1}) \in \path_{\T_1}(v_t, v_{t+1})$ for some $0\leq t \leq L\,$  if and only if $v_t\in\SS_{k}$ and $v_{t+1}\in\widetilde{\SS}_{k+1}$. For any given $k$, if such $t$ does not exist, then either $v_{t'}\in\SS_{k}$ or  $v_{t'}\in\widetilde{\SS}_{k}$ for all $1\leq t' \leq L+1$, which is a contradiction as $v_1\in\SS_{k}$ and $v_{L+1}\in\widetilde{\SS}_k$.  Hence $h_{\T_1}(k)$ is well-defined for all $k$. Equivalently, we could define $h_{\T_1}(k)=\min\{t:v_{t+1}\in\widetilde{\SS}_k\}$. As $\widetilde{\SS}_{k}\subseteq \widetilde{\SS}_{k+1}$, the function $h_{\T_1}(k)$  is non-decreasing in $k\,.$

Similarly, the function 
$h_{\T_2}(l)= \min\{t: (v_l, v_{l+1}) \in \path_{\T_2}(u_t, u_{t+1})\}\, $
is well-defined and non-decreasing. 
To prove the result, we first show that there exists a pair of integers $l^*$ and $k^*$ such that $h_{\T_1}(k^*)=l^*$ and $h_{\T_2}(l^*)=k^*$. To do so, let function $h(l): \{1,\cdots,L\}\to \{1,\cdots,L\}$ be $h(l)=h_{\T_1}(h_{\T_2}(l))$. This is a nondecreasing function over its domain. We now show that $h(l) $ has a fixed point. Assume that $h(1)>1$  and $h(L)<L$, since otherwise the result is proven. Define $m=\min\{t:h(t)\leq t\}$ and note that $1<m<L$.  Thus, $h(m)\leq m$ while $h(m-1)>m-1$. Since $h(l)$ is non-decreasing, we have $m-1 < h(m-1) \leq h(m)\leq m$, which implies that $h(m)=m$. 

Let $l^*$ be a fixed point of $h(l)$. According to the definition of $h(l)$, the pair of integers $l^*$ and $k^*=h_{\T_2}(l^*)$ satisfies $h_{\T_1}(k^*)=l^*$. Hence, the edge pair $(u_{k^*},u_{k^*+1})\in\path_{\T_1}(v_{l^*},v_{l^*+1})$ and $(v_{l^*},v_{l^*+1})\in\path_{\T_2}(u_{k^*},u_{k^*+1})$ satisfies the properties given in the Lemma. 
\end{proof}
\section{Bounds for risk (proof of Corollaries 3.5 and 3.6)}
\label{s:risk-upper-bd}

We provide a proof for Corollaries \ref{t:RiskUppererBound} on the risk which is a direct application of Theorem \ref{t:CLLowerBound}.
By Theorem \ref{t:CLLowerBound}, for any $0\geq \eta\geq 1$, given $n$ samples, $ \LL^{(2)}(P,\Pi_{\Tc}(\Ph))\geq \eta$ with probability at most $p\exp(-Cne^{-2\beta}) + p/\eta \exp(-Cn\eta^2)$. Hence,

\begin{eqnarray*}
\Ex\left[\LL^{(2)}(P,\Pi_{\Tc}(\Ph)\right] & = &\int_{0}^1 \P\left[\LL^{(2)}(P,\Pi_{\Tc}(\Ph)\geq \eta\right]\, d\eta \\
&\stackrel{(a)}{\leq} & 
\int_{0}^1  \min\{1, p\exp(-Cne^{-2\beta}) + \frac{p}{\eta}\exp(-Cn\eta^2)\}\, d\eta\\
& \stackrel{(b)}{\leq} & \eta^* + \int_{\eta^*}^1   p\exp(-Cne^{-2\beta}) + p \exp(-Cn\eta^2)\, d\eta\\
& {\leq} & \eta^* +  p\exp(-Cne^{-2\beta}) + p\int_{\eta^*}^{\infty}   \exp(-Cn\eta^2)\, d\eta\\
&  \stackrel{(c)}{\leq} & \eta^* +  p\exp(-Cne^{-2\beta}) + p \sqrt{\frac{\pi}{C n}} \exp(-Cn\eta^{*^2}/2)\\
& \stackrel{(d)}{\leq} & 3 \eta^*\,,
\end{eqnarray*}
where (a) is a direct application of Theorem \ref{t:CLLowerBound}. Inequality (b) is true for any $0<\eta^*\leq 1$. (c) uses the inequality $1/\sqrt{2\pi} \int_{a}^{\infty} \exp(-x^2/2)\,dx \leq \exp(-a^2/2)$. (d) holds for $\eta^*= p\exp(-Cne^{-2\beta}) + C' \sqrt{\frac{ \log p}{{n}}}$ with a constant $C'.$

\section{Accurate $k$-wise marginals} 
\label{sec:MarginalK}
We first describe our result, expressing marginals in terms of correlations, and then delineate implications to learning Ising models in order to make subsequent predictions from partial observations.

Let $\SS\subset \VV$ be a subset of variables. We are interested in the marginal $P(x_\SS)$ for an Ising model $P(x)$ of the form~\eqref{e:Ising}. Note that the variables have mean zero by symmetry of the model, hence $P(x_i)=\frac12$ for each $i\in \VV$. When $|\SS|=2\,,$ the marginal over any pair of nodes $\SS=\{i,j\}$ can be expressed as $P(X_i=x_i,X_j=x_j)=\frac{1}{4}[1+x_i x_j \mu_{ij}]$ for $x_i,x_j\in\{-,+\}\,.$

We seek a simple closed form for the marginal distribution $P(x_\SS)$ as a function of pairwise correlations between pairs of nodes $i,j\in\SS\,.$ By the previous paragraph, we have already accomplished this when $|\SS|=2\,.$ The general case $|\SS|>2 $ is (as one might expect) rather more complicated and involves a certain maximization over matchings. 

Let ${\VV\choose 2}$ denote the collection of pairs of nodes in $\VV$. 
For $\pr\in{\VV\choose 2}$, if $\pr=\{i,j\}$, then we also write the correlation $\mu_{\pr}=\Ex_{P}{X_iX_j}\,.$

\begin{defn}
For subset of nodes $\CC\subseteq \VV$ of even cardinality $|\CC|=2t$, we define the set $\Mc_{\CC}$ to be the collection of all partitions of $\CC$ 
into $t$ disjoint pairs of nodes, i.e. the set of all perfect matchings
in the complete graph with node set $\CC$:
$$ \mathcal{M}_\CC=\big\{M=\{\pr_1,\cdots,\pr_t\}:\, \pr_l\in{\CC\choose 2}\text{ for each }l \in [t],\, \pr_l\cap \pr_k=\varnothing \text{ for } l \neq k\big\}\,.$$
\end{defn}
\begin{defn} Let $t\geq 1$ be a positive integer and $\CC\subseteq \VV$ be of size $|\CC|=2t\,.$ The set-valued function $G(\CC)$ is defined as the set of matchings $M$ in $\mathcal{M}_\CC$ that maximize the product of correlations for elements in $M$: 
\begin{equation} \label{e:defG}
G(\CC) = \argmax_{M\in\mathcal{M}_{\CC}} \prod_{m\in M}|\mu_{\pr}|\,.
\end{equation} 
We define $G^*(\CC)$ to be an arbitrary element in $G(\CC)$, and $G(\varnothing)=\varnothing$ and $G^*(\varnothing)=\varnothing\,.$ We also use the notation $G^*(\CC,P)$ to be optimal matching $G^*(\CC)$ under the distribution $P$ to explicitly determine the underlying distribution. 
\end{defn}
We emphasize that often the maximum in \eqref{e:defG} is not uniquely achieved and the set $G(\CC)$ can be large. In what follows we use the convention that the product $\prod_{m\in \mathcal{I}}\mu_m=1$ whenever $\mathcal{I}=\varnothing$. 
Our main result is the following:

\begin{theorem}\label{t:decomposition}
If $P(x)$ is the probability mass function for an Ising model~\eqref{e:Ising} on any tree $\T$, then the marginal $P(x_\SS)$ over the subset $\SS\subseteq\VV$ can be written as:
\begin{equation}\label{e:decomposition}
P(x_{\SS})= {2^{-|\SS|}} \sum_{t=0}^{\lfloor |\SS|/2\rfloor} \,\sum_{\substack{\CC\subset\SS\\ |\CC|=2t}} \,\,\prod_{c\in\CC}x_c\prod_{\pr\in G^*(\CC)} \mu_{\pr}\,.
\end{equation}
\end{theorem}

\begin{proof}
For subset of variables $\SS\subseteq\VV$, the marginal distribution $P_{\SS}$ is
\begin{align*}
P_\SS(x_{\SS}) & = \sum_{x_{\SS^c}} P(x) \stackrel{(a)}{=} \sum_{x_{\SS^c}} \prod_{(i,j)\in\EE}\frac{1+\mu_{ij} x_i x_j}{2}  \stackrel{(b)}{=} \frac{1}{2^p} \sum_{x_{\SS^c}} [1+ \sum_{\F\subseteq\EE,\F\neq\varnothing} \prod_{(i,j)\in\F}\mu_{ij} x_i x_j]\\
& \stackrel{(c)}{=} \frac{1}{2^{|\SS|}} [1+ \sum_{\F\subseteq\EE, \atop \F\in\Omega(\SS),  \F\neq \varnothing} \prod_{(i,j)\in\F}\mu_{ij} x_i x_j]
\end{align*}
where (a) uses the Equation~\eqref{eq:TreeFac} for $P\in\PP_\T$. In (b) we expand the product term over all the edges in the tree. In (c), $\Omega(\SS)\subseteq 2^{\EE}$ is defined such that for  $\F\subseteq\EE$, if $\F\in\Omega(\SS)$, then for all nodes $i\in\SS^c$, there are even number of edges $e\in\F$ that have $i$ as their end point. Equivalently, any $\F\in\Omega(\SS)$ is the collection of edges in the set of edge-disjoint paths between distinct pairs of nodes in $\SS$. 
\end{proof}

The utility of this expression is in part due to the fact that it
depends on the tree topology $\T$ only through the optimization implicit in $G(\CC)$, which depends only on end-to-end correlations between variables in $\CC\subseteq\SS\,.$ In particular, the expression is agnostic to the topology over variables not in $\SS$: even the \emph{number} of such variables is irrelevant. This is important because  there can be several (non-isomorphic) subtrees which give the same marginal over $\SS$. This means that it is in general impossible to detect the tree topology just by observing end-to-end correlations, but this expression is robust to this. 

\begin{lem} \label{l:robustMatching}
For any set $\CC\subseteq\VV$ with $|\CC|=2t\,,$ and any pair of tree structured Ising distributions $P$ and $Q$ as in ~\eqref{e:Ising} (with possibly different underlying tree), let $\mu_{ij}=\Ex_P X_i X_j$ and $\wtmu_{ij}=\Ex_Q X_i X_j\,.$
If $|\mu_{ij}-\wtmu_{ij}|\leq \epsilon$ for all $i,j\in\CC\,,$ then $$\Big|\mkern-10mu\prod_{m\in G^*(\CC,P)}\mkern-10mu\mu_{m}-\mkern-10mu\prod_{m'\in G^*(\CC,Q)}\mkern-10mu\wtmu_{m'}\Big|\leq 2t\epsilon\,.$$
\end{lem}

\begin{proof}
First, we will show that for all $M\in\Mc_{\CC}\,,$ $\prod_{m\in M}\mu_m$ takes the same sign (\textit{i.e., for $M,M'\in\Mc_{\CC}\,, \prod_{m\in M}\mu_m \prod_{m'\in M'}\mu_{m'}\geq 0$}). Similar statement holds for $\wtmu\,.$
To do so, note that for all $M,M'\in\Mc_{\CC}\,,$ one can make a sequence of matchings $M_0=M, M_1, \cdots, M_r=M' $ such that $M_r$ and $M_{r+1}$ differ only at two pairs of nodes (\textit{i.e.,} There exists $(i,j), (k,l)\in M_r$ such that $(i,k), (j,l)\in M_{r+1}$).  
The product of correlations on $M_r$ and $M_{r+1}$ have the same sign (\textit{i.e.,} $\prod_{m\in M_r}\mu_m \prod_{m\in M_{r+1}}\mu_m \geq 0$). The same argument applies for all $r$ and hence $\prod_{m\in M}\mu_m \prod_{m'\in M'}\mu_{m'}\geq 0\,.$

Next, we show that if $|\mu_{ij}-\wtmu_{ij}|\leq \epsilon$ for all $i,j\in\CC$, then for any $M=\{m_1, m_2, \cdots, m_t\}\in\Mc_{\CC}\,, |\prod_{m\in M}\mu_m - \prod_{m'\in M}\wtmu_{m'}|\leq t\epsilon\,.$  
\begin{align*}
\Big|\prod_{i=1}^t\mu_{m_i}-\prod_{i=1}^t\wtmu_{m_i}\Big| &= \Big|\sum_{i=1}^t \prod_{j=1}^{i-1}\mu_{m_j}\prod_{j'=i+1}^{t}\wtmu_{m_{j'}} (\mu_i-\wtmu_i)\Big|\\
&\leq \sum_{i=1}^t \Big|\prod_{j=1}^{i-1}\mu_{m_j}\prod_{j'=i+1}^{t}\wtmu_{m_{j'}} (\mu_i-\wtmu_i)\Big|\leq \sum_{i=1}^t \big|\mu_i-\wtmu_i\big|\leq t\epsilon\,.
\end{align*}
Let's define $$a =\mkern-18mu\prod_{m\in G^*(\CC,P)}\mkern-10mu\mu_{m}, 
\quad a'=\mkern-18mu\prod_{m\in G^*(\CC,Q)}\mkern-10mu\mu_{m}, 
\quad b =\mkern-18mu\prod_{m\in G^*(\CC,P)}\mkern-10mu\wtmu_{m}, 
\quad b' = \mkern-18mu\prod_{m\in G^*(\CC,Q)}\mkern-10mu\wtmu_{m}\,.$$  Hence, we showed that $a a'\geq 0\,, b b'\geq 0, |a-b|\leq t\epsilon $ and $ |a'-b'|\leq t\epsilon\,.$ We also know that $|a|\geq |a'|$ and $|b'|\geq |b|$ according to the definition of $G^*(\CC,P)$ and $G^*(\CC,Q)\,.$ To bound $|a-b'|\,,$ we study two different cases: Either $a b \geq 0$ in which case $|a-b'| = ||a|-|b'||=\leq \max\{|a-b|,|a'-b'|\}\leq t\epsilon\,$ Or $a b \leq 0$ in which case $|a-b'|=|a|+|b'|\leq |a-b|+|a'-b'|\leq 2 t \epsilon\,.$
\end{proof}

\begin{prop}\label{prop:ssTVwithcorr}
Let $\T=(\VV,\EE)$ be a tree and $P$ an Ising model on $\T$. 
Let $Q$ be an Ising model on tree $\T'=(\VV,\EE')\,. $ If $|\Ex_{P}{X_i X_j}-\Ex_{Q}{X_i X_j}|\leq \eta$ for all $i,j\in\VV\,$ (i.e., $\sstv^{(2)} (P,Q)\leq \eta/2\,$), then $\sstv^{(k)} (P,Q)\leq k 2^k \eta\,.$
\end{prop}
\begin{proof}

For $i,j\in\VV\,,$ let $\mu_{ij}\triangleq \Ex_{P}X_iX_j$ and $\wtmu_{ij}\triangleq \Ex_{P}X_iX_j\,.$ We also use the notation $G^*(\CC,P)$ to be optimal matching $G^*(\CC)$ under the distribution $P\,.$

To bound $\SS\subseteq\VV$ with $|\SS|=k\,,$ we use Lemma \ref{l:robustMatching} to get: \begin{align*}
d_\mathrm{TV} (P_\SS,Q_\SS) & = \sum_{x_\SS\in\{-,+\}^k} |P(x_\SS)-Q(x_\SS)|\\
& 
 = \frac{1}{2^k}\mkern-5mu\sum_{x_\SS\in\{-,+\}^k} \bigg| \sum_{t=0}^{\lfloor \frac{k}{2}\rfloor} \,\sum_{\substack{\CC\subset\SS \\ |\CC|=2t}} \,\,\prod_{c\in\CC}x_c\prod_{\pr\in G^*(\CC,P)} \mu_{\pr} - \sum_{t=0}^{\lfloor \frac{k}{2}\rfloor} \,\sum_{\substack{\CC\subset\SS\\ |\CC|=2t}} \,\,\prod_{c\in\CC}x_c\prod_{\pr\in G^*(\CC,Q)} \wtmu_{\pr} \bigg|
  \notag\\
& \leq  \sum_{t=0}^{\lfloor \frac{k}{2}\rfloor} \,\sum_{\substack{\CC\subset\SS\\ |\CC|=2t}}\frac{1}{2^k}\mkern-5mu\sum_{x_\SS\in\{-,+\}^k} \,\,\bigg| \prod_{c\in\CC}x_c \bigg| \,\cdot\,\bigg|\prod_{\pr\in G^*(\CC,P)} \mu_{\pr} - \prod_{\pr\in G^*(\CC,Q)}\wtmu_{\pr} \bigg|
\\
& =  \sum_{t=0}^{\lfloor \frac{k}{2}\rfloor} \,\sum_{\substack{\CC\subset\SS\\ |\CC|=2t}} \,\,\bigg|\prod_{\pr\in G^*(\CC,P)} \mu_{\pr} - \prod_{\pr\in G^*(\CC,Q)}\wtmu_{\pr} \bigg|\\
&\leq \sum_{t=0}^{\lfloor \frac{k}{2}\rfloor} \,\sum_{\substack{\CC\subset\SS\\ |\CC|=2t}} 2t\eta \leq \sum_{t=0}^{\lfloor \frac{k}{2}\rfloor} {k \choose 2t} 2t\eta \leq k2^k \eta
\,.\qed
\end{align*}
\end{proof}

%
%

Note that learning the tree $\Tc$ from $n$ samples using the variant of Chow-Liu algorithm has time complexity of $O((n+1)p^2)$. Finding low-order marginals and posteriors of any subset of variables on a tree-structured distribution takes $O(p)$ operations using belief-propagation algorithm. 





\section{Numerical simulations}
\label{sec:NumSimSUpp}
In this section we give numerical numerical simulations on the following setups: 1) comparison of ssTV loss for the forest approximation algorithm with the Chow-Liu algorithm; 2) approximation of higher-order marginals, beyond pairwise; 3) tree-structured Ising model with external field; 4) model misspecification.

\paragraph{Forest approximation algorithm}
\label{sec:forestNumSim}
We simulate the forest approximation algorithm introduced in Section~\ref{s:truncation} in the paper. In this simulation, the data is generated from a distribution factorized as Equation~\eqref{e:Ising} using a tree with $p=31$ nodes and maximum edge strength $\beta=1.5$ and various values of $\alpha$. 

The forest approximation algorithm implemented in our simulations is slightly different from the one introduced in Section~\ref{s:truncation}. The algorithm introduced in Section~\ref{s:truncation} chooses a threshold, $\tau(\eps)=4\eps/\sqrt{1-\tanh\beta}$ for $\eps=\sqrt{2/n \log(2p^2/\delta)}$ as function of the parameters of the model and removes the edges added by the Chow-Liu algorithm whose weights are smaller than the threshold $\tau(\eps)$. This process guarantees that with probability $1-\delta$ the reconstructed forest is a subgraph of the original tree.
 In our simulations, we use a slightly different approach in choosing the value of the threshold.  Instead of fixing (a conservative) value for the threshold, for each sample and corresponding run of the algorithm, 
the threshold  used is the best choice for that specific sample. 
  The threshold $\widetilde\tau^{(i)}$ for sample $i$ is the minimum number such that removing the 
 edges with weights at most $\widetilde\tau^{(i)}$  from the  the maximum weight tree results in
 a subgraph of the original tree. 
We claim that the $\LL^{(2)}$ loss calculated for this algorithm (which uses aid from an oracle) is strictly smaller than the loss according to $\widetilde{d}$ defined in~\eqref{eq:forest-loss}. To see that, note that  if $\tau(\eps)<\widetilde\tau^{(i)}$, then the output of the original forest approximation algorithm is not a subgraph of the original tree (in which case $\widetilde{d}=1$). If $\tau(\eps)\geq\widetilde\tau^{(i)}$, then the original forest approximation algorithm (which uses $\tau(\eps)$) might remove more correct edges from the Chow-Liu tree than necessary. We verified 
that in each simulated instance $\widetilde\tau^{(i)}$ yields better performance in terms of $\LL^{(2)}$
(We ran the original forest approximation  algorithm introduced in Section~\ref{s:truncation} for various values of $\delta$ and this was the case for all $\delta$). 

\begin{figure}[h]
  \centering
\includegraphics[width=.5\linewidth]{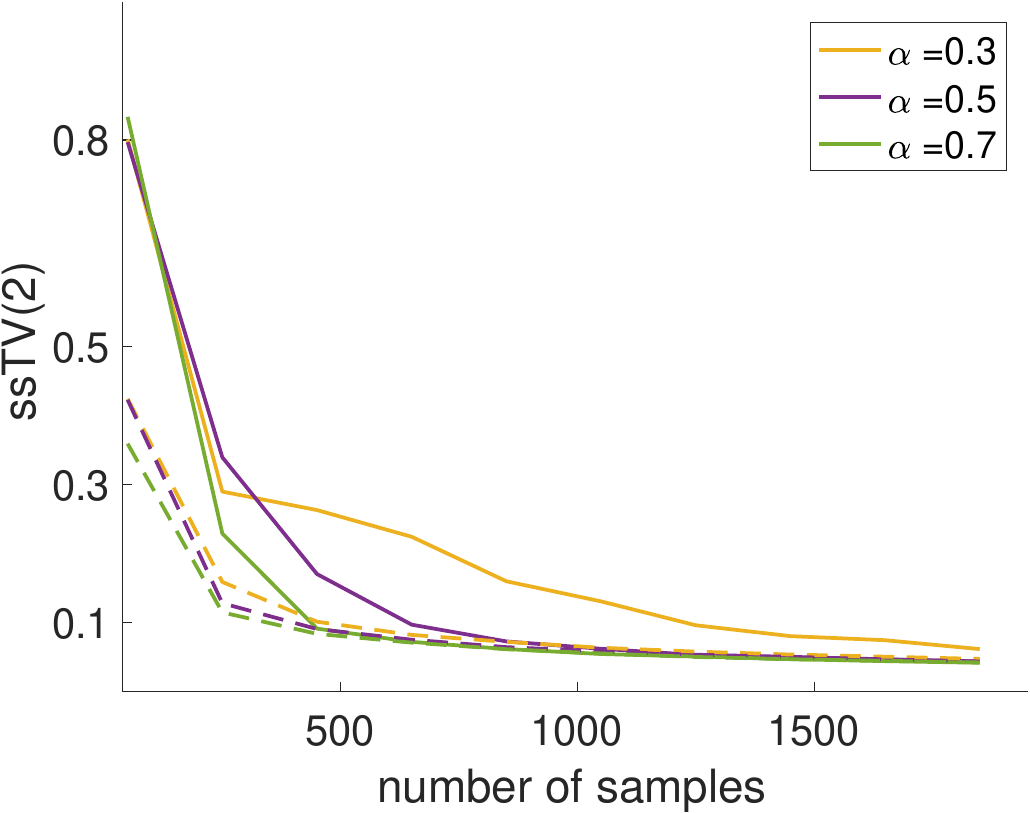}
\caption{The generative distribution is factorized as Equation~\eqref{e:Ising} on a tree with $p=31$ nodes, $\beta=1.5$ and various values of $\alpha$. The solid lines depict the $\LL^{(2)}$ loss of the forest approximation algoirhtm. The dashed lines are the $\LL^{(2)}$ loss of the Chow-Liu algorithm.}\label{fig:ForestApproxssTV}   
\end{figure}

In
 Figure~
\ref{fig:ForestApproxssTV}, 
 the $\LL^{(2)}$ loss for the output of this forest approximation algorithm is compared with the Chow-Liu algorithm for various values of $\alpha$. Again, note that the output of our variation of oracle-aided forest approximation algorithm is always a subgraph of the original tree and its $\LL^{(2)}$ loss is no larger than the original  forest approximation algorithm (\textit{i.e.}, it is more powerful as it uses the help from the oracle). Figure~\ref{fig:ForestApproxssTV} suggests that the $\LL^{(2)}$ loss entailed by the  Chow-Liu algorithm is smaller than both the original forest approximation and the variation used in the simulations. This confirms our main message of the paper that learning possibly incorrect tree by the Chow-Liu algorithm  performs better than the forest approximation algorithm in terms of $\LL^{(2)}$ loss.

\paragraph{Accurate $k$-wise marginals for general $k$}
\label{sec:generalKNumSim}
 Theorem~\ref{t:mainResultIntro} bounds the number of samples necessary to guarantee small $\LL^{(2)}$ loss using the Chow-Liu algorithm in learning tree-structured Ising models with no external field. Corollary~\ref{cor:margK} states the preliminary results to bound the $\LL^{(k)}$ loss for general $k$. Next,
we use numerical simulations to bound  the $\LL^{(k)}$ for general values of $k$ as a function of number of samples.
The generative distribution $P(x)$ is factorized according to~\eqref{e:Ising} on a random tree with $p=7$ nodes, $\alpha = 0.8$ and $\beta = 2$. Figure~\ref{fig:MargK} depicts the $\LL^{(k)}(P,Q)$ as a function of number of samples where $Q(x)$ is the output of the Chow-Liu algorithm. 

\begin{figure}[h]
  \centering
\includegraphics[width=.5\linewidth]{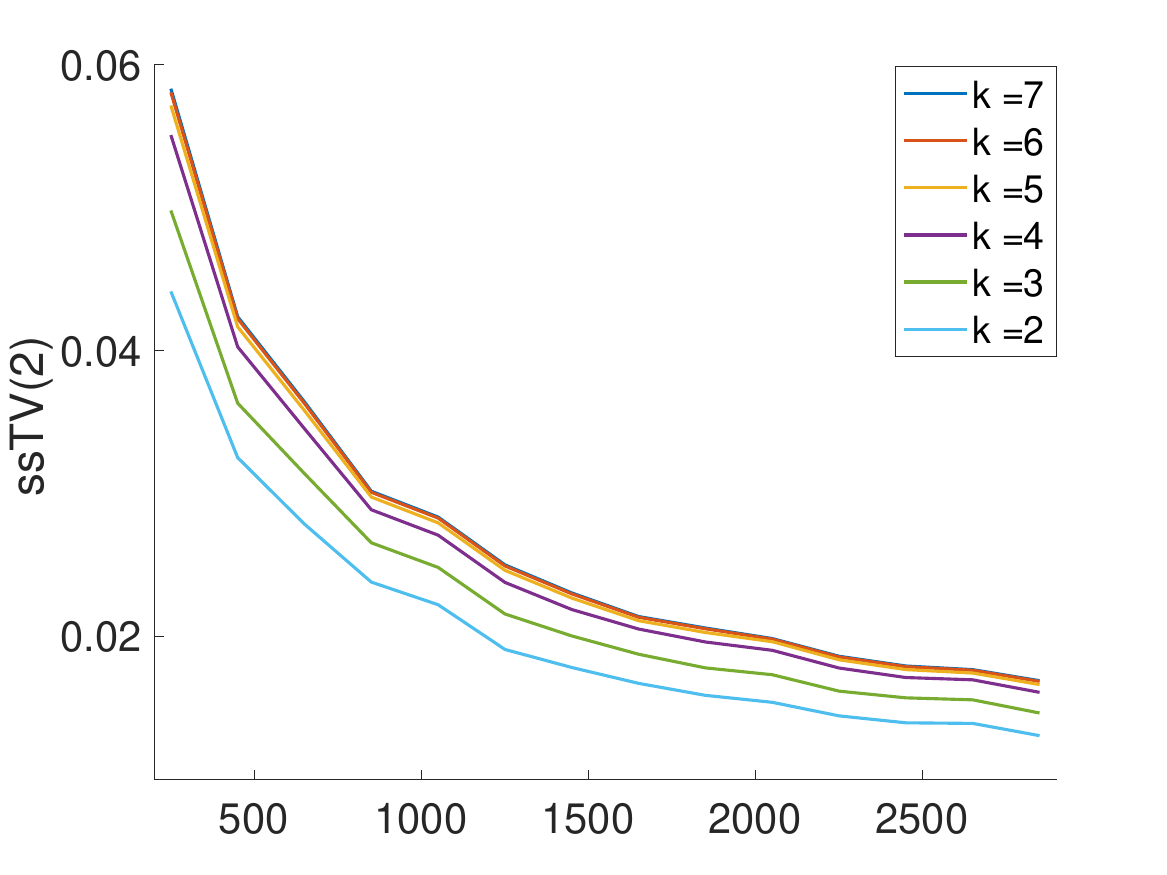}
\caption{The generative distribution $P(x)$ is factorized according to~\eqref{e:Ising} on a random tree on  $p=7$ nodes, with $\alpha = 0.8$ and $\beta = 2$. The distribution $Q(x)$ is the output of the Chow-Liu algorithm. This figure depicts the
  $\LL^{(k)}(P,Q)$ loss as a function of number of samples for various values of $k$.}\label{fig:MargK}   
\end{figure}

\paragraph{Ising model with external field}
\label{sec:ExternalNumSim}
We use the numerical simulations to study the performance of Chow-Liu algorithm on tree-structured Ising models in presence of external field.  
In this case, the singleton marginals over each variable are non-uniform. 
In our simulation, we generate a random tree and choose random values of parameters to make a generative tree-structured distribution $P(x)$.
The output of Chow-Liu algorithm is a tree structured Ising model $Q(x)$. Figure~\ref{fig:externalField} depicts the $\LL^{(2)}(P,Q)$ loss as function of number of samples used by the Chow-Liu algorithm. It conforms the main message of the paper which states that the output of the Chow-Liu algorithms has accurate pairwise marginals and can be used for subsequent predictions.  

\begin{figure}[H]
  \centering
\includegraphics[width=.5\linewidth]{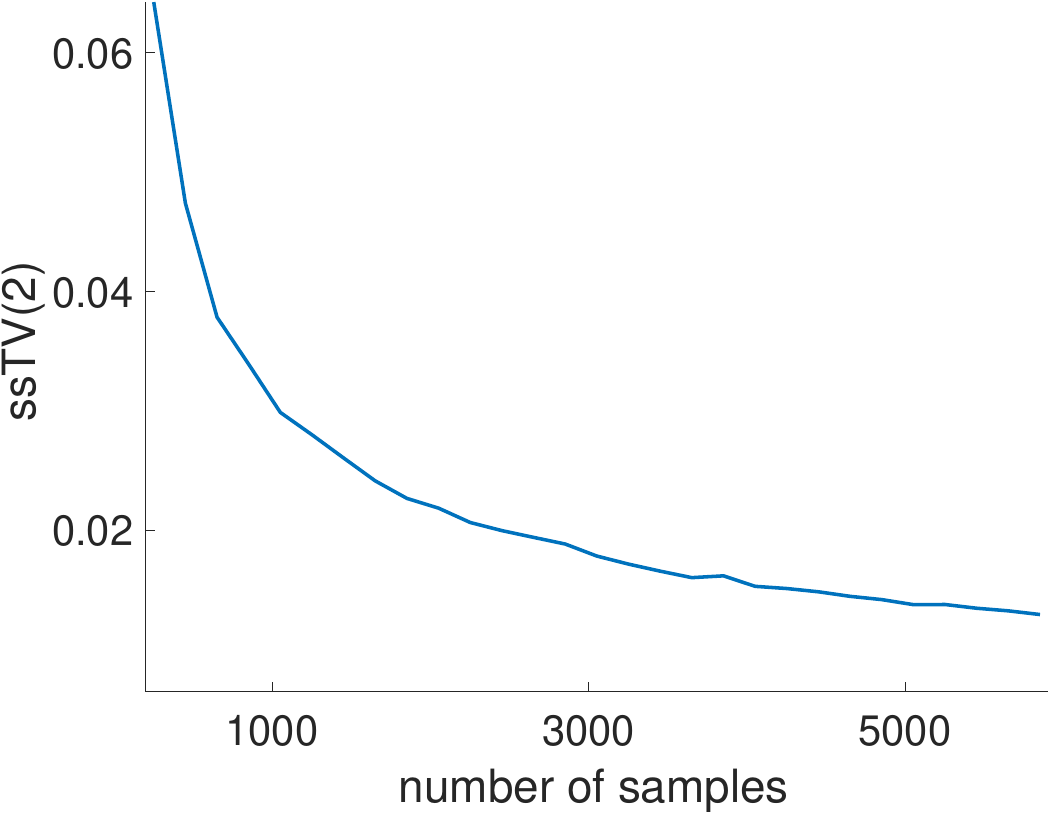}
\caption{The generative distribution is a tree structured Ising model with external field over a random tree and random values of parameters. The distribution $Q(x)$ is the output of Chow-Liu algorithm. This figure depicts $\LL^{(2)}(P,Q)$ loss as function of number of samples.}\label{fig:externalField}   
\end{figure}

\paragraph{Misspecified model}
\label{sec:misspecNumSim}
We study the effect of misspeficiation in the modelling on the performance of Chow-Liu algorithm using numerical simulations. The distribution $\Pt(x)$ is a tree structured Ising model factorized as Equation~\eqref{e:Ising} over a tree (uniformly randomly chosen on the set of trees on $p=6$ nodes) with $\alpha=0.1$ and $\beta=0.8$. The generative distribution $P$ is constructed by adding the random isotropic offset to the distribution $\Pt(x)$. 
Hence, the distribution $P(x)$ used to generate samples is not a tree structured Ising model  in this simulation.
 Let $\gamma\triangleq \LL^{(2)}(P,\Pt)$. The output of Chow-Liu algorithm on $n$ samples is a tree structured distribution $Q(x)$. Figure~\ref{fig:misspec} depicts the $\LL^{(2)}(P,Q)$ loss for various values of $\gamma$. As the number of samples $n$ grows, $\LL^{(2)}(P,Q)$ decays and $\lim_{n\to\infty} \LL^{(2)}(P,Q) =\gamma$.   The simulation results suggest that the Chow-Liu algorithm is robust to the misspecification in the model. 
 
\begin{figure}[h]
  \centering
\includegraphics[width=.5\linewidth]{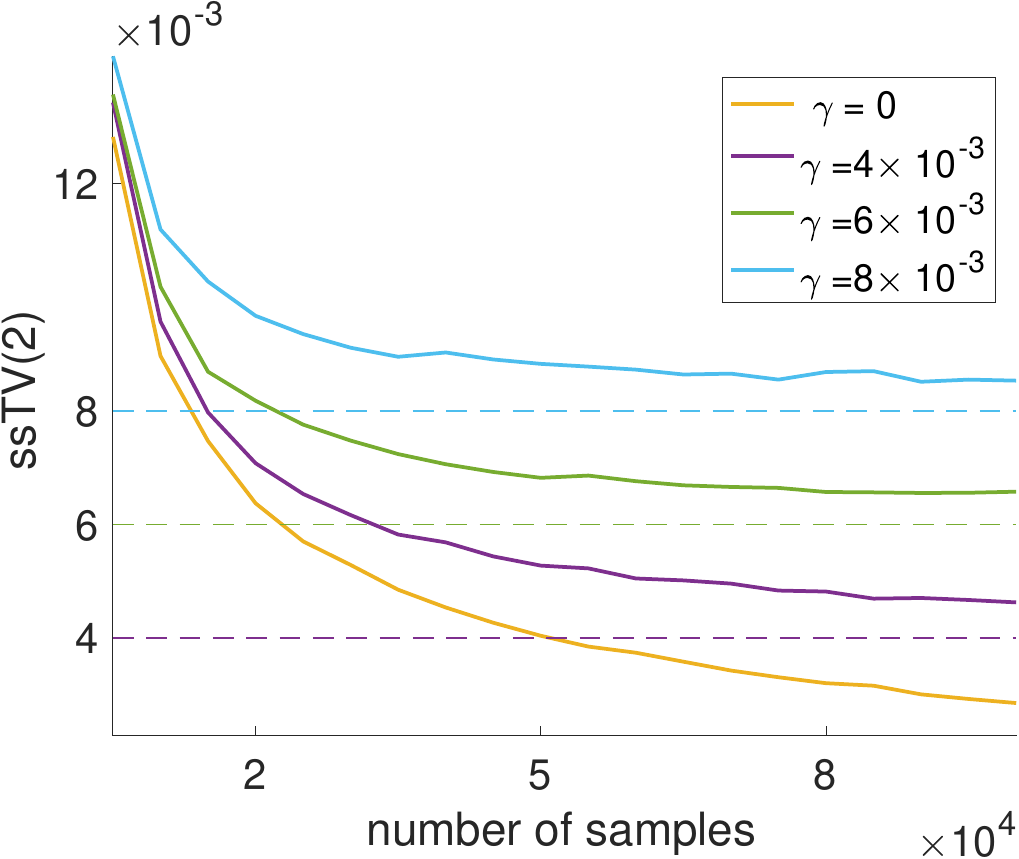}
\caption{The distribution $\Pt(x)$ is  factorized as Equation~\eqref{e:Ising} over a random tree with $p=6$ nodes with $\alpha=0.1$ and $\beta=0.8$. The generative distribution $P(x)$ is constructed by adding the random isotropic offset to the distribution $\Pt(x)$. The distribution $Q(x)$ is the output of Chow-Liu algorithm. The figure depicts $\LL^{(2)}(P,Q)$ as function of number of samples
for various values of $\gamma\triangleq \LL^{(2)}(P,\Pt)$. }\label{fig:misspec}   
\end{figure}

\section{Comparison of various tree-learning algorithms}\label{sec:comparisonLiterature}
Here, we review different approaches that could be taken toward learning a tree-structured distribution. 
We also review some known algorithms and their sample complexity.
The input of these algorithm is $n$ i.i.d. samples from a tree-structured distribution $P(x)\in\PP_{\T}(\alpha,\beta)$ with $\T=(\VV,\EE)$ being a tree. The gurantees mentioned below are with high probability.

\begin{itemize}
\item \textit{Guaranteed recovery of the structure}

\textbf{Algorithm:} Chow-Liu algorithm \cite{chow1968approximating}

\textbf{Output:}  $\Tc$

\textbf{Guarantee:} $\T=\Tc$

\textbf{Requirement:} $n > C\frac{e^{2\beta}}{\alpha^2}\log p$

In Theorems \ref{t:structUpper} and \ref{t:structLower} we provided the sample complexity of this problem (tight up to a constant factor).\\
\item \textit{Guaranteed recovery of the structure and accurate pairwise marginal}

\textbf{Algorithm:} Chow-Liu algorithm \cite{chow1968approximating}

\textbf{Output:}  $\Tc$ and $Q\in\PP_{\Tc}$

\textbf{Guarantee:} $\T=\Tc$ and $\LL^{(2)}(P,Q)\leq \eta\,.$

\textbf{Requirement:} $n > C \max\{\frac{e^{2\beta}}{\alpha^2},\frac{1}{\eta^2}\log \}\log p$\\
\item  \textit{Forest Approximation and accurate pairwise marginal}

\textbf{Algorithm:} Chow-Liu algorithm with a proper threshold over the weight of edges

\textbf{Output:}  $\Th=(\VV,\EE_{\Th})$ and $Q\in\PP_{\Th}$

\textbf{Guarantee:} $\EE_{\Th}\subseteq\EE_{\T}$ and $\LL^{(2)}(P,Q)\leq \eta\,.$

\textbf{Requirement:} $n> C \frac{e^{2\beta}}{\eta^{2}}\log p$

Proved in Proposition \ref{prop:TruncationUpper}
\\
\item \textit{Accurate pairwise marginals}

\textbf{Algorithm:} Chow-Liu Algorithm for the purpose of inference

\textbf{Output:}  $\Tc=(\VV,\EE)$ and $Q\in\PP_{\Tc}$

\textbf{Guarantee:}  $\LL^{(2)}(P,Q)\leq \eta\,.$

\textbf{Requirement:} $n> C \max\{e^{2\beta},\eta^{-2}\}\log p$

Proved in Theorem \ref{t:CLLowerBound}
\\

\item  \textit{Latent tree with accurate pairwise marginals}

\textbf{Algorithm:} Agarwala, Bafna, Farach, Paterson, and Thorup \cite{agarwala1998approximability}.

\textbf{Output:}  $S=(\VV\cup\widetilde{\VV},\EE_{S})$ and $Q\in\PP_{S}$ such that $\VV$ is the set of leaves of $S$.

\textbf{Guarantee:} $\LL^{(2)}(P,Q_{\VV})\leq \eta\,.$

\textbf{Requirement:} $n> C \frac{1}{\alpha^2 \eta^2}\log p $

Discussed in Section \ref{sec:RelatedWork}.
\\
\item  \textit{Latent tree with accurate joint distribution over the leaves}

\textbf{Algorithm:} Ambainis, Desper, Farach, and Kannan \cite{ambainis1997nearly}.

\textbf{Output:}  $S=(\VV\cup\widetilde{\VV},\EE_{S})$ and $Q\in\PP_{S}$ such that $\VV$ is the set of leaves of $S$.

\textbf{Guarantee:} $\LL^{(p)}(P,Q_{\VV})\leq \eta\,.$

\textbf{Requirement:} $n> C \max\{\frac 1{\alpha^2},p \frac{e^{2\beta}}{\eta^2} \}$

Discussed in Section \ref{sec:RelatedWork}.
\\

\end{itemize}

The last two introduced problems are examples of improper learning in which the class of models from which we learn is bigger than the class of models assumed over the original distribution. In this case, we assume that the original distribution is tree-structured over the nodes in $\VV$ whereas the learned model is represented with a latent tree such that its leaves are the nodes in $\VV$. Specific applications determine the requirements over the learned models and the guarantees which are necessary for the subsequent purposes. But in general proper learning is necessary when the identification of the qualitative structure behind the distributions is necessary. The extreme case of which is when one is only interested in the underlying structure  which corresponds to the first problem in above category.

Note that every tree-structured distribution over $p$ nodes in $\VV$ can be equivalently represented using a latent tree with $\VV$ as its leaves. To do so, one has to add some dummy variables and infinite strength edges to the original tree. But the converse is not true: A distribution represented on the leaves of a latent tree cannot be factorized as a tree-structured distribution in general.

\end{document}